\documentclass{article}

\usepackage[a4paper,left=20mm,top=5mm,right=20mm,bottom=5mm,includeheadfoot,headheight=20mm, headsep=5mm, footskip=15mm]{geometry}
\usepackage[utf8]{inputenc}
\usepackage[all]{xy}
\usepackage{amsthm,amsmath,amssymb,amsfonts,amscd}
\usepackage[backend=bibtex]{biblatex}
\usepackage{bbm}
\usepackage{booktabs}
\usepackage{stackrel}
\usepackage{csquotes}
\usepackage[english]{babel}
\usepackage{enumerate}
\usepackage{fancyhdr}
\usepackage[final]{pdfpages}
\usepackage{graphicx}
\usepackage{imakeidx}
\usepackage{inconsolata}
\usepackage{latexsym,keyval,ifthen}
\usepackage[makeroom]{cancel}
\usepackage{mathalfa}
\usepackage{mathdots} 
\usepackage{multicol}
\usepackage{multicol}
\usepackage{nameref}
\usepackage{newpxmath}
\usepackage[OT2,T1]{fontenc}
\usepackage{prerex}
\usepackage{pst-node}
\usepackage{randomwalk}
\usepackage{soul}
\usepackage{tgpagella}
\usepackage{theoremref}
\usepackage{tikz}
\usepackage{tikz-cd}
\usepackage{wrapfig}
\usepackage{xgalley}
\usepackage{yfonts}
\usepackage{calc}

\DeclareMathAlphabet{\mathboondoxfrak}{U}{BOONDOX-frak}{m}{n}

\DeclareMathOperator{\Cl}{cl}

\DeclareMathOperator{\Gal}{Gal}
\DeclareMathOperator{\GL}{GL}

\DeclareMathOperator{\Hom}{\textnormal{\textbf{Hom}}}

\DeclareMathOperator{\rank}{rank}

\DeclareMathOperator{\PGL}{PGL}

\DeclareSymbolFont{cyrletters}{OT2}{wncyr}{m}{n}
\def\Yint#1{\mathchoice
    {\YYint\displaystyle\textstyle{#1}}
    {\YYint\textstyle\scriptstyle{#1}}
    {\YYint\scriptstyle\scriptscriptstyle{#1}}
    {\YYint\scriptscriptstyle\scriptscriptstyle{#1}}
      \!\int}
\def\YYint#1#2#3{{\setbox0=\hbox{$#1{#2#3}{\int}$}
    \vcenter{\hbox{$#2#3$}}\kern-.52\wd0}}

\def\mint{\Yint\times}

\newcommand{\A}{\mathbb{A}}

\newcommand{\bbm}{\begin{bm}}

\newcommand{\bes}{\begin{center}\begin{tikzcd}[ampersand replacement=\&]}

\newcommand{\bpf}{\begin{proof}}
\newcommand{\bsm}{\big(\begin{smallmatrix}}

\newcommand{\C}{\mathbb{C}}

\newcommand{\ebm}{\end{bm}}
\newcommand{\ees}{ \end{tikzcd} \end{center}}
\newcommand{\epf}{\end{proof}}
\newcommand{\esm}{\end{smallmatrix}\big)}

\newcommand{\foodnote}[1]{\ignorespaces}

\newcommand{\icla}[1]{\left\lbrace#1\right\rbrace}

\newcommand{\indi}{\mathbbm{1}}
\newcommand{\Ind}{\textnormal{Ind}}
\newcommand{\inv}{^{-1}}
\newcommand{\ipa}[1]{\left(#1\right)}
\newcommand{\iy}{\infty}

\newcommand{\mbT}{\mathbb{T}}
\newcommand{\mcA}{\mathcal{A}}

\newcommand{\mcC}{\mathcal{C}}

\newcommand{\mcE}{\mathcal{E}}
\newcommand{\mcF}{\mathcal{F}}
\newcommand{\mcG}{\mathcal{G}}

\newcommand{\mcK}{\mathcal{K}}

\newcommand{\mcO}{\mathcal{O}}

\newcommand{\mcU}{\mathcal{U}}

\newcommand{\mfgltwo}{\mathfrak{gl}_2}
\newcommand{\mfg}{\mathfrak{g}}

\newcommand{\mfH}{\mathfrak{H}}

\newcommand{\mfo}{\mathfrak{o}}
\newcommand{\mfp}{\mathfrak{p}}

\newcommand{\mfq}{\mathfrak{q}}

\newcommand{\ord}{\textnormal{ord}}
\newcommand{\oti}{\otimes}

\newcommand{\ovl}{\overline}

\newcommand{\PP}{\mathbb{P}}

\newcommand{\Q}{\mathbb{Q}}

\newcommand{\ra}{\rightarrow}
\newcommand{\R}{\mathbb{R}}

\newcommand{\sicc}{\Sigma_{\C}^\C(K/F)}
\newcommand{\sirc}{\Sigma_{\R}^\C(K/F)}
\newcommand{\sirr}{\Sigma_{\R}^\R(K/F)}
\newcommand{\siun}{\Sigma_{\tno{un}}(K/F)}

\newcommand{\tbf}{\textbf}

\newcommand{\ti}{\times}

\newcommand{\tno}{\textnormal}

\newcommand{\udl}{\underline}
\newcommand{\uhp}{\mathfrak{H}}

\newcommand{\Z}{\mathbb{Z}}

\newtheorem{theorem}{Theorem}[section]
\newtheorem{lemma}[theorem]{Lemma}
\newtheorem{corollary}[theorem]{Corollary}
\newtheorem{remark}[theorem]{Remark}
\newtheorem{hypothesis}[theorem]{Hypothesis}
\newtheorem{conjecture}[theorem]{Conjecture}

\newtheorem{proposition}[theorem]{Proposition}

\newenvironment{bm}{ \begin{pmatrix} }{ \end{pmatrix} } 

\bibliography{biblio}
\begin{document}
\title{Exceptional zero formulas for anticyclotomic $p$-adic $L$-functions}
\author{V\'ictor Hern\'andez Barrios and Santiago Molina Blanco}
\maketitle
\begin{abstract}
In this note we define anticyclotomic $p$-adic measures attached to a finite set of places $S$ above $p$, a modular elliptic curve $E$ over a general number field $F$ and a quadratic extension $K/F$.
We study the exceptional zero phenomenon that arises when $E$ has multiplicative reduction at some place in $S$.
In this direction, we obtain $p$-adic Gross-Zagier formulas relating derivatives of the corresponding $p$-adic L-functions to the extended Mordell-Weil group of $E$. 
Our main result 
uses the recent construction of plectic points on elliptic curves due to Fornea and Gehrmann and generalizes their main result in \cite{fornea2021plectic}. We obtain a formula that computes the $r$-th derivative of the $p$-adic L-function, where $r$ is the number of places in $S$ where $E$ has multiplicative reduction, in terms of plectic points and Tate periods of $E$.
\end{abstract}

\begin{small}
\tableofcontents
\end{small}

\section{Introduction}
The Birch and Swinnerton-Dyer conjecture relates the classical $L$-function $L(E,K,s)$ of an elliptic curve $E$ defined over a field $K$ with its group of rational points $E(K)$.
It can be stated by means of an equality which involves several invariants of the elliptic curve.
In particular, it implies that the algebraic rank, the $\Z$-rank of $E(K)$, is equal to the analytic rank, the order of vanishing of $L(E,K,s)$ at $s=1$.
Some of the progress in proving partial versions of this conjecture use a set of special points, called \emph{Heegner} points, that emerge when $K$ is a totally imaginary quadratic extension of a totally real field $F$ and $E$ is modular and defined over $F$. Indeed, modularity provides a modular parametrization of $E$, namely, a morphism defined over $F$ between a Shimura curve $X$ and the elliptic curve $E$.
Therefore, Heegner points can be constructed as images through modular parametrizations of the CM points of the Shimura curve $X$.

We can explain the complex construction of Heegner points with a little more detail: The elliptic curve $E$ is associated with an automorphic form $f_E$ of weight 2 for the group of units of a quaternion algebra over $F$ that splits at a single archimedean place. 
The group of complex points $E(\C)$ can be described as the quotient $\C/\Lambda$, where $\Lambda$ is the lattice of periods of $f_E$.
Similarly, the set of complex points $X(\C)$ can be described as a finite union of quotients $\mfH/\Gamma$, where $\mfH$ is the Poincar\'e upper-half plane, and $\Gamma$ is an arithmetic subgroup associated to the quaternion algebra. Finally, the modular parametrization is an extension of the natural morphism
\begin{equation}\label{complexunifor}
\tno{Div}^0(\mfH/\Gamma)\longrightarrow\C/\Lambda,\qquad (\tau_1-\tau_2)\longmapsto\int_{\tau_2}^{\tau_1}f_E(z)dz,
\end{equation}
where $\tno{Div}^0$ denotes the set of degree zero divisors. A Heegner point is the image through the modular parametrization of a point $\tau_K\in\mfH$ whose stabilizer in $\Gamma$ is isomorphic to a subgroup in $K^\ti$.

This classical construction can be emulated $p$-adically. Indeed, if $E$ has split multiplicative reduction modulo $\mfp\mid p$ and $K$ does not split at $\mfp$, the group of points $E(K_\mfp)$ can be described as the quotient $K_\mfp^\ti/q_{E,\mfp}^\Z$, where $q_{E,\mfp}\in F_\mfp^\ti$ is \emph{Tate's period}. Similarly, the set $X(K_\mfp)$ admits a $p$-adic 
uniformization as union of quotients $\mfH_\mfp/\Gamma_B$, where $\mfH_\mfp=K_\mfp\setminus F_\mfp$ is the $p$-adic upper half plane and $\Gamma_B$ is a $\mfp$-arithmetic group in a definite quaternion algebra $B$ over $F$. 
By exchanging the usual integral by a multiplicative integral $\mint$, we can describe the modular parametrization as an extension of the natural morphism
\begin{equation}\label{p-adicunifor}
\tno{Div}^0(\mfH_\mfp/\Gamma_B)\longrightarrow K_\mfp^\ti/q_{E,\mfp}^\Z,\qquad (\tau_1-\tau_2)\longmapsto\mint_{\tau_2}^{\tau_1}\phi_E(z)dz,
\end{equation}
where $\phi_E$ is the \emph{harmonic cocycle} attached to $E$. Again the Heegner point is the image through this modular parametrization of a point $\tau_\mfp\in\mfH_\mfp$ whose stabilizer in $\Gamma_B$ is isomorphic to a subgroup of $K^\ti$.

In the seminal work \cite{Darmon2001}, Darmon introduced a new kind of points that appear when $F=\Q$ and $K$ is a real quadratic field. In analogy with their imaginary counterparts Heegner points, these \emph{Darmon points} are conjectured to be defined over abelian extensions of $K$ and their heights are supposed to be related with the derivatives of twists of $L(E,K,s)$ at $s=1$.
The idea of Darmon's construction relies on the treatment of the above $p$-adic construction of Heegner points in purely group theoretical terms. Indeed, while the morphism \eqref{p-adicunifor} can be seen as an element of $H^0(\Gamma_B,\Hom(\tno{Div}^0(\mfH_\mfp),E(K_\mfp)))$, Darmon was able to construct an element in $H^1(\Gamma_{\tno{M}_2(\Q)},\Hom(\tno{Div}^0(\mfH_p),E(K_p)))$ using the local nature of the multiplicative integral. Such a cohomology class can be extended to the full group of divisors $\tno{Div}(\mfH_p)$,
thus, the image of a previously characterized $\tau_\mfp\in\mfH_\mfp$ gives rise to a class 
$\mint_{\tau_\mfp}\phi_E\in H^1(\mcO_+^p,E(K_p))$, where $\mcO_+^p$ 
is the stabilizer of $\tau_\mfp$ in $\Gamma_{\tno{M}_2(\Q)}$.
On the other hand, the nature of the group of $p$-units of $K$ provides a natural fundamental class in $\eta^p\in H_1(\mcO_+^p,\Z)$. Hence, the Darmon point is given by the cap product $P_K=\mint_{\tau_\mfp}\phi_E\cap\eta^p\in E(K_p)$.

Darmon's construction was generalized to many other situations (see \cite{Dasgupta}, \cite{Greenberg}, \cite{Gartner} and \cite{GMS}), each of which was relaxing requirements on $F$, $K$ and $E$. Finally, the second author contributed to a joint work with X. Guitart and M. Masdeu \cite{guitart2017automorphic} where a completely general construction of Darmon points was given. This implies that Darmon points are available for arbitrary base fields $F$, arbitrary quadratic extensions $K/F$, and arbitrary modular elliptic curves over $F$ admitting $p$-adic Tate uniformizations (in fact, arbitrary modular abelian varieties and both $p$-adic and archimedean uniformizations are considered in \cite{guitart2017automorphic}).  

Very recently in \cite{fornea2021plectic}, M. Fornea and L. Gehrmann adapt the construction of Darmon points of \cite{guitart2017automorphic} to compute elements in certain completed tensor products of groups $E(K_\mfp)$, for primes $\mfp$ of (split and non-split) multiplicative reduction. These elements, called \emph{plectic points}, are conjectured to come from determinants of points defined over abelian extensions of $K$. Moreover, they are also conjectured to be non-zero in rank $r\leq[F:\Q]$ situations. These conjectures, that can be found in \S \ref{section:conj}, are backed by some numerical evidence. Even with the restriction $r\leq[F:\Q]$ and the fact that plectic points are supposed to be zero when $E$ is the extension of an elliptic curve defined over a smaller field, this new construction could shed a lot of light on our understanding of the Birch and Swinnerton-Dyer conjecture in rank $r\geq 2$ situations. 

Parallel to the research on new constructions of points in the Mordell-Weil group $E(K)$, the theory of $p$-adic L-functions has been developing. $p$-adic L-functions are $p$-adic analogues of the classical $L$-functions $L(E,K,s)$. They can be thought of as $p$-adic analytic functions $L_p(E,K,s):\C_p\rightarrow\C_p$ that are usually defined by means of a $p$-adic measure $\mu_{E,K}$ of a Galois group $\mcG_p$ isomorphic to $\Z_p$:
\[
L_p(E,K,s)=\int_{\mcG_p}\exp_p(s\ell(\gamma))d\mu_{E,K}(\gamma),\qquad \ell:\mcG_p\stackrel{\simeq}{\longrightarrow}\Z_p.
\]
The link between classical and $p$-adic L-functions is the so-called \emph{interpolation formula}. It relates the image through $\mu_{E,K}$ of certain characters $\chi:\mcG_p\ra\C_p$ to critical values of twists of the classical L-function $L(E,\bar \chi,s)$, where $\bar\chi$ is the automorphic character associated with $\chi$ via Class Field Theory. Usually, the space of characters where the interpolation property applies is dense, hence these critical values characterize $\mu_{E,K}$, and consequently $L_p(E,K,s)$.

The study of $p$-adic L-functions took a new dimension after the work of Mazur, Tate and Teitelbaum \cite{MTT}, who in 1986 formulated a $p$-adic Birch and Swinnerton-Dyer conjecture  over $\Q$ relating a (cyclotomic) $p$-adic L-function and the \emph{extended Mordell-Weil group} $\bar E(\Q)$. Such extended Mordell-Weil group $\bar E(K)$, for any number field $K$, is the concatenation of the classical Mordell-Weil group $E(K)$ and the group generated by Tate periods $q_{E,\mfp}$ at primes $\mfp\mid p$ where $E$ has split multiplicative reduction. Indeed, the presence of these split multiplicative primes gives rise to
\emph{exceptional zero} phenomena, by which the $p$-adic $L$-function gets additional zeroes. 

The ideas of \cite{MTT} have been generalized with success in other contexts.
For instance, Bertolini and Darmon defined in \cite{Bertolini1998} a new (anticyclotomic) $p$-adic L-function associated with an elliptic curve $E/\Q$ and an imaginary quadratic field $K$. The $p$-adic description of the geometry of the Shimura curve $X$ over $K_p$ allowed them to attack a new $p$-adic Birch and Swinnerton-Dyer ($p$-adic BSD) in this anticyclotomic setting. One of their main results \cite[Theorem B]{Bertolini1998} relates a Heegner point with the derivative of the $p$-adic $L$-function, obtaining a $p$-adic analogue to the celebrated Gross-Zagier formula.

In order to deal with the case where $K/\Q$ is real quadratic one has to slightly change the point of view. In this situation, the corresponding anticyclotomic $p$-adic Galois group $\mcG_{K,p}$ is finite, and therefore no analytic $L_p(E,K,s)$ can be defined. But following the philosophy of Mazur-Tate in \cite{MT}, we can rewrite the $p$-adic BSD working directly with the measure $\mu_{E,K}$ instead of the analytic function $L_p(E,K,s)$. In fact, the space of measures $\tno{Meas}(\mcG_{K,p},\C_p)$ has a natural ring structure, called \emph{Iwasawa algebra}. The order of vanishing of $L_p(E,K,s)$ at $s=0$ is precisely the maximum $r$ for which $\mu_{K,E}\in I_1^r$, where for any character $\chi$ the augmentation ideal $I_\chi$ is the kernel of the ring homomorphism $\varphi_\chi$:
\[
I_\chi=\ker\ipa{\tno{Meas}(\mcG_{K,p},\C_p)\stackrel{\varphi_\chi}{\longrightarrow}\C_p},\qquad\varphi_\chi(\mu)=\int_{\mcG_{K,p}}\chi d\mu.
\]
Moreover, finding the $r$th derivative of $L_p(E,K,s)$ at $s=0$ amounts to computing the image of $\mu_{K,E}$ in $I_1^r/I_1^{r+1}$. With this new formalism in mind, Bertolini and Darmon were able to obtain in \cite{Bertolini2007} an analogous $p$-adic Gross-Zagier formula relating the class of $\mu_{K,E}$ in $I_1/I_1^{2}$ to a Darmon point when $K/\Q$ was real quadratic. Indeed, the Artin map of Class Field Theory provides a morphism 
\begin{equation}\label{defrecintro}
    \tno{rec}_p:K_p^\times/\Q_p^\times\longrightarrow\mcG_{K,p}.
\end{equation}
Moreover, we have a the natural group homomorphism
\begin{equation}\label{varphiGI2}
\begin{tikzcd}
\varphi: &[-3em]\mcG_{K,p} \ar[r] & I_\chi/I_\chi^2; & \int_{\mcG_{K,p}}f d\varphi(\sigma) =\chi(\sigma)\inv\cdot f(\sigma)-f(1), & \sigma\in \mcG_{K,p}.
\end{tikzcd}
\end{equation}
Hence, if we write $\sigma_p$ for the nontrivial automorphism of $K_p/\Q_p$, then by means of the composition $\varphi\circ\tno{rec}_p$ we can realize the difference $(\sigma_p-1)P$ as an element of $I_\chi/I_\chi^2$, for any $P\in E(K_p)\simeq K_p^\times/q_E^\Z$.

Going back to the general setting of arbitrary $F$ and arbitrary quadratic extension $K/F$, Fornea and Gehrmann 
are able to construct in the aforementioned paper \cite{fornea2021plectic} an analogous anticyclotomic measure $\mu_{K,S}$ associated with $K$ and a set $S=\{\mfp_1,\cdots,\mfp_r\}$ of primes above $p$ of multiplicative reduction, under the assumption that $K$ is inert at any $\mfp_i\in S$. 
For any character $\chi$ with certain prescribed behavior at all $\mfp_i\in S$, they show that $\mu_{K,S}\in I_\chi^r$.
Moreover, they obtain in \cite[Theorem 5.13]{fornea2021plectic} a $p$-adic Gross-Zagier formula that relates the corresponding (twisted) plectic $P_\chi^S$ point with the class of $\mu_{K,S}$ in $I_\chi^r/I_\chi^{r+1}$.

In this paper, we have been able to generalize Fornea-Gehrmann's 
$\mu_{K,S}$ by constructing a $p$-adic distribution associated to any set of places $S$ above $p$ of non-supercuspidal reduction. In the ordinary setting, such distribution extends to a $p$-adic measure. If the primes in $S$ are of good or multiplicative reduction, $\mu_{K,S}$ has good interpolation properties as shown in Theorem \ref{THMintprop}. Moreover, we have been able to prove a $p$-adic Gross-Zagier formula that takes into account all the extended Mordell-Weil group, thus including all Tate's periods coming from exceptional zeroes. In that sense, our main result theorem \ref{mainTHM2} is both a generalization of \cite[theorem 5.13]{fornea2021plectic} and the exceptional zero formulas of \cite{Bergunde_2018}. In this introduction we can state a simplified version of theorem \ref{mainTHM2} where primes in $S$ have either good or split multiplicative reduction:

\vspace{+1em}
{\bf Theorem}
\emph{
Write $S=S_+\cup S_-$, and $S_+=S_+^1\cup S_+^2$, where $S_+$ is the set of primes of split multiplicative reduction, 
\[
S_+^1:=\{\mfp\in S_+,\;T\mbox{ splits in }\mfp\},\qquad S_+^2=\{\mfq\in S_+,\;T\mbox{ does not split in }\mfq\}.
\]
Then for any character $\chi$ that is trivial at all $\mfp_i\in S$,
we have that $\mu_{K,S}\in I_\chi^r$, where $r=\#S_+$, and
the image of $\mu_{K,S}$ in $I_\chi^r/I_\chi^{r+1}$ is given by
\[\mu_{K,S}\equiv c\cdot \epsilon_{S_-}(\pi_{S_-},\chi_{S_-})\cdot\varphi\circ\tno{rec}_{S_+}\ipa{q_{S_+^1}\otimes \ipa{\prod_{\mfp\in S_+^2}(\sigma_\mfp-1)}P_\chi^{S_+^2}}\mod I_\chi^{r+1},\]
where  $c\in \Q^\ti$ is an explicit constant, $\epsilon_{S_-}(\pi_{S_-},\chi_{S_-})$ is the epsilon factor at $S_-$ appearing in the interpolation formula, the morphisms $\varphi$ and $\tno{rec}_{S_+}$ are analogous to that of \eqref{defrecintro} and \eqref{varphiGI2}, $q_{S_+^1}=\bigotimes_{\mfp\in S_+^1}q_{E_\mfp}\in  T(F_{S_+^1})$ is the product of Tate periods, $\sigma_\mfp$ is the non-trivial automorphism of $\Gal(K_\mfp/F_\mfp)$ and $P_\chi^{S_+^2}$ is a (twisted) plectic point attached to the set $S_+^2$.
}

\subsection*{Acknowledgements.}
This project has received funding from the European Research Council
(ERC) under the European Union's Horizon 2020 research and innovation
programme (grant agreement No. 682152). Moreover, the second author has been partially funded by the project PID2021-124613OB-I00 from Ministerio de Ciencia e innovaci\'on. We would also like to thank Lennart Gehrmann and Marc Masdeu for all their helpful remarks and comments.

\section{Setup and notation}\label{section:setup}
Let $F$ be a number field
and let $r_F,s_F$ be the number of real and complex places, respectively, so that $[F:\Q]=r_F+2s_F$.
Let $K/F$ be a quadratic extension and let $\Sigma_{\tno{un}}(K/F)$ be the set of infinite places of $F$ splitting when extended to $K$.
Then
\[\siun=\sirr\sqcup\sicc\]
where $\sirr$ is the set of real places of $F$ which remain real in $K$ and $\sicc$ the set of complex places of $F$.
Let $u$,$r_{K/F,\R}$ and $s_F$ be the cardinality of these three sets, hence
$u=r_{K/F,\R}+s_F$. Analogously, we write $\Sigma_\R^\C(K/F)$ for the set of places $\infty\setminus\siun$ and $r_{K/F,\C}$ for its cardinality.

For any finite set of places $S$ of $F$ we write $F_S=\prod_{v\in S} F_v$. 
We denote by $\A_F$ the ring of adeles of $F$, and by $\A_F^S$ the ring of adeles outside $S$, namely, $\A_F^S=\A_F\cap\prod_{v\not\in S}F_v$. We will often write $\iy$ for the set of infinite places of $F$. For a finite place $\mfp$, we will write $v_\mfp:F_\mfp^\times\ra\Z$ for the $p$-adic valuation, $\mfo_{F_\mfp}$ for the integer ring, $\varpi_\mfp\in \mfo_{F_\mfp}$ for a fixed uniformizer, and $q_\mfp$ for the cardinal of the residue field $k_\mfp:=\mfo_{F_\mfp}/\mfp$.

Let $B/F$ be a quaternion algebra for which there exists an embedding $K\hookrightarrow B$, that we fix now. Let $\Sigma_B$ be the set of archimedean classes of $F$ splitting the quaternion algebra $B$.
We can define $G$ an algebraic group associated to $B^\ti/F^\ti$ as follows:
$G$ represents the functor that sends any $\mfo_F$-algebra $R$ to
\[G(R)=(\mcO_B\oti_{\mfo_F} R)^\ti/R^\ti,\]
where $\mcO_B$ is a maximal order in $B$ that we fix once and for all.
Similarly, we define the algebraic group $T$ associated to $K^\ti/F^\ti$ by
\[T(R)=(\mfo_{c_0}\oti_{\mfo_F} R)^\ti/R^\ti,\]
where $\mfo_{c_0}:=\mfo_B\cap K$ is an order of conductor $c_0$ in $\mfo_K$.
Note that $T\subset G$.
We denote by $G(F_\iy)_+$ and $T(F_\iy)_+$ the connected component of the identity of $G(F_\iy)$ and $T(F_\iy)$, respectively.
We also define $T(F)_+:=T(F)\cap T(F_\iy)_+$ and $G(F)_+:=G(F)\cap G(F_\iy)_+$.


Let $M$ be a $G(\A_F^S)$-representation over a field $L$, and let $\rho$ be an irreducible $G(\A_F^{S'})$-representation over $L$ with $S'\subseteq S$. We will write 
\[
M_\rho:=\Hom_{G(\A_F^S)}\left(\rho\mid_{G(\A_F^S)},M\right).
\]
as representations over $L$.

For any number field $F$ and any place $v$, we choose the Haar measure $dx_v^\times$ for $F_v^\times$:
\[
d^\times x_v=\zeta_v(1)|x_v|_v^{-1}dx_v;\qquad\mbox{where}\quad \left\{\begin{array}{ll}
 dx_v\mbox{ is $[F_v:\R]$ times the usual Lebesgue measure,}    &v\mid\infty;  \\
   dx_v\mbox{ is the Haar measure of $F_v$ such that }{\rm vol}(\mfo_{F,v})=|d_{F_v}|_v^{1/2},  &v\nmid\infty, 
\end{array}\right.
\]
$d_{F_v}$ is the different of $F_v$, $\zeta_v(s)=(1-q_v^{-s})^{-1}$, if $v\nmid\infty$, $\zeta_v(s)=\pi^{-s/2}\Gamma(s/2)$, if $F_v=\R$, and $\zeta_v(s)=2(2\pi)^{-s}\Gamma(s)$, if $F_v=\C$. The product of such measures provides a Tamagawa measure $d^\times x$ on $\A_F^\times/F^\times$. If we choose $d^\times t$ to be the quotient measure for $T(\A_F)/T(F)=\A_K^\times/\A_F^\times K^\times$, one has that ${\rm vol}(T(\A_F)/T(F))=2L(1,\psi_K)$, where $\psi_K$ is the quadratic character associated with $K/F$.

\section{Fundamental classes of tori}\label{section:fundcla}

In this section we will define certain fundamental classes associated with the torus $T$ attached to $K^\ti/F^\ti$.

\subsection{The fundamental class $\eta$ of the torus}\label{section.fundamentalclass}

Let $U=\prod_{\mfq}T(\mfo_{F,\mfq})$ and denote by $\mcO:=T(F)\cap U$ the group of relative units. Similarly, we 
define $\mcO_+:=\mcO\cap T(F_\iy)_+$ to be the group of totally positive relative units. By an straightforward argument using Dirichlet Units Theorem
\[\tno{rank}_\Z\mcO=\tno{rank}_\Z \mcO_+=(2r_{K/F,\R}+r_{K/F,\C}+2s_F-1)-(r_F+s_F-1)=r_{K/F,\R}+s_F=u.\] We also define the class group
\[
\Cl(T)_+:=T(\A_{F})/\ipa{T(F)\cdot U\cdot T(F_\iy)_+}\simeq T(\A_{F}^\iy)/\ipa{T(F)_+\cdot U}.
\]

Note that the subgroup
\[T(F_{\siun})_0=\R_+^u=(\R_+)^{\#\sirr}\ti(\R_+)^{\#\sicc}\subset T(F_{\siun})=(\R^\ti)^{\#\sirr}\ti(\C^\ti)^{\#\sicc}\]
 is isomorphic to $\R^u$ by means of the homomorphism $T(F_\iy)_+\ra \R^u$ given by $z\mapsto (\log|\sigma z|)_{\sigma\in\siun}$.
Moreover, under this isomorphism the image of $\mcO_+$ is a $\Z$-lattice $\Lambda\subset \R^u$, as in the proof of Dirichlet's Unit Theorem.

We can identify $\Lambda$ with its preimage in $T(F_{\siun})_0$.
Then $T(F_{\siun})_0/\Lambda$ is a $u$-dimensional real torus.
The \tbf{fundamental class} $\xi$ is a generator of $H_u(T(F_{\siun})_0/\Lambda,\Z)\simeq \Z$.
We can give a better description of $\xi$: Let $M:=T(F_{\siun})_0\simeq \R^u$.
The de Rham complex $\Omega_M^\bullet$ is a resolution for $\R$.
This implies that we have an edge morphism of the corresponding spectral sequence 
\[e: H^0(\Lambda,\Omega_M^u)\ra H^u(\Lambda,\R).\]
We identify any $c\in H_u(M/\Lambda,\Z)$ with $c\in H_u(\Lambda,\Z)$ by means of the relation
\[\int_c \omega=e(\omega)\cap c,\tno{ where }\omega\in H^0(\Lambda,\Omega_M^u)=\Omega_{M/\Lambda}^u.\]
We can think of $\xi\in H_u(\Lambda,\Z)$ as such that:
\begin{equation}\label{eq.wcto}
e(\omega)\cap \xi=\int_{T(F_\iy)/\Lambda}\omega,\tno{ for all }\omega\in H^0(\Lambda,\Omega_M^u).
\end{equation}

Note that
\begin{equation}\label{definition.mbT}
T(F_\iy)_+=\mbT\times T(F_{\siun})_0,\qquad \mbT:=(S^1)^{\#\sirc+\#\sicc}.
\end{equation}
Let $\mcO_{+,\mbT}=\mcO_+\cap \mbT$.
Since $\mcO_+$ is discrete and $\mbT$ is compact, $\mcO_{+,\mbT}$ is finite.
\begin{lemma}\label{lemma:quotientproduct}
We have that $\mcO_+\simeq \Lambda\ti \mcO_{+,\mbT}$. In particular,
\begin{equation}\label{equation:quotientproduct}
T(F_\iy)_+/\mcO_+\simeq T(F_{\siun})_0/\Lambda\ti \mbT/\mcO_{+,\mbT}
\end{equation}
\end{lemma}
\bpf
The image of the morphism $\pi|_{\mcO_+}: \mcO_+ \hookrightarrow T(F_\iy)_+ \stackrel{\pi}{\ra}T(F_{\siun})_0$ is $\Lambda$ by definition.
Since $\mbT=\ker\pi$, $\ker\pi|_{\mcO_+}=\mbT\cap \mcO_+=\mcO_{+,\mbT}$ and we deduce the following exact sequence
\begin{center}
\begin{tikzcd}
0 \ar[r] & \mcO_{+,\mbT} \ar[r] & \mcO_+ \ar[r,"\pi|_{\mcO_+}"] & \Lambda \ar[r] & 0.\\[-2em]
\end{tikzcd}
\end{center}
Since $\Lambda$ is a free $\Z$-module, such sequence splits and the result follows. 
\epf

Notice that the group $\Cl(T)_+$ fits in the following exact sequence
\begin{center}
\begin{tikzcd}
0 \ar[r] & T(F)_+/\mcO_{+} \ra T(\A_F^\iy)/U \ar[r,"p"]  & \Cl(T)_+ \ar[r] & 0. \\[-2em]
\end{tikzcd}
\end{center}
We fix preimages $\ovl t_i\in T(\A_F^\iy)$ for every element $t_i\in \Cl(T)_+$ and we consider the compact set
$\mcF=\bigcup_i \ovl{t_i}U\subset T(\A_F^\iy)$.
It is compact indeed because $\Cl(T)_+$ is finite and $U$ compact. 
\begin{lemma}\label{lemma1p2}
For any $t\in T(\A_F)$ there exists a unique $\tau_t\in T(F)/\mcO_+$ such that $\tau_t\inv t\in T(F_\iy)_+\ti \mcF$.
\end{lemma}
\bpf
Since $T(F_\iy)/T(F_\iy)_+=T(F)/T(F)_+$, given $t=(t_\iy,t^\iy)\in T(\A_F)$ there exists $\gamma\in T(F)$ such that $\gamma t_\iy\in T(F_\iy)_+$.
On the other hand, $p(\gamma t^\iy U)=t_i$ for some $i$. 
By the definition of $\Cl(T)_+$, there exists $\tau\in T(F)_+$ such that $\tau\gamma t^\iy U=\ovl{t_i}U$.
Hence
\[\tau\gamma t=(\tau\gamma t_\iy,\tau\gamma t^\iy)\in T(F_\iy)_+\ti \ovl{t_i}U\subset T(F_\iy)_+\ti \mcF.\]
By considering the image of $\tau\inv \gamma\inv$ in $T(F)/\mcO_+$, we deduce the existence of $\tau_t$.

For the unicity, suppose there exist $\tau,\tau'\in T(F)$ with
\[(\tau t_\iy,\tau t^\iy)=\tau t\in T(F_\iy)_+\ti \mcF\ni \tau't=(\tau't_\iy,\tau't^\iy)\]
Then, on the one hand, $\tau\inv\tau'=(\tau t_\iy)\inv(\tau't_\iy)\in T(F_\iy)_+$.
On the other hand,
$\tau t^\iy =\ovl t_i u$ for some $\ovl t_i$ and $u\in U$,
and $\tau' t^\iy =\ovl t_j u'$ for some $\ovl t_j$ and $u'\in U$. 
So $\tau\inv \tau'=(\ovl t_i)\inv \ovl t_j u\inv u'$.
This implies $t_i=p(\ovl t_i U)=p(\ovl t_jU)=t_j$.
By the construction of $\mcF$, we must have $\ovl t_i=\ovl t_j$.
Thus, $\tau\inv\tau'\in U\cap T(F)\cap T(F_\iy)_+=\mcO_+$.
\epf
The set of continuous functions $C(\mcF,\Z)$ has a natural action of $\mcO_+$ given by translation.
The characteristic function $\indi_\mcF$ is $\mcO_+$-invariant.
Consider the cap product
\begin{equation}\label{def:fundamentalclass}
\eta=\indi_\mcF\cap \xi\in H_u(\mcO_+,C(\mcF,\Z)),
\end{equation}
where $\indi_\mcF\in H^0(\mcO_+,C(\mcF,\Z))$ and $\xi\in H_u(\mcO_+,\Z)$ is the image of $\xi$ through the corestriction morphism.

For any ring $R$ and any $R$-module $M$, write $C^0_c(T(\A_F),M)$ is the set of locally constant $M$-valued functions of $T(\A_F)$ that are compactly supported when restricted to $T(\A_F^\iy)$. Notice that if $M$ is endowed with a natural $T(F)$-action then we have a natural action of $T(F)$ on $C^0_c(T(\A_F),M)$.
\begin{lemma}\label{lemma.thereisaniso}
There is an isomorphism of $T(F)$-modules
\[\Ind_{\mcO_+}^{T(F)}(C(\mcF,\Z))\simeq C^0_c(T(\A_F),\Z).\]
\end{lemma}
\bpf
Since $\mcF\subset T(\A_F)$ is compact there is an $\mcO_+$-equivariant embedding
\begin{center}
\begin{tikzcd}
\iota: C(\mcF,\Z) \ar[r,hook] & C^0_c(T(\A_F),\Z) \\[-2em]
\phi \ar[r,mapsto] & (\iota\phi)(t_\iy,t^\iy) = \indi_{T(F_\iy)_+}(t_\iy)\cdot \phi(t^\iy)\cdot\indi_\mcF(t^\iy). \\[-2em]
\end{tikzcd}
\end{center}
By definition, an element $\Phi\in \Ind_{\mcO_+}^{T(F)}(C(\mcF,\Z))$ is a function $\Phi: T(F) \ra C(\mcF,\Z)$ finitely supported in $T(F)/\mcO_+$, 
and satisfying the compatibility condition $\Phi(t\lambda)=\lambda\inv\cdot\Phi(t)$ for all $\lambda\in \mcO_+$. 
We define the morphism 
\begin{center}
\begin{tikzcd}
\varphi: \Ind_{\mcO_+}^{T(F)}(C(\mcF,\Z)) \ar[r] & C^0_c(T(\A_F),\Z) \\[-2em]
\Phi \ar[r,mapsto] & \varphi(\Phi)=\sum_{t\in T(F)/\mcO_+} t\cdot\iota(\Phi(t)). \\[-2em]
\end{tikzcd}
\end{center}
The sum is finite because $\Phi$ is finitely supported,
and from the $\mcO_+$-equivariance of $\iota$ and the $\mcO_+$-compatibility of $\Phi$
it follows that $\varphi$ is well-defined
and $T(F)$-equivariant.

By lemma \ref{lemma1p2}, for all $x\in T(\A_F)$ there exists a unique $\tau_x\in T(F)/\mcO_+$ such that $\tau_x\inv x\in T(F_\iy)_+\ti \mcF$.
Hence
$\varphi(\Phi)(x)=\iota(\Phi(\tau_x))(\tau_x\inv x)$.
This last relation together with Lemma \ref{lemma1p2} and the injectivity of $\iota$ shows that $\varphi$ is bijective and the result follows.
\epf
Thus, by Shapiro's lemma one may regard
$\eta\in H_u(T(F),C^0_c(T(\A_F),\Z))$.

\subsubsection{The $S$-fundamental classes $\eta^S$}\label{section.fundclS}

In \S\ref{section.fundamentalclass} we have defined a fundamental class $\eta$. As shown in \cite{preprintsanti2}, by means of $\eta$ we can compute certain periods related with certain critical values of classical L-functions.
Nevertheless in \cite{guitart2017automorphic} different fundamental classes $\eta^\mfp$ are defined in order to construct Darmon points.
We devote this section to slightly generalize the definition of $\eta^\mfp$ in \cite{guitart2017automorphic} and we will study the relation between the fundamental classes in the next section.

Let $S$ be a set of places $\mfp$ of $F$ above $p$. Write $S:=S^1\cup S^2$, being $S^1$ the set of places $\mfp$ in $S$ where $T$ splits, namely $T(F_\mfp)= F_\mfp^\ti$, and $S^2$ the set of places where $T$ does not split. 
We consider:
\begin{equation}\label{notation.letusrecall}
\begin{tabular}{cc}
$\mcF^S=\bigsqcup_i \ovl s_i U^S, \quad$ & $\quad U^S =\prod_{\mfq\not\in S} T(\mfo_{F_\mfq}),$ 
\end{tabular}
\end{equation}
where $\ovl s_i\in T(\A_F^{S\cup\iy})$ are representatives of all the elements in $\Cl(T)_+^S=T(\A_F^{S\cup\iy})/\ipa{U^S T(F)_+}$.
Let us consider also the set of totally positive relative $S$-units $\mcO_+^S:=U^S\cap T(F)_+$.
Note that we have an exact sequence
\begin{equation} \label{diagram.exact}
\begin{tikzcd}
0 \ar[r] & T(F_S)/\mcO^S_+T(\mfo_{F_S})  \ar[r] & \Cl(T)_+ \ar[r] & \Cl(T)_+^S \ar[r] & 0 \\[-2em]
\end{tikzcd}
\end{equation}
where $\mfo_{F_S}:=\prod_{\mfp\in S} \mfo_{F_\mfp}$.

It is clear that $\Z$-rank of the quotient $T(F_S)/T(\mfo_{F_S})$ is $r=\#S^1$.
Moreover, we have the natural exact sequence
\begin{equation}\label{diagram.notexact}
\begin{tikzcd}
0 \ar[r] & \mcO_+ \ar[r] & \mcO_+^S \ar[r] & T(F_S)/T(\mfo_{F_S}). \\[-2em]
\end{tikzcd}
\end{equation}
Note that \eqref{diagram.notexact} and \eqref{diagram.exact} induce an exact sequence
\begin{equation}\label{diagram.longerexact}
\begin{tikzcd}
0 \ar[r] & \mcO_+ \ar[r] & \mcO_+^S \ar[r] & T(F_S)/T(\mfo_{F_S}) \ar[r] & \Cl(T)_+ \ar[r] & \Cl(T)_+^S \ar[r] & 0 \\[-2em]
\end{tikzcd}
\end{equation}
This implies that the $\Z$-rank of $\mcO_+^S$ is $r+u$, since $\Cl(T)_+$ is finite.
Write $H^S:=\mcO_+^S/\mcO_+$. The free part $\Lambda_H^S$ of $H^S$ provides a fundamental class $c^S\in H_r(H^S,\Z)$. Indeed, the map
$(v_\mfp)_{\mfp\in S^1}:H^S\rightarrow\R^r$
given by the $p$-adic valuations identifies $\Lambda_H^S$ as a lattice in $\R^r$, hence we can proceed as in the previous section to define $c^S$. We can consider the image $\xi^S\in H_{u+r}(\mcO_+^S,\Z)$ of $c^S$ through the composition
\[
c^S\in H_r(H^S,\Z)\stackrel{1\mapsto\xi}{\longrightarrow}H_r(H^S,H_{u}(\mcO_+,\Z))\longrightarrow H_{u+r}(\mcO_+^S,\Z),
\]
where $\xi\in H_u(\mcO_+,\Z)$ is the fundamental class defined previously, and the last arrow is the edge morphism of the Lyndon–Hochschild–Serre spectral sequence $H_p(H^S,H_q(\mcO_+,\Z))\Rightarrow H_{p+q}(\mcO_+^S,\Z)$.
We have, similarly as in Lemma \ref{lemma.thereisaniso},
\[
C^0_c(T(\A_F^{S\cup\iy}),\Z)=\Ind_{\mcO_+^S}^{T(F)_+} C(\mcF^S,\Z),
\]
where $C^0_c(T(\A_F^{S\cup \iy}),M)$ is the set of $M$-valued locally constant and compactly supported functions of $T(\A_F^{S\cup\iy})$.
Thus, the cap product $\indi_{\mcF^S}\cap \xi^S$ provides an element
\begin{equation}\label{definition.fundamentalclassmfp}
\eta^S:=\indi_{\mcF^S}\cap \xi^S \in H_{u+r}(\mcO_+^S,C(\mcF^S,\Z))=H_{u+r}(T(F)_+,C^0_c(T(\A_F^{S\cup\iy}),\Z)),
\end{equation}
where the last equality follows from Shapiro's Lemma.

\subsubsection{Relation between fundamental classes}

Since $T(F_\iy)/T(F_\iy)_+\simeq T(F)/T(F)_+$, there is a $T(F)$-equivariant isomorphism
\begin{equation}\label{eq.isomorphisminduced}
C_c^0(T(\A_F),M)\ra\Ind_{T(F)_+}^{T(F)}C^0_c(T(\A_F^{\iy}),M):=C^0_c(T(\A_F^{\iy}),M)\oti_{R[T(F)_+]}R[T(F)]
\end{equation}
given by $f\mapsto\sum_{\ovl t\in T(F)/T(F)_+}\ipa{(t\inv \cdot f)|_{T(\A^{\iy})}\oti t}$. Moreover, for any ring $R$
\begin{align}\label{isomorphism.tensorprod}
C^0_c(T(\A_F^{\iy}),R)\simeq C^0_c(T(\A_F^{S\cup\iy}),R)\otimes_R C_c^0(T(F_S),R)\simeq C^0_c(T(\A_F^{S\cup\iy}),C^0_c(T(F_S),R)),
\end{align}
where $C^0_c(T(F_S),R)$ is the set of $R$-valued locally constant and compactly supported functions on $T(F_S)$, seen as $T(F)$-module by means of the diagonal embedding $T(F)\hookrightarrow T(F_S)$. 
Putting these two identities together we obtain
\[C_c^0(T(\A_F),\Z)
=\Ind_{T(F)_+}^{T(F)}\ipa{C_c^0(T(\A_F^{S\cup\iy}),\Z)\oti_\Z C_c^0(T(F_S),\Z)}.\]
The following result relates the previously defined fundamental classes (see also \cite[Lemma 1.4]{Bergunde_2018}): 
\begin{lemma}\label{lemma.onedotseven}
We have that
\[\#(H^S_{\tno{tor}})\cdot \eta = \eta^S \cap \bigcup_{\mfp\in S}\tno{res}^{T(F)_+}_{T(F_\mfp)}z_{\mfp}\]
in $H_u\ipa{T(F)_+,C_c^0(T(\A_F^{S\cup\iy}),\Z)\oti_\Z C_c^0(T(F_S),\Z)}=H_u(T(F),C_c^0(T(\A_F),\Z))$, where $H^S_{\tno{tor}}\subseteq H^S$ is the torsion subgroup, \[z_\mfp=\indi_{T(F_\mfp)}\in H^0(T(F_\mfp),C_c^0(T(F_\mfp),\Z)),\qquad \mbox{if }\mfp\in S^2,\] 
and, if $\mfp\in S^1$ then $z_\mfp\in H^1(T(F_\mfp),C_c^0(T(F_\mfp),\Z))$ is the class associated with the exact sequence
\[0\longrightarrow C_c^0(T(F_\mfp),\Z))\longrightarrow C_c^0(F_\mfp,\Z))\stackrel{f\mapsto f(0)}{\longrightarrow} \Z\longrightarrow 0.\]
\end{lemma}
\bpf
On the one hand that, if $\mfp\in S^1$, the class $z_{\mfp}$ has representative $z_\mfp(t)=\indi_{\mfo_{F_\mfp}}-t\indi_{\mfo_{F_\mfp}}$. Hence $z_\mfp(t)$ is characterized to be the function such that $\sum_{n\in\Z}t^nz_\mfp(t)=\indi_{T(F_\mfp)}$, for all $t\in T(F_\mfp)$ with $v_\mfp(t)> 0$. Thus, $\ipa{\bigotimes_{\mfp\in S} z_{\mfp}}(c^S)$ is characterized so that 
\[\sum_{\gamma\in \Lambda_H^S}\gamma\ipa{\bigotimes_{\mfp\in S}z_{\mfp}}(c^S)=\indi_{T(F_S)}.\]

On the other hand, the exact sequence \eqref{diagram.longerexact} implies
\begin{center}
\begin{tikzcd}
0 \ar[r] & H^S \ar[r] & T(F_S)/T(\mfo_{F_S}) \ar[r] & W \ar[r] & 0 \\[-2em]
\end{tikzcd}
\end{center}
where $W:=\ker\ipa{\Cl(T)_+ \ra \Cl(T)_+^S}$,
and note that
\[\mcF=\bigsqcup_{\ovl s_i\in \Cl(T)_+}s_i U=\bigsqcup_{\ovl t_j\in W}\bigsqcup_{\ovl s_k\in \Cl(T)_+^S}s_k U^S\ti t_j T(\mfo_{F_S})=\bigsqcup_{\ovl t_j\in W}\mcF^S\ti t_j T(\mfo_{F_S}).\]

Hence 
\[\sum_{\beta\in \Lambda_H^S}\beta\sum_{\ovl \gamma\in H^S_{\tno{tor}}}\gamma\cdot\indi_{\mcF}=\sum_{\ovl \alpha\in H^S}\alpha\cdot\ipa{\sum_{\ovl t_j\in T(F_\mfp)/T(\mfo_{F_S})\mcO_+^S}\indi_{\mcF^S}\oti t_j\indi_{T(\mfo_{F_S})}}=\indi_{\mcF^S}\oti\indi_{T(F_S)},\]
and so $\sum_{\ovl \gamma\in H^S_{\tno{tor}}}\gamma\cdot\indi_{\mcF}=\indi_{\mcF^S}\oti\ipa{\bigotimes_{\mfp\in S} z_{\mfp}}(c^S)$. This implies that 
\[
\eta^S\cap\bigcap_{\mfp\in S}\tno{res}_{T(F)_+}^{T(F_\mfp)}z_{\mfp}
= \ipa{\indi_{\mcF^S}\cap\bigcap_{\mfp\in S}\tno{res}_{T(F)_+}^{T(F_\mfp)}z_{\mfp}}\cap\xi^S = \ipa{\sum_{\ovl \gamma\in H^S_{\tno{tor}}}\gamma\indi_{\mcF}}\cap \xi = \#(H^S_{\tno{tor}})\cdot(\indi_\mcF\cap\xi)=\#(H^S_{\tno{tor}})\cdot\eta,
\]
and the result follows.
\epf




\section{Cohomology classes attached to modular elliptic curves}

In this section we will describe cohomology classes associated with modular elliptic curves by means of the Eichler-Shimura map.


Let $H\subset G(F)$ be a subgroup,
$R$ ring,
and $S$ a finite set of places of $F$ above $p$.
We will usually let $H$ to be $G(F), G(F)_+, T(F)$ or $T(F)_+$.
For any $R[H]$-modules $M,N$ let
\begin{equation}\label{definition.automorphicforms}
\mcA^{S\cup\iy}(M,N):=\icla{\phi: G(\A^{S\cup\iy}_F)\ra \Hom_{R}(M,N),\tno{ \begin{tabular}{c} there exists an open compact\\ subgroup $U\subset G(\A_F^{S\cup\iy})$\\ with $\phi(\cdot U)=\phi(\cdot)$ \end{tabular}}},
\end{equation}
and let also $\mcA^{S\cup\iy}(N):=\mcA^{S\cup\iy}(R,N)$.
Then $\mcA^{S\cup\iy}(M,N)$ has a natural action of $H$ and $G(\A_F^{S\cup\iy})$, namely 
\[(h\phi)(x):=(h\cdot \phi(h\inv x)),\qquad (y\phi)(x):=\phi(xy),\]
where  $h \in H$ and $x,y \in G(\A_F^{S\cup\iy})$. For an open compact subgroup $U\subset G(\A_F^{S\cup\iy})$ we denote the subspace of $U$-invariant functions by $\mcA^{S\cup\iy}(\Hom_R(M,N))^U$. We define the cohomology space
\[
H_\ast^r(H,\mcA^{S\cup\iy}(N)):=\varinjlim_U H^r(H,\mcA^{S\cup\iy}(N)^U), \qquad U\subset G(\A_F^{S\cup\iy})\;\mbox{ open compact subgroup.}
\]
\begin{remark}\label{remonadmrepcoho}
    Note that $H_\ast^r(H,\mcA^{S\cup\iy}(N))$ is an admissible $G(\A_F^{S\cup\infty})$-representation. Moreover, the group ${\rm Pic}(U):=G(F)_+\backslash G(\A^{S\cup\infty})\slash U$ is finite for any $U$ as above. Hence, given representatives $\{g_1,\cdots,g_r\}$ of ${\rm Pic}(U)$,
    \[
    \mcA^{S\cup\iy}(N)^U=
    \bigoplus_{i=1}^r{\rm coInd}_{\Gamma_{g_i}}^{G(F)_+}(N);\qquad \Gamma_{g_i}=G(F)_+\cap g_iUg_i^{-1}.
    \]
    This implies that 
    $H^r(G(F)_+,\mcA^{S\cup\iy}(N)^U)=\bigoplus_{i=1}^r H^r(\Gamma_{g_i},N)$. Thus, our cohomology spaces fit with the classical theory of cohomology on $S$-arithmetic groups. In particular the natural morphism
    \[
    H_\ast^r(G(F)_+,\mcA^{S\cup\iy}(N))\otimes_RM\longrightarrow H_\ast^r(G(F)_+,\mcA^{S\cup\iy}(N\oti_RM))
    \]
    is an isomorphism for any flat $R$-module $M$ (see \cite[remark 4.1]{HndzMol2}). 
\end{remark}

\subsection{Automorphic cohomology classes}\label{section:autcla}

Let $E/F$ be a modular elliptic curve. Hence, attached to $E$, we have an automorphic form for $\PGL_2/F$ of parallel weight 2. Let us assume that such form admits a Jacquet-Langlands lift to $G$, and denote by $\pi$ the corresponding automorphic representation. Let $s:=\#\Sigma_B$. The Eichler-Shimura morphism (explained in \cite{ESsanti} for example), allows us to realize $\pi\mid_{G(\A_F^\infty)}$ in the cohomology space $H^s_\ast(G(F)_+,\mcA^\iy(\C))^\lambda$, where $\lambda:G(F)/G(F)_+\rightarrow\pm 1$ is a fixed character and the super-index $\lambda$ stands for the subspace where the natural action of $G(F)/G(F)_+$ on the cohomology groups is given by $\lambda$. Moreover, since the coefficient ring of the automorphic representation is $\Z$, by remark \ref{remonadmrepcoho} such realization is generated by a class
\[
\phi_\lambda\in H_\ast^s(G(F)_+,\mcA^\iy(\Z))^\lambda.
\]
The $G(\A_F^\iy)$-representation $\rho$ over $\Q$ generated by $\phi_\lambda$ satisfies $\rho\otimes_{\Q}\C\simeq\pi^\iy:=\pi\mid_{G(\A_F^\iy)}$. 
For any set $S$ of places above $p$ that split $B$,
write $V_S:=\rho\mid_{G(F_S)}$, with $V_S=\bigotimes_{\mfp\in S}V_\mfp$. As explained in \cite[remark 4.2]{HndzMol2}, we can realize the representation $\rho\mid_{G(\A_F^{\infty\cup S})}$ in the cohomology space $H^s_\ast(G(F)_+,\mcA^{S\cup\iy}(V_S,\Q))^\lambda$. Indeed, for any $v_S\in V_S$ we have a $G(F)$-equivariant morphism 
\begin{equation}\label{eqevaluationmap1}
    \cdot (v_S):\mcA^{S\cup\iy}\ipa{V_S,\Q}\longrightarrow \mcA^{\iy}\ipa{\Q};\qquad \phi(v_S)(g_S,g^S)=\phi(g^S)(g_Sv_S).
\end{equation}
Hence, such realization is generated by an element $\phi_\lambda^S\in H_\ast^s(G(F)_+,\mcA^{S\cup\iy}(V_S,\Q))^\lambda$ such that $\phi_\lambda^S(x_S)=\phi_\lambda$, where $x_S$ is a generator of $V_S$.
For any ring $R$, we denote by $V_S^R=\bigotimes V_\mfp^R$ the $R[G(F_S)]$-module generated by $\phi_\lambda$. 
By remark \ref{remarkonconstants}, we will usually regard $\phi_\lambda^S$ as an element $\phi_\lambda^S\in H_\ast^s(G(F)_+,\mcA^{S\cup\iy}(V_S^L,L))^\lambda$, for $L=\C_p$, $\bar\Q_p$ or $\C$. In case that all primes in $S$ are of multiplicative reduction, by \cite[remark 4.14]{HndzMol2}, we can even consider it as
\[
\phi_\lambda^S\in H_\ast^s(G(F)_+,\mcA^{S\cup\iy}(V_S^{\Z_p},\Z_p))^\lambda.
\]
The element $\phi_\lambda^S$ is the generalization in our setting of the classical harmonic cocycle, and it is essential in our construction of anti-cyclotomic $p$-adic L-functions and plectic points.

\subsection{Pairings}

In order to relate the $S$-arithmetic cohomology groups defined in \S \ref{section:autcla} and the homology groups defined in \S \ref{section:fundcla}, we will define certain pairings that will allow us to perform cap products. 
For this purpose, we will assume the following hypothesis throughout the rest of the paper:
\begin{hypothesis}\label{hypothesis}
Assume that $\Sigma_B=\Sigma_{\tno{un}}(K/F)$. Hence, in particular, $u=s$ and $G(F)/G(F)_+=T(F)/T(F)_+$. 
\end{hypothesis}

For any $T(F)$-modules $M$ and $N$, let us consider the $T(F)_+$-equivariant pairing
\begin{equation}\label{equation.pairingplus}
\begin{tikzcd}
\langle\cdot,\cdot\rangle_+: &[-3em] C_c^0(T(\A_F^{S\cup\iy}),M)\ti\mcA^{S\cup\iy}(M,N) \ar[r] &[-2em] N, \\[-2em]
& (f,\phi) \ar[r,mapsto] & \langle f,\phi\rangle_+ := \int_{T(\A_F^{S\cup\iy})}\phi(t)(f(t)) d^\ti t, \\[-2em]
\end{tikzcd}
\end{equation}
where $d^\times t$ is the aforementioned Haar measure. Hence, once fixed a character $\lambda:G(F)/G(F)_+=T(F)/T(F)_+\rightarrow\pm 1$, it induces a well-defined $T(F)$-equivariant pairing
\begin{equation}\label{equation.pairingnoplus}
\begin{tikzcd}
\langle\cdot,\cdot\rangle:&[-7em] \Ind_{T(F)_+}^{T(F)} C_c^0(T(\A_F^{S\cup\iy}),M)\ti \mcA^{S\cup\iy}(M,N)(\lambda) \longrightarrow N,\\[-2em]
&\Big\langle \sum_{\ovl t\in{T(F)/T(F)_+}}f_t\oti t,\phi\Big\rangle:=[T(F):T(F)_+]\inv\sum_{\ovl t\in T(F)/T(F)_+}t\cdot \langle f_t,t\inv \phi\rangle_+, \\[-2em]
\end{tikzcd}
\end{equation}
where $\mcA^{S\cup\iy}(M,N)(\lambda)$ is the twist of the $T(F)$-representation $\mcA^{S\cup\iy}(M,N)$ by the character $\lambda$.

Assume that $M=C_c^0(T(F_S),R)$, then by  \eqref{eq.isomorphisminduced} and \eqref{isomorphism.tensorprod}
\[
\Ind_{T(F)_+}^{T(F)}C_c^0(T(\A_F^{S\cup\iy}),M)=C^0_c(T(\A_F),R)
\]
This implies that \eqref{equation.pairingnoplus} provides a final $T(F)$-equivariant pairing 
\begin{equation}\label{equation.finalpairing}
\begin{tikzcd}
\langle \cdot|\cdot \rangle: C_c^0(T(\A_F),R)\ti \mcA^{S\cup\iy}(C_c^0(T(F_S),R),N)(\lambda) \longrightarrow N, \\[-2em]
\langle f_S\otimes f^S|\phi\rangle =[T(F):T(F)_+]\inv\sum_{\ovl x\in T(F)/T(F)_+}\lambda(x)\inv\int_{T(\A^{S\cup\iy})}f^S(x,t)\cdot \phi(t)(f_S)d^\ti t. 
\end{tikzcd}
\end{equation}

All the pairings above induce cap products in $H$-(co)homology by their $H$-equivariance.
Now denote by $f_\lambda$ the projection of $f$
to the subspace
\[C_c^0(T(\A_F),R)_\lambda:=\icla{f\in C_c^0(T(\A_F),R)\tno{ with }f|_{T(F_\iy)}=\lambda}.\]
One easily compute that
\[
\langle f|\phi\rangle=\langle f_\lambda|\phi\rangle=\langle f_\lambda\mid_{T(\A^{\iy})},\phi\rangle_+, \qquad f\in C_c^0(T(\A_F),R),\quad\phi\in \mcA^{S\cup\iy}(C_c^0(T(F_S),R),N)(\lambda).
\]
Since we can identify
$H^u(G(F)_+,\bullet)^\lambda\simeq H^u(G(F),\bullet(\lambda))$,
we deduce that for all $f\in H_u(T(F),C_c^0(T(\A_F),R))$ and $\phi\in H^u(T(F)_+,\mcA^{S\cup\iy}(C_c^0(T(F_S),R),N))^\lambda$,
\begin{equation}\label{equation:capprod}
f\cap\phi=f_\lambda\cap\phi=f_\lambda\mid_{T(\A_F^{\iy})}\cap \tno{res}^{T(F)}_{T(F)_+}\phi\in N, 
\end{equation}
where $\tno{res}^{T(F)}_{T(F)_+}$ is the restriction morphism and the cap products are the induced by \eqref{equation.pairingplus},\eqref{equation.finalpairing}, respectively.

\section{Anticyclotomic $p$-adic L-functions}\label{antipLfunct}

In this section we will define anticyclotomic $p$-adic L-functions associated with a finite set $S$ of primes above $p$ that split $B$, the torus $T$ and the automorphic cohomology class $\phi_\lambda^S$. We will understand such $p$-adic L-functions as $p$-adic distributions $\mu_{\phi_\lambda^S}$ of the Galois group associated with $T$.
From this point we will assume that hypothesis \ref{hypothesis} is fulfilled.

\subsection{Defining the distribution}

Let $\mcG_T$ be the Galois group of the abelian extension of $K$ associated with $T$.
By class field theory, there is a continuous morphism
\begin{align}\label{definition.galoisT}
\rho:\ipa{T(F_\iy)/T(F_\iy)_+\ti T(\A_F^\iy)}/T(F) \ra \mcG_T.
\end{align}

If we denote by $C^0$ the space of locally constant functions, the pullback of $\rho$ together with the cap product by $\eta$ give the following morphism
\begin{equation}\label{definition.deltamfp}
\begin{tikzcd}
 C(\mcG_T,\C) \ar[r,"\rho^*"] & H^0(T(F),C^0(T(\A_F),\C)) \ar[r,"\cap\eta"] & H_u(T(F),C_c^0(T(\A_F),\C)). \\[-2em]
\end{tikzcd}
\end{equation}

In order to define our distribution we need to construct a $T(F_S)$-equivariant morphism:
\begin{equation}\label{eq:deltaS}
    \delta_S: C^0_c(T(F_S),\C) \longrightarrow V_S^\C.
\end{equation}
Given such a $\delta_S$ we can define our distribution associated with $\phi_\lambda^S\in H_\ast^u(G(F)_+,\mcA^{S\cup\iy}(V_S^\C,\C))^\lambda$ as follows:
\begin{align}\label{definition.distribution}
\int_{\mcG_T}g d\mu_{\phi^S_\lambda}:=\ipa{\rho^\ast g\cap \eta}\cap\delta_S^\ast\phi^S_\lambda,\quad\tno{ for all }g\in C(\mcG_T,\C)
\end{align}
where the cap product is induced by the pairing \eqref{equation.finalpairing} and $\delta_S^*: \mcA^{S\cup\iy}(V_S^\C,\C)\ra \mcA^{S\cup\iy}\ipa{C_c^0(T(F_S),\C),\C}$ is the corresponding $T(F)$-equivariant pullback.

Notice that $\delta_S$ is defined once you construct $\delta_\mfp:C^0_c(T(F_\mfp),\C) \rightarrow V_\mfp^\C$, for all $\mfp\in S$.
The aim of next sections is to construct $\delta_\mfp$ for principal series or Steinberg representations.

\subsection{The morphism $\delta_\mfp$}\label{deltap}

As seen in the previous section, we want to construct a $T(F_\mfp)$-equivariant morphism
\[
     \delta_\mfp: C_c^0(T(F_\mfp),\ovl\Q) \longrightarrow V_\mfp^{\ovl\Q}.
\]


Let $\pi_\mfp$ be the automorphic local representation and fix an isomorphism $G(F_\mfp)\simeq\PGL_2(F_\mfp)$. From now on we will do the following assumptions:
\begin{hypothesis}
Let $P$ be the subgroup of upper triangular matrices, then we assume that $\imath(K_\mfp^\ti)\not\subset P$. Moreover we will assume that $\pi_\mfp$ is either principal series or Steinberg.
\end{hypothesis}

By the previous hypothesis, $V_\mfp^{\ovl\Q}$ is a quotient of 
\[
\Ind_P^G(\chi_\mfp)=\icla{f\in\GL_2(F_\mfp)\rightarrow\ovl\Q,\mbox{ locally constant }\;f\left(\bbm x_1 & y \\ & x_2 \ebm g \right)=\chi_\mfp\left(\frac{x_1}{x_2}\right)\cdot f(g)},
\]
for a locally constant character $\chi_\mfp$.
Moreover $\imath(K_\mfp^\ti)\cap P=F_\mfp^\ti$, hence we construct  
\begin{equation}\label{eq:deltaSexplicit}
\begin{tikzcd}[ampersand replacement=\&]
\delta_\mfp: C_c^0(T(F_\mfp),\ovl\Q)\ar[r] \&[-3em] \Ind_P^G(\chi_\mfp), \\[-2em] 
f\ar[r,mapsto] \& \delta_\mfp(f)\left(g\right):=
\begin{cases}
\chi_\mfp\left(\frac{x_1}{x_2}\right)\cdot f(t^{-1}), & g=\bsm x_1 \& y \\ \& x_2 \esm\imath(t)\in P\iota(K_\mfp^\ti), \\
\quad 0, &  g\not\in P\iota(K_\mfp^\ti).\\
\end{cases}
\end{tikzcd}
\end{equation}
It is clearly $T(F_\mfp)$-equivariant, thus, it induces the desired 
morphism.

Let us consider $X_\mfp$ to be 
\[
X_\mfp:=\left\{\begin{array}{ll}\uhp_\mfp=K_\mfp\setminus F_\mfp, &\mbox{if $T$ does not split at $\mfp$},\\
\PP^1(F_\mfp),&\mbox{ if splits at $\mfp$}.\end{array}\right.
\]
In both cases $X_\mfp$ comes equipped with a natural action of $G(F_\mfp)$ given by fractional linear transformations. We write $\tau_\mfp$ and $\bar\tau_\mfp$ for the two fixed points by $\iota(T(F_\mfp))$ in $X_\mfp$. Since $\iota(K_\mfp^\ti)\cap P=F_\mfp^\ti$, in the split case $\tau_\mfp,\bar\tau_\mfp\neq\iy$. 
Notice that $v_1^\mfp=(1,-\bar\tau_\mfp)$ and $v_2^\mfp=(1,-\tau_\mfp)$ define a pair of simultaneous eigenvectors of all $\iota(T(F_\mfp))$. Indeed, for any $\tilde t\in K_\mfp^\ti\subseteq\GL_2(F_\mfp)$, if we write $\iota(\tilde t)=\bsm a&b\\c&d\esm$, then 
\begin{equation}\label{equation.eigenvalue}
(1,-\bar\tau_\mfp)\iota(\tilde t)=(1,-\bar\tau_\mfp)\bbm a & b \\ c & d \ebm =(a-c\bar\tau_\mfp) (1,-\iota(\tilde t\inv)\bar \tau_\mfp )=\lambda_{\tilde t}(1,-\bar\tau_\mfp),\qquad (1,-\tau_\mfp)\iota(\tilde t)=\bar \lambda_{\tilde t}(1,-\tau_\mfp).
\end{equation}
\begin{lemma}\label{lemma:mordelta}
The morphism $\delta_\mfp$ is given by  
\[
\delta_\mfp:C^0_c(T(F_\mfp),\ovl\Q)\longrightarrow \Ind_P^G(\chi_\mfp);\qquad\delta_\mfp(f)\bsm a&b\\c&d\esm=\chi_\mfp\ipa{\frac{ad-bc}{(c\tau_\mfp+d)(c\ovl\tau_\mfp+d)}}\cdot f\ipa{\frac{c\bar\tau_\mfp+d}{c\tau_\mfp+d}}.
\]
In particular, if $T(F_\mfp)$ does not split then $\delta_\mfp$ is bijective.
\end{lemma}
\begin{proof}
From the relations
$(1,-\bar\tau_\mfp)\iota(\tilde t)=\lambda_{\tilde t}(1,-\bar\tau_\mfp)$ and $(1,-\tau_\mfp)\iota(\tilde t)=\bar \lambda_{\tilde t}(1,-\tau_\mfp)$
we deduce
\begin{equation}\label{eqiotat}
\imath(\tilde t)=\frac{1}{\tau_\mfp-\ovl \tau_\mfp}\bbm \lambda_{\tilde t}\tau_\mfp-\ovl\lambda_{\tilde t}\ovl \tau_\mfp & \tau_\mfp\ovl\tau_\mfp(\ovl\lambda_{\tilde t}-\lambda_{\tilde t}) \\ \lambda_{\tilde t}-\ovl\lambda_{\tilde t}  & \ovl\lambda_{\tilde t}\tau_\mfp-\lambda_{\tilde t}\ovl\tau_\mfp \ebm=\frac{\ovl\lambda_{\tilde t}}{\tau_\mfp-\ovl \tau_\mfp}\bbm t\tau_\mfp-\ovl \tau_\mfp & \tau_\mfp\ovl\tau_\mfp(1-t) \\ t-1  & \tau_\mfp-t\ovl\tau_\mfp \ebm,
\end{equation}
where $t=\lambda_{\tilde t}/\bar\lambda_{\tilde t}\in T(F_\mfp)$. Hence, if we have
\[
\bbm a&b\\c&d\ebm=\bbm x_1&y\\&x_2\ebm\imath(\tilde t)=\frac{\ovl\lambda_{\tilde t}}{\tau_\mfp-\ovl \tau_\mfp}\bbm x_1&y\\&x_2\ebm\bbm t\tau_\mfp-\ovl \tau_\mfp & \tau_\mfp\ovl\tau_\mfp(1-t) \\ t-1  & \tau_\mfp-t\ovl\tau_\mfp \ebm
\]
we obtain the identities
\[
 -\frac{d}{c}  =  \frac{t\inv \tau_\mfp-\ovl\tau_\mfp}{t\inv -1},\qquad
 ad-bc=x_1x_2\lambda_{\tilde t}\ovl\lambda_{\tilde t},\qquad 
 t\inv =  \frac{\ovl\tau_\mfp+\frac{d}{c}}{\tau_\mfp+\frac{d}{c}}=\frac{c\ovl\tau_\mfp+d}{c\tau_\mfp+d} ,\qquad c=\frac{x_2(t-1)\ovl\lambda_{\tilde t}}{\tau_\mfp-\ovl\tau_\mfp}.
\]
Using such identities, the result follows from a straightforward computation.
\end{proof}

Our newly explained construction provides an element in $\Hom(C^0(\mcG_T,\C_p),\C_p)$, namely, a $p$-adic distribution. It is interesting to study whether these distributions extend to continuous $p$-adic measures or locally analytic distributions. In the ordinary setting, that is to say when $\chi_\mfp$ has values in $\mfo_{\C_p}^\times$, this has been done in \cite[theorem 5.7]{HndzMol2}. In this situation, $\mu_{\phi_\lambda^S}$ extends to a continuous $p$-dic measure.

\subsection{Interpolation properties}\label{section:IntProp}

As we have previously emphasized, we have to think of $\mu_{\phi_\lambda^S}$ as a generalization of Bertolini-Darmon anticyclotomic $p$-adic L-function. Hence, it should have a link to the classical L-function, namely, an interpolation property.

In order to have an explicit interpolation formula, write $N\subseteq\mfo_F$ for the level of $\pi$. Write $\Pi$ for the Jacquet-Langlands lift of $\pi$ to $\PGL_2$, and let $U_0(N)\subseteq \GL_2(\A_F^\infty)$ be the usual open compact subgroup of upper triangular matrices modulo $N$. 
Notice that the usual space of automorphic forms for $\PGL_2$ of parallel weight $2$ and level $U_0(N)$ can be described as
\begin{equation}\label{defS2}
M_{\udl 2}(U_0(N)):=\Hom_{(\mfg_\infty,\mcK_\infty)}\ipa{D(\underline{2}),\mcA(\PGL_2)^{U_0(N)}},    
\end{equation}
where $\mcA(\PGL_2)^{U_0(N)}$ is the space of $U_0(N)$-invariant automorphic forms for $\PGL_2/F$ and the $(\mfg_\infty,\mcK_\infty)$-module $D(\underline 2)=\bigotimes_{\sigma\mid\infty}D_\sigma(2)$, where $D_\sigma(2)$ is the $(\mfg_\sigma,\mcK_\sigma)=(\mfgltwo(\R),O(2))$-module of discrete series of weight $2$, when $F_\sigma=\R$, and $D_\sigma(2)$ is the $(\mfg_\sigma,\mcK_\sigma)=(\mfgltwo(\C),U(2))$-module $\pi(\mu,\mu^{-1})$ with $\mu(t)=(t/\bar t)^{\frac{1}{2}}$ when $F_\sigma=\C$. In \cite[\S 3.2]{preprintsanti2}, a normalized automorphic form $\Psi\in M_{\udl 2}(U_0(N))$ generating $\Pi$ is introduced so that certain period integral of $\Psi$ agrees with the (completed) L-function associated with $\Pi$. In fact, in case $F$ totally real $\Psi$ corresponds to the normalized Hilbert newform.

To provide similar definitions for the group $G$, let $D$ be the distriminant of $K$ and let us consider $c\subset\mfo_F$ an ideal prime to $S$. In order to simplify the resulting formulas, we will assume that, for every finite place $v$, either $\ord_v(c)\geq \ord_v(N)$ or $\ord_v(c)=0$, with $\ord_v(N)\leq 1$ in case $\ord_v(c)=0$ and $K_v$ non-split. In a situation where these conditions are not fulfilled one can still use the results of \cite{preprintsanti3} to compute the interpolation formulas, but they become much more complicated. For any finite place $v$, we say that an Eichler order $\mcO_{N,v}\subseteq B_v$ of discriminant $N\mfo_{F_v}$ is \emph{$c$-admissible} if the intersection $\mfo_{c,v}=K_v\cap\mcO_{N,v}$ is an order of conductor $c\mfo_{F_v}$ and, if in addition $v\mid (N,c)$, then $\mcO_{N,v}$ is the intersection of two maximal orders $\mcO_{B,v}'$ and $\mcO_{B,v}''$ of $B_v$ such that 
    $\mcO_{B,v}'\cap K_v=\mfo_{c,v}$ and $\mcO_{B,v}''\cap K_v=\mfo_{c/N,v}$.
By \cite[lemma 3.6]{preprintsanti3}, for $c$ and $N$ satisfying the above hypothesis admissible Eichler orders always exist.  In addition, we will assume that $V_S$ is a product of local representations that are either spherical or Steinberg, equivalently $\mfp^2\nmid N$ for all $\mfp\in S$, hence $V_\mfp$ are quotients of $\Ind_P^G(\chi_\mfp)$ with $\chi_\mfp$ unramified (this corresponds to hypothesis \ref{restEF}). 

 Let $U_{NS}=U_N^S\times U_0(S)$, where $U_N^S=\prod_{v\nmid S\cup\infty}\mcO_{N,v}^\ti$ and $U_0(S)=\prod_{\mfp\in S}U_0(\mfp)$.
In analogy with \eqref{defS2} 
\begin{equation}\label{defS22}
M_{\udl 2}(U_{NS}):=\Hom_{(\mfg_\infty,\mcK_\infty)}\ipa{D(\underline{2}),\mcA(G)^{U_{NS}}},    
\end{equation}
where $\mcA(G)^{U_{NS}}$ is the space of $U_{NS}$-invariant automorphic forms for $G/F$ and the $(\mfg_\infty,\mcK_\infty)$-module $D(\underline 2)=\bigotimes_{\sigma\in\Sigma_B}D_\sigma(2)$. In \cite[\S 3.2]{preprintsanti2}, a natural bilinear inner product $\langle\;,\;\rangle:M_{\udl 2}(U_{NS})\times M_{\udl 2}(U_{NS})\rightarrow\C$ is defined.
As explained in \cite[remark 3.2]{preprintsanti2}, if $F$ is totally real and $G=\PGL_2$ then we have $\langle\phi,\bar\phi\rangle={\rm vol}(U_0(N))\left(2\pi\right)^d(\phi,\phi)_{U_0(N)}$, where $(\;,\;)_{U_0(N)}$ is the usual Petersson inner product.
In \cite[\S 3.1]{preprintsanti2} an Eichler-Shimura morphism ${\rm ES}_\lambda$ is described for a fixed characted $\lambda$ (see also \cite{ESsanti}):
\[
{\rm ES}_\lambda:M_{\udl 2}(U_{NS})\longrightarrow H^u(G(F)_+,\mcA^{\iy}(\C)^{U_{NS}})^{\lambda}.
\] 
If we assume that $\phi^S_\lambda\in H^u(G(F)_+,\mcA^{S\cup\iy}(V_S,\C)^{U_N^S})^{\lambda}$ then we can consider $\Phi_{\alpha_S}\in M_{\udl 2}(U_{NS})$ such that
\[
{\rm ES}_\lambda\ipa{\Phi_{\alpha_S}}=\phi^S_\lambda(f_S),\qquad f_S=\bigotimes_{\mfp\in S}f_{0,\mfp}\in V_S,
\]
where the vectors $f_{0,\mfp}\in V_\mfp^{U_0(\mfp)}$ are defined in the formula \eqref{deff0} of the appendix and $\phi^S_\lambda(f_S)$ is the image through the evaluation morphism \eqref{eqevaluationmap1}. Since $U_\mfp f_{0,\mfp}=\alpha_\mfp f_{0,\mfp}$, with $\alpha_\mfp=\chi_\mfp(\varpi_\mfp)^{-1}$, the form $\Phi_{\alpha_S}$ can be seen as the $\mfp$-stabilization of a newform for all $\mfp\in S$.

Let $\epsilon(B_v)=1$ if $B_v\simeq M_2(F_v)$ and $\epsilon(B_v)=-1$ otherwise.
We say that a finite character $\xi\in C(\mcG_T,\C)$ has conductor $c$ outside $S$ if the conductor $c_\xi$ of $\rho^\ast\xi$ satisfies $c\mfo_{F_v}=c_\xi\mfo_{F_v}$, for any finite place $v\not\in S$. We will write $\xi_v=\rho^\ast\xi\mid_{T(F_v)}$, for any finite place $v$. The embedding $K\subset B$ provides a decomposition $B=K\oplus KJ_0$, with  $J_0$ in the normalizer of $K$ in $B$ satisfying $J_0^2=M\in F^\times$. By \cite[lemma 3.6]{preprintsanti3} there exists $k_0=(k_{0,v})_v\in T(\A_F)$ such that $k_0^{-1}J_0\in G(F_\infty)_+\times w_{\Sigma^{D,S}}\hat\mcO_N^\times$, where $w_{\Sigma^{D,S}}$ is the Atkin-Lehner involution $w_{\Sigma^{D,S}}=\prod_{v\in \Sigma^{D,S}}w_v$, with $\Sigma^{D,S}=\{v\mid N,\;v\not\in S:\;{\rm ord}_v(c)=0;\;v\nmid D\}$. We denote by $\varepsilon(\Sigma^{D,S})$ the sign that provides the action of $w_{\Sigma^{D,S}}$ on $\pi^{U_N^S}$.

\begin{theorem}\label{THMintprop}
Let $\xi\in C(\mcG_T,\C)$ be a finite character of conductor $c$ outside $S$.  If $\rho^\ast\xi\mid_{T(F_\iy)}\neq\lambda$ or the local root number $\epsilon(1/2,\pi_v,\xi_v)\neq\xi_v\psi_{K_v}(-1)\epsilon(B_v)$ for any finite place $v\not\in S$, then $\int_{\mcG_T}\xi d\mu_{\phi^S_\lambda}=0$. Otherwise, 
\[
\ipa{\int_{\mcG_T}\xi d\mu_{\phi^S_\lambda}}^2=
\frac{\varepsilon({\Sigma^{D,S}})2^{\#\Sigma_D^S}L_{c}(1,\psi_{K})^2h^2
}{\rho^\ast\xi^{-1}(k^S)|c_\xi^2 D|^{\frac{1}{2}}C} L^{\Sigma^S}(1/2,\Pi,\rho^\ast\xi)\frac{\langle \Phi_{\alpha_S},\Phi_{\alpha_S}\rangle}{\langle \Psi,\Psi\rangle}\frac{{\rm vol}(U_0(N))}{{\rm vol}(U_N){\rm vol}(\mfo_{K_S}^\ti/\mfo_{F_S}^\ti)^{2}}\prod_{\mfp\in S}C(\pi_\mfp)C(\pi_\mfp,\xi_\mfp),
\]
where $h=\#\mcO_{+,\mbT}$, $L_c(1,\psi_{K})$ is the product of local L-functions at places $v\mid c$, $\Sigma_D^S=\{v\mid (N,D);\;\ord_v(c_\xi)=0;\;\mfp\not\in S\}$, $\Sigma^S=\{v\mid (N,c)\}\cup\{\mfp\in S;\;\mfp\nmid(N,D)\}$,  $L^{\Sigma^S}(1/2,\Pi,\rho^\ast\xi)$ is the L-function with the local factors at places $v\mid \Sigma^S\cup\infty$ excluded, $k^S=(k^S_v)\in T(\A_F)$ satisfies $k^S_v=k_{0,v}$ at places $v\not\in S$ and $J_0k_\mfp^S \in P$ at places $\mfp\in S$,
\[
C=(-2^{2}\pi^{2})^{r_{K/F,\R}}\cdot (-2^{3}\pi^{3})^{s_F}\cdot(2^{4}\pi^{3})^{r_{K/F,\C}};\qquad C(\pi_\mfp)=\left\{
    \begin{array}{ll}
       \zeta_\mfp(1)^2L(0,\chi_\mfp^2)\zeta_{\mfp}(2)^{-1},  &\mfp\nmid N;  \\
        \zeta_\mfp(1)\zeta_{\mfp}(2)^{-1},&\mfp\mid N,\\
    \end{array}\right.
\]
and, finally, if we write $\beta_\mfp=\chi_\mfp(\varpi_\mfp)q_\mfp$ then
    \[
    C(\pi_\mfp,\xi_\mfp)=\left\{
    \begin{array}{ll}
    \varepsilon(\pi_\mfp,\xi_\mfp)\chi_{\mfp}\left(\frac{D}{(\tau_\mfp-\ovl\tau_\mfp)^2}\right),&\ord_\mfp(c_\xi)=0;\\
    \beta_\mfp^{2\ord_\mfp(c_\xi)}\chi_{\mfp}\left(\frac{D}{(\tau_\mfp-\ovl\tau_\mfp)^2}\right),&\ord_\mfp(c_\xi)>0,
    \end{array}\right.,\qquad \varepsilon(\pi_\mfp,\xi_\mfp)=\frac{L(-1,\xi_\mfp\cdot\chi^{-1}_\mfp\circ{\rm N}_{K_\mfp/F_\mfp})}{L(0,\xi_\mfp\cdot\chi_\mfp\circ{\rm N}_{K_\mfp/F_\mfp})}.
    \]
\end{theorem}
\begin{proof}
Let us consider the $T(F)$-equivariant morphism
\[
\varphi:C^0(T(\A_F),\C)\otimes\mcA^\iy(\C)(\lambda)\longrightarrow C^0(T(\A_F),\C);\qquad\varphi(f\oti\phi)(z,t):=\frac{f(z,t)}{\lambda(z)}\cdot\phi(t),
\]
for all $z\in T(F_\iy)$ and $t\in T(\A_F^\iy)$, and the natural pairing $\langle\cdot,\cdot\rangle_T:C^0(T(\A_F),\C)\ti C_c^0(T(\A_F),\C)\ra\C$ given by the Haar measure
\[
\langle f_1,f_2\rangle_T:=\sum_{z\in CC}\int_{\mbT}\int_{T(\A_F^\iy)}f_1(xs,t)\cdot f_2(xs,t)d^\ti t d^\times s,\qquad CC=T(F_\iy)/T(F_\iy)_+.
\]
For any $f\in C^0(T(\A_F),\C)$ and $f_2\in C_c^0(T(\A_F),\C)$, write $f\cdot f_2=f_S\otimes f^S$, where $f_S\in C_c(T(F_S),\C)$ and $f^S\in C_c(T(\A_F^S),\C)$, and let $H=\prod_\mfp H_\mfp\subseteq T(F_S)$ be a small enough open compact subgroup so that $f_S$ is $H$-invariant, namely, $f_S=\sum_{t_S\in T(F_S)/H}f_S(t_S)\indi_{t_SH}$. We compute using the concrete description of $\langle\cdot|\cdot\rangle$ provided in \eqref{equation.finalpairing}: 
\begin{eqnarray*}
\langle\varphi(f\oti \phi(\delta_S(\indi_H))),f_2\rangle_T
&=&{\rm vol}(\mbT)\sum_{z\in CC}\lambda(z)\inv\cdot\int_{T(\A_F^{S\cup\iy})}\int_{T(F_{S})}f^S(z,t^S)\cdot f_S(t_S)\cdot\phi(t^S)(\delta_S(\indi_{t_SH}))d^\ti t^S d^\ti t_S\\
&=&{\rm vol}(\mbT)\cdot{\rm vol}(H)\sum_{z\in CC}\lambda(z)\inv\cdot\int_{T(\A_F^{S\cup\iy})}f^S(z,t^S)\cdot\delta_S^\ast\phi(t^S)(f_S)d^\ti t^S \\
&=&[T(F):T(F)_+]\cdot{\rm vol}(\mbT)\cdot{\rm vol}(H)\cdot\langle f\cdot f_2|\delta_S^\ast\phi\rangle,
\end{eqnarray*}
for all $\phi\in \mcA^{S\cup\iy}(V_S,\C)(\lambda)$. 
Since $[T(F):T(F)_+]=2^{r_{K/F,\R}}$ and ${\rm vol}(\mbT)=2^{r_{K/F,\C}}4^{s_F}$, we obtain by definition
\[\int_{\mcG_T}\xi d\mu_{\phi^S_\lambda}=(\rho^\ast\xi\cap\eta)\cap\delta_S^\ast\phi^S_\lambda=\frac{1}{2^{[F:\Q]}\cdot{\rm vol}(H)}(\rho^\ast\xi\cup \phi_H)\cap\eta,\qquad \phi_H:=\phi_\lambda^S(\delta_S(\indi_H)),
\]
where the cap and cup products in the  third identity correspond to the pairings $\langle\cdot,\cdot\rangle_T$ and $\varphi$. In \cite[Theorem 1.3]{preprintsanti3} and expression for $\left((\rho^\ast\xi\cup \phi_H)\cap\eta\right)^2$ is obtained in terms of the classical L-function. More precisely, $\left((\rho^\ast\xi\cup \phi_H)\cap\eta\right)=0$ unless $\epsilon(1/2,\pi_v,\xi_v)=\xi_v\psi_{K_v}(-1)\epsilon(B_v)$ and $\rho^\ast\xi\mid_{T(F_\iy)}=\lambda$ and, if this is the case, then 
\begin{equation*}\label{equation.theorem}
\ipa{\frac{(\rho^\ast\xi\cup \phi_H)\cap\eta}{{\rm vol}(H)}}^2 = \frac{\varepsilon({\Sigma^D})2^{\#\Sigma_D}L_{c_\xi}(1,\psi_{K})^2h^2
\bar C}{\rho^\ast\xi^{-1}(k_{0})|c_\xi^2 D|^{\frac{1}{2}}}\cdot L^{\Sigma_{c_\xi}}(\frac{1}{2},\Pi,\rho^\ast\xi)\cdot\frac{\langle \Phi_H,\Phi_H\rangle}{\langle \Psi,\Psi\rangle}\cdot\frac{{\rm vol}(U_0(N))}{{\rm vol}(U_N)}\prod_{\mfp\in S}\frac{\lambda_\mfp(\delta_\mfp(\indi_{H_\mfp}),J_0)}{{\rm vol}(H_\mfp)^2\lambda_\mfp(\phi_{0,\mfp},J_0)},
\end{equation*}
where $\bar C=(-\pi^{-2})^{r_{K/F,\R}}\cdot (-2\pi^{-3})^{s_F}(2^{-2}\pi^{-3})^{r_{K/F,\C}}$,  $\Sigma_{c_\xi}:=\{v\mid (N,c_\xi)\}$, $\Sigma_D=\{v\mid (N,D);\;\ord_v(c_\xi)=0\}$, $\Sigma^{D}=\{v\mid N:\;{\rm ord}_v(c_\xi)=0;\;v\nmid D\}$, $\Phi_H\in M_{\underline k}(U)$ for some $U\subseteq G(\A_F^\infty)$ is such that ${\rm ES}_\lambda(\Phi_H)=\phi_H$, the open $U_N=\prod_{v\nmid\infty}\mcO_{N,v}^\times$ with $\mcO_{N,\mfp}$ a $c_\xi$-admissible Eichler order, $\phi_{0,\mfp}=\pi^{\mcO_{N,\mfp}^\times}$ and
    \begin{equation}\label{localfactlambda}
\lambda_\mfp(\phi_\mfp,J_0)=\int_{T(F_\mfp)}\xi_\mfp(t)\frac{\langle\pi_\mfp(t)\phi_{\mfp}^\lambda,\pi_\mfp(J_0)\phi_{\mfp}^\lambda\rangle_\mfp}{\langle\phi_{\mfp}^\lambda,\phi_{\mfp}^\lambda\rangle_\mfp} d^\times t,
    \end{equation}
being $\langle\;,\;\rangle_\mfp$ any $G(F_\mfp)$-invariant bilinear inner product, and $d^\times t$ any Haar measure. 
Notice that
\begin{equation}\label{PhiPhiint}
    \langle \Phi_H,\Phi_H\rangle\prod_{\mfp\in S}\frac{\lambda_\mfp(\delta_\mfp(\indi_{H_\mfp}),J_0)}{{\rm vol}(H_\mfp)^2\lambda_\mfp(\phi_{0,\mfp},J_0)}=\langle \Phi_{\alpha_S},\Phi_{\alpha_S}\rangle\prod_{\mfp\in S}\int_{T(F_\mfp)}\xi_\mfp(t)\frac{\langle\pi_\mfp(t)\delta_{\mfp}(\indi_{H_\mfp}),\pi_\mfp(J_0)\delta_{\mfp}(\indi_{H_\mfp})\rangle_\mfp}{{\rm vol}(H_\mfp)^2\lambda_\mfp(\phi_{0,\mfp},J_0)\langle f_{0,\mfp},f_{0,\mfp}\rangle_\mfp} d^\times t.
\end{equation}
On the one hand, the main factor is computed in corollary \ref{coroLIcomp}:
\[
\int_{T(F_\mfp)}\xi_\mfp(t)\frac{\langle\pi_\mfp(t)\delta_{\mfp}(\indi_{H_\mfp}),\pi_\mfp(J_0)\delta_{\mfp}(\indi_{H_\mfp})\rangle_\mfp}{{\rm vol}(H_\mfp)^2\langle f_{0,\mfp},f_{0,\mfp}\rangle_\mfp} d^\times t=\frac{\chi_\mfp(D)\zeta_\mfp(1){\rm vol}(\mfo_{K_\mfp}^\ti/\mfo_{F_\mfp}^\ti)^{-1}\xi_\mfp^{-1}(k_{\mfp}^S)}{\chi_{\mfp}\left((\tau_\mfp-\ovl\tau_\mfp)^2\right)L(1,\psi_{K_\mfp}) L(2,\chi_\mfp^{-2})}\left\{\begin{array}{ll}
        \varepsilon(\pi_\mfp,\xi_\mfp),&\mfp\nmid c_\xi;\\
        \alpha_\mfp^{-n_\xi}\beta_\mfp^{n_\xi},&\mfp\mid c_\xi.
    \end{array}\right.
\] 
where $n_\xi=\ord_\mfp(c_\xi)$.
On the other hand, it is computed in \cite[proposition 3.12]{CST} the value 
\begin{equation*}\label{beta_p}
    \int_{T(F_\mfp)}\xi_\mfp(t_\mfp)\frac{\langle\pi_\mfp(t)\phi_{0,\mfp},\phi_{0,\mfp}\rangle_\mfp}{\langle \phi_{0,\mfp},\phi_{0,\mfp}\rangle_\mfp} d^\times t={\rm vol}(\mfo_{K_\mfp}^\ti/\mfo_{F_\mfp}^\ti)\left\{
    \begin{array}{ll}
       \frac{\zeta_{\mfp}(2) L(1/2,\Pi_\mfp,\xi_\mfp)}{L(1,\psi_{K_\mfp}) L(1,\Pi_\mfp,{\rm ad})},  &\ord_\mfp(N)=n_{\xi_\mfp}=0;  \\
        \frac{\zeta_{\mfp}(2) L(1/2,\Pi_\mfp,\xi_\mfp)}{L(1,\Pi_\mfp,{\rm ad})}L(1,\psi_\mfp)q_\mfp^{-n_{\xi_\mfp}}, &\ord_\mfp(N)=0,\;n_{\xi_\mfp}>0;  \\
        \frac{ L(1/2,\Pi_\mfp,\xi_\mfp)}{ \zeta_\mfp(1)},&\mfp\parallel N,\;n_{\xi_\mfp}=0,\;T(F_\mfp)=F_\mfp^\times;\\
        2,&\mfp\mid (N,D),\;n_{\xi_\mfp}=0;\\
        L(1,\psi_\mfp)q_\mfp^{-n_{\xi_\mfp}},&n_{\xi_\mfp}>\ord_\mfp(N)>0
    \end{array}\right.
\end{equation*}
Notice that $L(1,\Pi_\mfp,{\rm ad})=\zeta_\mfp(1)$ in the Steinberg case and $L(1,\Pi_\mfp,{\rm ad})=\zeta_\mfp(1)L(2,\chi_\mfp^{-2})L(0,\chi_\mfp^2)$ in the spherical case. Moreover, 
\[
\lambda_\mfp(\phi_{0,\mfp},J_0)=s_\mfp\xi_\mfp(k_{0,\mfp})\int_{T(F_\mfp)}\xi_\mfp(t_\mfp)\frac{\langle\pi_\mfp(t)\phi_{0,\mfp},\phi_{0,\mfp}\rangle_\mfp}{\langle \phi_{0,\mfp},\phi_{0,\mfp}\rangle_\mfp} d^\times t
\]
where $s_\mfp$ is the eigenvalue of the Atkin-Lehner $w_\mfp$ operator in case $\mfp\nmid c_{\xi}D$, otherwise it is $1$. 
Hence, we obtain that the local terms of the right hand side of \eqref{PhiPhiint} are given by
\begin{equation}\label{PhiPhiint2}
\frac{\xi_\mfp^{-1}(k_{\mfp}^Sk_{0,\mfp})}{{\rm vol}(\mfo_{K_\mfp}^\ti/\mfo_{F_\mfp}^\ti)^{2}}\chi_{\mfp}\left(\frac{D}{(\tau_\mfp-\ovl\tau_\mfp)^2}\right)
\left\{
    \begin{array}{ll}
       \frac{\zeta_\mfp(1)^2L(0,\chi_\mfp^2)}{\zeta_{\mfp}(2) L(1/2,\Pi_\mfp,\xi_\mfp)}\varepsilon(\pi_\mfp,\xi_\mfp),  &\ord_\mfp(N)=n_{\xi_\mfp}=0;  \\
        \frac{\zeta_\mfp(1)^2L(0,\chi_\mfp^2)}{\zeta_{\mfp}(2) L(1/2,\Pi_\mfp,\xi_\mfp)}\frac{\beta_\mfp^{2n_\xi}}{L(1,\psi_{K_\mfp})^2 }, &\ord_\mfp(N)=0,\;n_{\xi_\mfp}>0;  \\
        \frac{s_\mfp\zeta_\mfp(1)}{ \zeta_{\mfp}(2)L(1/2,\Pi_\mfp,\xi_\mfp)}\varepsilon(\pi_\mfp,\xi_\mfp),&\mfp\parallel N,\;n_{\xi_\mfp}=0,\;T(F_\mfp)=F_\mfp^\times;\\
        \frac{\zeta_\mfp(1)}{2 \zeta_{\mfp}(2)}\varepsilon(\pi_\mfp,\xi_\mfp),&\mfp\mid (N,D),\;n_{\xi_\mfp}=0;\\
        \frac{\zeta_\mfp(1)}{\zeta_{\mfp}(2)}\frac{\beta_\mfp^{2n_\xi}}{L(1,\psi_{K_\mfp})^2 },&n_{\xi_\mfp}>\ord_\mfp(N)>0
    \end{array}\right.
\end{equation}
and the result follows.
\end{proof}

\begin{remark}\label{remark:excepzeros}
Notice that $\varepsilon(\pi_\mfp,\xi_\mfp)=\frac{\ipa{1-\xi_\mfp^{-1}(\varpi_\mfp)\alpha_\mfp^{-1}}\ipa{1-\xi_\mfp(\varpi_\mfp)\alpha_\mfp^{-1}}}{\ipa{1-\xi_\mfp^{-1}(\varpi_\mfp)\beta_\mfp^{-1}}\ipa{1-\xi_\mfp(\varpi_\mfp)\beta_\mfp^{-1}}}$ if $T(F_\mfp)$ splits, $\varepsilon(\pi_\mfp,\xi_\mfp)=\frac{1-\alpha_\mfp^{-2}}{1-\beta_\mfp^{-2}}$ if $T(F_\mfp)$ inert, and $\varepsilon(\pi_\mfp,\xi_\mfp)=\frac{1-\xi_\mfp(\varpi_\mfp^{K})\alpha_\mfp^{-1}}{1-\xi_\mfp(\varpi_\mfp^{K})\beta_\mfp^{-1}}$ if $T(F_\mfp)$ ramifies, being $\varpi_\mfp^K$ the uniformizer of $K_\mfp$ (see proposition \ref{EulerFapp}).
If $\pi_\mfp\simeq \tno{St}_\C(F_\mfp)(\pm)$ is the Steinberg representation twisted by $(\pm1)^{v_\mfp\det g}$, then we have that $\alpha_\mfp=\pm 1$.

Hence if $\xi_\mfp$ is unramified and $\xi_\mfp(\varpi_\mfp^K)=\pm1$ then $\varepsilon(\pi_\mfp,\xi_\mfp)=0$. This phenomena is known as {\bf exceptional zero} and the aim of the following chapters is to relate these exceptional zeroes with points in the extended Mordell-Weil group.
\end{remark}



\section{Construction of Fornea-Gehrmann plectic points}\label{section:fornea-gehrmann}

M. Fornea and L. Gehrmann came in \cite{fornea2021plectic} with a very novel and interesting construction with which they are able to provide elements in certain completed tensor products of elliptic curves. These elements, called {\bf plectic points}, generalize classical constructions of Heegner and Darmon points. Moreover, they are conjectured to be non-zero in rank $r\leq[F:\Q]$ situations, hence they open the door to important progress towards our understanding of the Birch and Swinnerton-Dyer conjecture in rank $r\geq 2$ situations (although we are restricted to the case $r\leq [F:\Q]$). 

During this section let $S$ be a set of finite places above $p$ and let us fix an isomorphism $G(F_S)\simeq\PGL_2(F_S)$. Assume that the representation $V_S^{\Z_p}=\bigotimes_{\mfp\in S}\tno{St}_{\Z_p}(F_\mfp)(\varepsilon_\mfp)$, where $\tno{St}_{\Z_p}(F_\mfp)=C^0(\PP^1(F_\mfp),\Z_p)/\Z_p$ with the usual action of $\PGL_2(F_\mfp)$, $\varepsilon_\mfp:G(F_\mfp)\ra \pm 1$ is given by $g\mapsto (\pm 1)^{v_\mfp\det g}$, and $T$ does not split at any $\mfp\in S$.

\subsection{$S$-adic uniformization}

\newcommand{\ddivz}{\Delta^0}
\newcommand{\ddivv}{\Delta}

Denote by $\tno{Cov}(X)$ the poset of open coverings of a topological space $X$ ordered by refinement. Write analogously as above $\tno{St}_{\Z}(F_\mfp)=C^0(\PP^1(F_\mfp),\Z)/\Z$. Let us consider again $\uhp_\mfp=K_\mfp\setminus F_\mfp$ the $p$-adic upper half plane.
Let $\Delta_\mfp:=\tno{Div}(\uhp_\mfp)$ and $\Delta_\mfp^0:=\tno{Div}^0(\uhp_\mfp)\subset \Delta_\mfp$ the set of divisors and degree zero divisors.
Notice we have the multiplicative integral $i:\Hom(\tno{St}_\Z(F_\mfp),\Z) \ra \Hom(\ddivz_\mfp,K_\mfp^\ti)$ defined as follows,
\begin{equation}\label{definition.multiplicativeintegral}
\begin{tikzcd}
i_\mfp:&[-2em] \Hom(\tno{St}_\Z(F_\mfp),\Z) \ar[r] & \Hom(\ddivz_\mfp,K_\mfp^\ti) \\[-2em]
& \psi\ar[r,mapsto] & \ipa{\tau_2-\tau_1\mapsto \mint_{{\scriptscriptstyle \PP^1(F_\mfp)}} \frac{x-\tau_2}{x-\tau_1}d\mu_{\psi}(x)=\lim_{\mcU\in\tno{Cov}(\PP^1(F_\mfp))}\prod_{U\in\mcU}\ipa{\frac{x_U-\tau_2}{x_U-\tau_1}}^{\psi(\indi_U)}}, \\[-2em]
\end{tikzcd}
\end{equation}
where each $x_U\in U$. 
Twisting by a character $\varepsilon_\mfp$, the $G(F_\mfp)$-invariant morphism \eqref{definition.multiplicativeintegral} provides a $G(F_\mfp)$-invariant morphism
\begin{equation}\label{definition.multiplicativeintegral2}
i_\mfp: \Hom(\tno{St}_\Z(F_\mfp)(\varepsilon_\mfp),\Z) \longrightarrow \Hom(\ddivz_\mfp,K_\mfp^\ti)(\varepsilon_\mfp).
\end{equation}

\begin{remark}\label{extSt}
In \cite{guitart2017automorphic} this multiplicative integral is described as follows:
If we denote $\tno{St}_{F_\mfp^\ti}:=C(\PP^1(F_\mfp),F_\mfp^\ti)/F_\mfp^\ti$,
then we have a natural morphism
\begin{equation}\label{morextSt}
   \varphi_{\tno{unv}}:\Hom(\tno{St}_\Z(F_\mfp),\Z)\longrightarrow \Hom(\tno{St}_{F_\mfp^\ti},F_\mfp^\ti),\qquad \varphi_{\tno{unv}}(\psi)(f):=\lim_{\mcU\in\tno{Cov}(\PP^1(F_\mfp))}\prod_{U\in\mcU}f(x_U)^{\psi(\indi_U)}
\end{equation}
Moreover, if we define the universal extension of $\tno{St}_{F_\mfp^\ti}$:
\begin{equation}\label{defextSt}
    \mcE_{F_\mfp^\ti}:=\left\{(\phi,y)\in C(\GL_2(F_\mfp),F_\mfp^\ti)\ti\Z:\phi\left(\bbm s & x \\  & t \ebm g\right)=t^y\cdot\phi(g) \right\}/(K_\mfp^\ti,0),
\end{equation}
then we have a natural morphism $\tno{ev}:\ddivv_\mfp\ra \mcE_{K_\mfp^\ti}$ making the diagram commutative
\[
\begin{tikzcd}
0\ar[r]&\ddivv_\mfp^0\ar[r]\ar[d,"\tno{ev}"]&\ddivv_\mfp\ar[r]\ar[d,"\tno{ev}"]&\Z\ar[r]\ar[d,"\tno{Id}"]&0\\
0\ar[r]&\tno{St}_{K_\mfp^\ti}\ar[r]&\mcE_{K_\mfp^\ti}\ar[r]&\Z\ar[r]&0
\end{tikzcd}
\]
Finally, $i$ can be described as $i=\tno{ev}^*\circ\varphi_{\tno{unv}}$.
\end{remark}

Notice that the restriction $\varepsilon_\mfp:T(F_\mfp)\ra\pm1$ is non trivial only if $\varepsilon_\mfp\neq 1$ and $T$ ramifies at $\mfp$.
Moreover, if $H_{\varepsilon_\mfp}/K_\mfp$ is the extension cut out by $\varepsilon_\mfp$, by \cite[Corollary 5.4]{Silverman}
\begin{equation}\label{non-splitTate}
E(K_\mfp)=\left\{u\in H_{\varepsilon_\mfp}^\times/q_{E_\mfp}^\Z:\;u^{-\varepsilon_\mfp(\tau)}u^\tau\in q_{E_\mfp}^\Z\right\},\qquad 1\neq\tau\in\Gal(H_{\varepsilon_\mfp}/K_\mfp).    
\end{equation}
Hence, we no longer have a Tate uniformization $K_\mfp^\times/q_{E_\mfp}^\Z\sim E(K_\mfp)$ but an isomorphism
\[
K_\mfp^\times/q_{E_\mfp}^\Z\sim E(K_\mfp)_{\varepsilon_\mfp}:=\{P\in E(H_{\varepsilon_\mfp});\;P^\tau=\varepsilon_\mfp(\tau)\cdot P\}.
\]

Given the automorphic modular symbol
$\phi_\lambda^S\in H^s_\ast(G(F)_+,\mcA^{S\cup\iy}(V_S^{\Z_p},\Z_p))^\lambda$ associated to $E$,
we can apply the morphism $i_S=\prod_{\mfp\in S}i_\mfp$ induced by \eqref{definition.multiplicativeintegral2} to obtain
\[
i_S\phi_\lambda^S\in H^u_\ast(G(F)_+,\mcA^{S\cup\iy}(\ddivv_S^0,\hat K_S^\times)(\varepsilon_S))^\lambda,\qquad \ddivv_{S}^0:=\bigotimes_{\mfp\in S}\ddivv_{\mfp}^0,\qquad \hat K_S^\ti:=\bigotimes_{\mfp\in S} \hat K_\mfp^\ti,\qquad \hat K_\mfp^\ti=K_\mfp^\ti\oti_\Z\Z_p,
\]
where in $\bigotimes_{\mfp\in S}\ddivv_{\mfp}^0$ the tensor product is taken as $\Z$-modules, and in $\bigotimes_{\mfp\in S} \hat K_\mfp^\ti$ is taken as $\Z_p$-modules. Applying the natural projection modulo $\sum_{\mfp\in S}\ipa{q_{E_\mfp}^{\Z_p}\otimes\bigotimes_{\mfq\in S\setminus\{\mfp\}} \hat K_\mfq^\ti}$, one obtains
\[
\ovl{i_S\phi_\lambda^S}\in H^u_\ast(G(F)_+,\mcA^{S\cup\iy}(\ddivv_S^0,\hat E(K_S)_{\varepsilon_S})(\varepsilon_S))^\lambda,\qquad \hat E(K_S)_{\varepsilon_S}:=\bigotimes_{\mfp\in S} \ipa{E(K_\mfp)_{\varepsilon_\mfp}\oti_\Z\Z_p},
\]
where $\varepsilon_{S}:=\prod_{\mfp\in S}\varepsilon_{\mfp}$.
Fornea an Gehrmann show in \cite{fornea2021plectic} that, assuming that Oda's conjecture \cite[Conjecture 3.8]{guitart2017automorphic} is true, $\ovl{i_S\phi_\lambda^S}$ is the restriction of a unique
\[
\psi_\lambda^S\in H^u_\ast(G(F)_+,\mcA^{S\cup\iy}(\ddivv_{S},\hat E(K_S)_{\varepsilon_S}\otimes_{\Z_p}\Q_p)(\varepsilon_{S}))^\lambda,
\]
where $\ddivv_{S}:=\bigotimes_{\mfp\in S}\ddivv_{\mfp}$ is the tensor product over $\Z$.  Throughout the paper we will also assume Oda's conjecture. 

\subsection{Plectic points}

For all $\mfp\in S$, let $\tau_\mfp\in\uhp_\mfp$ be the points fixed by $T(F)$ as in \S \ref{deltap}.
Let us consider the following well-defined morphism of $G(F)_+$-modules
\begin{equation}
\tno{ev}: \Z[G(F)_+/T(F)_+] \longrightarrow \ddivv_S; \qquad  n\cdot gT(F)_+ \longmapsto n\ipa{\bigotimes_{\mfp\in S}g\tau_\mfp}.
\end{equation}
It induces a $G(F)_+$-equivariant morphism
\[\tno{ev}:\mcA^{S\cup\iy}(\ddivv_S,\hat E(K_S)_{\varepsilon_S})(\varepsilon_{S})\longrightarrow \mcA^{S\cup\iy}(\Z[G(F)_+/T(F)_+],\hat E(K_S)_{\varepsilon_S})(\varepsilon_{S}).\]
Since we have (see \cite[Lemma 4.1]{guitart2017automorphic})
\[\mcA^{S\cup\iy}(\Z[G(F)_+/T(F)_+],\hat E(K_S)_{\varepsilon_S})(\varepsilon_{S})\simeq \tno{co}\Ind_{T(F)_+}^{G(F)_+}\ipa{\mcA^{S\cup\iy}(\hat E(K_S)_{\varepsilon_S})(\varepsilon_{S})},\]
we obtain by Shapiro's lemma 
\[\tno{ev}(\psi^S_\lambda)\in H^u_\ast(T(F)_+,\mcA^{S\cup\iy}(\hat E(K_S)_{\varepsilon_S})(\varepsilon_{S}))^\lambda.\]

Let us consider now the subspace of functions
\[
C(\mcG_T,\ovl \Q)^{\varepsilon_S}:=\{f\in C(\mcG_T,\ovl \Q),\;\rho^*f\mid_{T(F_S)}=\varepsilon_S\},
\]
where $\varepsilon_S:T(F_S)\ra\pm 1$ is now the product of the 
$\varepsilon_\mfp$. 
In this situation, the Artin map provides a decomposition:
\begin{equation}
C(\mcG_T,\ovl \Q)^{\varepsilon_S}=\bigoplus_{\lambda:T(F)/T(F)_+\ra\icla{\pm 1}}H^0\ipa{T(F)_+,C^0(T(\A_F^{S\cup\iy}),\ovl \Q)(\varepsilon_S)}.
\end{equation}
Given a locally constant character $\chi\in C(\mcG_T,\ovl \Z)^{\varepsilon_S}$ we can consider $\chi_\lambda\in H^0\ipa{T(F)_+,C^0(T(\A_F^{S\cup\iy}),\ovl \Z)(\varepsilon_S)}$ its $\lambda$-component.
Then we can define a twisted \tbf{plectic point}
\[P_\chi^S=\left(\eta^S\cap \chi_\lambda\right)\cap \tno{ev}(\psi^S_\lambda)\in \hat E(K_S)_{\varepsilon_S}\oti_\Z \ovl \Z,\]
where again the cap product is with respect to the pairing $\langle\cdot,\cdot\rangle_+$ of \eqref{equation.pairingplus} twisted by $\varepsilon_S$.

\subsection{Conjectures}\label{section:conj}

Analogously to the setting of Darmon points in \cite[Conjecture 4.3]{guitart2017automorphic}, it is conjectured that plectic points come from certain rational points of the elliptic curve $E$. Let us recall the corresponding conjectures that can be found in \cite[\S 1.4]{fornea2021plectic}.

Let $\chi\in C(\mcG_T,\ovl \Q)^{\varepsilon_S}$ be a locally constant character with $\chi\mid_{T(F_\infty)}=\lambda$ (namely, $\chi_\lambda\neq 0$). Let $H_\chi/K$ the abelian extension cut out by $\chi$. Since $\rho^*\chi\mid_{T(F_S)}=\varepsilon_S$, we can embed $H_\chi\subset H_{\varepsilon_\mfp}$, for all $\mfp\in S$. Moreover, if we consider 
\[
E(K)_\chi=\{P\in E(H_\chi)\otimes_\Z\bar\Z;\quad P^\sigma=\chi(\sigma)\cdot P\},
\]
clearly $E(K)_\chi\subset E(K_\mfp)_{\varepsilon_\mfp}$. Fornea and Gehrmann consider the natural morphism
\[
\imath_\mfp:E(K)_\chi\hookrightarrow E(K_\mfp)_{\varepsilon_\mfp}\otimes_\Z\bar\Z\hookrightarrow \hat E(K_\mfp)_{\varepsilon_\mfp}\otimes_\Z\bar\Z,
\]
and, if we write $r:=\#S$,
\[
\det:\bigwedge^r E(K)_\chi\longmapsto\hat E(K_S)_{\varepsilon_S};\qquad \det(P_1\wedge\cdots\wedge P_r):=\det\left(\begin{array}{ccc}\imath_{\mfp_1}(P_1)&\cdots&\imath_{\mfp_r}(P_1)\\&\cdots&\\\imath_{\mfp_1}(P_r)&\cdots&\imath_{\mfp_r}(P_r)\end{array}\right).
\]
The following conjecture is analogous to those of \cite[\S 1.4]{fornea2021plectic}.
\begin{conjecture}
We have that 
\begin{itemize}
\item \emph{Algebraicity and reciprocity law}: There exists $R_\chi\in \bigwedge^r E(K)_\chi$ such that
$\det R_\chi=P_\chi^S$.

\item \emph{Connection with BSD}: Assume that $\rank(E(K)_\chi)\geq r$. If
$P_\chi^S\neq 0$, we have that $\rank(E(K)_\chi)= r$.
\end{itemize}
\end{conjecture}

\section{Plectic points and anticyclotomic $p$-adic L-functions}

\subsection{Derivatives of characters}

This section is an adaptation of \cite[\S 3]{DasSp} to our setting.
Let $X$ be a totally disconnected topological Hausdorff space and $A$ a locally profinite group. We write 
\[
C^\diamond(X,A):=C^0(X,A)+\sum_\mcO C(X,\mcO)\subseteq C(X,A), \qquad C_c^\diamond(X,A):=C_c^0(X,A)+\sum_\mcO C_c(X,\mcO)\subseteq C_c(X,A),
\]
where the the sums are taken over all compact open subgroups $\mcO$ of $A$.

Let $S$ be set of places above $p$ and let $U^S\subset T(\A_F^{\iy\cup S})$ be an open compact subgroup.
Given a locally profinite ring $R$, let 
\[
\vartheta:T(\A_F)/U^ST(F_\iy)_+T(F)\longrightarrow R^\ti,
\]
be a continuous homomorphism. Let $I\subset R$ be a closed ideal such that 
\[
\bar\vartheta:T(\A_F)/U^ST(F_\iy)_+T(F)\stackrel{\vartheta}{\longrightarrow} R^\ti\longrightarrow \bar R^\ti,\qquad \bar R:=R/I,
\]
is locally constant. Assume also that the image of $\varepsilon_S:=\bar\vartheta\mid_{T(F_S)}:T(F_S)\rightarrow\bar R$ lies in $R$. 
Then, $\bar\vartheta^S:=\bar\vartheta\mid_{T(\A_F^{\iy\cup S})}$ defines an element
\[
\bar\vartheta^S\in H^0(T(F)_+,C^0(T(\A_F^{\iy\cup S}),\bar R)(\varepsilon_S^{-1})).
\]
For any $\mfp\in S$, write $\vartheta_\mfp$ for the restriction of $\vartheta$ to $T(F_\mfp)$, and let us consider the continuous map
\[
d\vartheta_\mfp:T(F_\mfp)\longrightarrow I/I^2; \qquad d\vartheta_\mfp(t_\mfp)=\vartheta_\mfp(t_\mfp)-\varepsilon_\mfp(t_\mfp)\;\mod\; I^2, 
\]
where $\varepsilon_\mfp:=\varepsilon_S\mid_{T(F_\mfp)}$.
If we write $\bar I:=I/I^2$, it is easy to check that $d\vartheta_\mfp$ defines an element
\[
d\vartheta_\mfp\in H^0(T(F_\mfp),\bar C^\diamond(T(F_\mfp),\bar I)(\varepsilon_\mfp)),\qquad \bar C^\diamond(T(F_\mfp),\bar I)(\varepsilon_\mfp):= C^\diamond(T(F_\mfp),\bar I)(\varepsilon_\mfp)/\ipa{\bar I\varepsilon_\mfp},
\]
where $\bar I\varepsilon_\mfp$ is the $\bar I$-module generated by the function $\varepsilon_\mfp$. Indeed, for any $\gamma\in T(F_\mfp)$
\begin{eqnarray*}
\gamma d\vartheta_\mfp(t_\mfp)&=&\varepsilon_\mfp(\gamma)\ipa{\vartheta_\mfp(\gamma^{-1}t_\mfp)-\varepsilon_\mfp(\gamma^{-1}t_\mfp)}=d\vartheta_\mfp(t_\mfp)+\vartheta_\mfp(t_\mfp)\ipa{\varepsilon_\mfp(\gamma)\vartheta_\mfp(\gamma)^{-1}-1}\\
&\equiv&d\vartheta_\mfp(t_\mfp)+\varepsilon_\mfp(t_\mfp)\ipa{\varepsilon_\mfp(\gamma)\vartheta_\mfp(\gamma)^{-1}-1}\mod I^2.
\end{eqnarray*}
Thus, $\gamma d\vartheta_\mfp\equiv d\vartheta_\mfp$ modulo $\bar I\varepsilon_\mfp$.

If $\mfp\in S$ does not split, we set $c_{\vartheta_\mfp}:={\rm res}_{T(F)_+}d\vartheta_\mfp\in H^0(T(F)_+,\bar C_c^\diamond(T(F_\mfp),\bar I)(\varepsilon_\mfp))$. In case $\mfp$ splits, 
we consider the canonical connection morphism introduced in \cite[Remark 3.3]{DasSp}:
\begin{equation}\label{conmorFfdiam}
H^0(F_\mfp^\ti,\bar C^\diamond(F_\mfp^\ti,\bar I)(\varepsilon_\mfp))\longrightarrow H^1(F_\mfp^\ti, C_c^\diamond(F_\mfp,\bar I)(\varepsilon_\mfp)).
\end{equation}
and we write $c_{\vartheta_\mfp}\in H^1(T(F)_+, C_c^\diamond(F_\mfp,\bar I)(\varepsilon_\mfp))$ for the restriction to $T(F)_+$ of the image of $d\vartheta_\mfp$ under such connection morphism.

We can regard $\vartheta$ as an element of $H^0(T(F),C(T(\A_F),R))$. 
Let us consider the natural cap-product with the fundamental class
$\vartheta\cap \eta\in H_u(T(F),C_c(T(\A_F),R))$,
where $C_c(T(\A_F),R)$ is the subspace of functions compactly supported when restricted to $T(\A_F^\infty)$. Since $\eta$ is also $T(F_\infty)_+$-invariant, we have in fact
\[
\vartheta\cap \eta\in H_u(T(F),C_c(T(\A_F),R)^{T(F_\infty)_+})=\bigoplus_{\bar\lambda:T(F)/T(F)_+\rightarrow\pm 1}H_u(T(F)_+,C_c(T(\A_F^{\infty}),R)).
\]
Moreover, it lies in the component corresponding to $\lambda:=\vartheta\mid_{T(F_\infty)}$. We denote by $\left(\vartheta\cap\eta\right)_\lambda\in H_u(T(F)_+,C_c(T(\A_F^{\infty}),R))$ the imatge at such component. Since it is $U^S$-invariant, it is easy to check that in fact
\[
\left(\vartheta\cap\eta\right)_\lambda\in H_u(T(F)_+,C_c^\diamond(T(\A_F^{\infty}),R)),\qquad C_c^\diamond(T(\A_F^{\infty}),R):=C^0_c(T(\A_F^{\iy\cup S}),R)\otimes_{R}\bigotimes_{\mfp\in S}  C_c^\diamond(T(F_\mfp),R)\subseteq C_c(T(\A_F^{\infty}),R).
\]

Similarly, if we write $S^1:=\{\mfp\in S,\;T\mbox{ splits in }\mfp\}$, $S^2=\{\mfq\in S,\;T\mbox{ non-split in }\mfq\}$, the cap-product:
\[
(\bar\vartheta^S\cap\eta^{S})\cap\ipa{\bigcup_{\mfp\in S}c_{\vartheta_\mfp}}\in H_u\ipa{T(F)_+,C^0_c(T(\A_F^{\iy\cup S}),\bar R)\otimes_{\bar R}\bigotimes_{\mfp\in S^1} C_c^\diamond(F_\mfp,\bar I)\oti_{\bar R}\bigotimes_{\mfp\in S^2}\bar C_c^\diamond(T(F_\mfp),\bar I)}.
\]
Since we have natural morphisms 
\begin{equation}\label{quot+extesc1}
C_c^\diamond(T(F_\mfp),R)\longrightarrow C_c^\diamond(F_\mfp,R),\quad\mfp\in S^1;\qquad  C_c^\diamond(T(F_\mfp),R)\longrightarrow \bar C_c^\diamond(T(F_\mfp),R):=C_c^\diamond(T(F_\mfp),R)/R,\quad\mfp\in S^2,
\end{equation}
given by extension by zero and quotient map, we have 
\begin{equation}\label{quot+extesc2}
\overline{(\cdot)}:C_c^\diamond(T(\A_F^{\infty}),R)
\longrightarrow C^0_c(T(\A_F^{\iy\cup S}),R)\otimes_{R}\bigotimes_{\mfp\in S^1} C_c^\diamond(F_\mfp,R)\oti_{R}\bigotimes_{\mfp\in S^2}\bar C_c^\diamond(T(F_\mfp),R)=:\bar C_c^\diamond(T(\A_F^{\infty}),R).
\end{equation}
If $r=\#S$ and we assume that $I^{r+1}=0$, we can consider the multilinear map given by multiplication
\[
m:\bar R\otimes_{\bar R}\bigotimes_{\mfp\in S}I/I^2\longrightarrow I^r/I^{r+1}=I^r\subseteq R,
\]
thus $m\ipa{(\bar\vartheta^S\cap\eta^{S})\cap\ipa{\bigcup_{\mfp\in S}c_{\vartheta_\mfp}}}$ can be seen as an element of $H_u\ipa{T(F)_+,\bar C_c^\diamond(T(\A_F^\infty),R)}$. The following result is analogous to that of \cite[Proposition 3.6]{DasSp}:
\begin{proposition}\label{propderchar}
Let $r=\# S$ and assume that $I^{r+1}=0$.
Then we have that 
\[
\#\ipa{\mcO_+^S/\mcO_+}_{\rm tor}\cdot\overline{\left(\vartheta\cap\eta\right)_\lambda}=m\ipa{(\bar\vartheta^S\cap\eta^{S})\cap\ipa{\bigcup_{\mfp\in S}c_{\vartheta_\mfp}}}\in H_u\ipa{T(F)_+,\bar C_c^\diamond(T(\A_F^\infty),R)}.
\]
\end{proposition}
\begin{proof}
Analogously to \cite[Proposition 3.6]{DasSp}, the result follows from the following commutative diagram:
\[
\xymatrix{
H^0(T(F)_+,\mcC_1)\ar[d]^{1}\ar[r]^{3}&H^0(T(F)_+,\mcC_4)\ar[d]^{5}&H^0(T(F)_+,\mcC_7)\ar[l]_{7}\ar[r]^{12}\ar[d]^{10}&H^0(T(F)_+,\mcC_{10})\ar[d]^{14}\\
H^{r}(T(F)_+,\mcC_2)\ar[r]^{4}\ar[d]^{2}&H^{r}(T(F)_+,\mcC_5)\ar[d]^{6}&H^{r}(T(F)_+,\mcC_8)\ar[l]_{8}\ar[r]^{13}\ar[d]^{11}&H^{r}(T(F)_+,\mcC_{11})\ar[d]^{15}\\
H_{u}(T(F_+),\mcC_3)\ar[r]^{\eqref{quot+extesc2}}&H_{u}(T(F_+),\mcC_6)&H_{u}(T(F_+),\mcC_9)\ar[l]_{9}\ar[r]^{=}&H_{u}(T(F_+),\mcC_{9})
}
\]
where the spaces $\mcC_i$ are 
\[
\begin{array}{ccc}
\mcC_1=\ipa{\bigotimes_{\mfp\in S}C^\diamond(T(F_\mfp),R)}\oti \mcC^S(R),&\mcC_2=\ipa{\bigotimes_{\mfp\in S}C_c^\diamond(T(F_\mfp),R)}\oti \mcC^S(R),& \mcC_3=C_c^\diamond(T(\A_F^{\iy}),R),\\
\mcC_4=\ipa{\bigotimes_{\mfp\in S}\bar C^\diamond(T(F_\mfp),R)}\oti \mcC^S(R),&\mcC_5=\ipa{\bigotimes_{\mfp\in S}\tilde C_c^\diamond(T(F_\mfp),R)}\oti \mcC^S(R),& \mcC_6=\bar C_c^\diamond(T(\A_F^{\iy}),R),\\
\mcC_7=\ipa{\bigotimes_{\mfp\in S}\bar C^\diamond(T(F_\mfp),I)}\oti \mcC^S(R),&\mcC_8=\ipa{\bigotimes_{\mfp\in S}\tilde C_c^\diamond(T(F_\mfp),I)}\oti \mcC^S(R),& \mcC_9=\bar C_c^\diamond(T(\A_F^{\iy}),I^r),\\
\mcC_{10}=\ipa{\bigotimes_{\mfp\in S}\bar C^\diamond(T(F_\mfp),\bar I)}\oti \mcC^S(\bar R),&\mcC_{11}=\ipa{\bigotimes_{\mfp\in S}\tilde C_c^\diamond(T(F_\mfp),\bar I)}\oti \mcC^S(\bar R),& 
\end{array}
\]
with $\mcC^S(\cdot)=C^0(T(\A_F^{\iy\cup S}),\cdot)$, $\mcC_c^S(\cdot)=C^0_c(T(\A_F^{\iy\cup S}),\cdot)$ and 
$\tilde C_c^\diamond(T(F_\mfp),\cdot)=\left\{\begin{array}{ll}
C_c^\diamond(F_\mfp,\cdot),&\mbox{ if }\mfp\in S^1,\\
\bar C_c^\diamond(T(F_\mfp),\cdot),&\mbox{ if }\mfp\in S^2;\end{array}\right.$. Morphisms 7, 8, 9 are induced by natural inclusions $I\subset R$, morphisms 12, 13 by the natural projections $I\to \bar I$, $R\to\bar R$, and morphisms 3, 4 by \eqref{quot+extesc1}. On the other hand, morphism 1 is provided by the cup-product $x\mapsto x\cup\bigcup_{\mfp\in S}{\rm res}_{T(F_\mfp)}^{T(F)_+}z_\mfp$, where $z_\mfp$ are the classes of Lemma \ref{lemma.onedotseven}, and morphisms 5, 10 and 11 are induced by connections morphisms \eqref{conmorFfdiam}. Finally, morphisms 2, 6, 11 and 15 are given by cap-products with $\eta^S$. Commutativity of the diagram follows from \cite[Lemma 3.4]{DasSp}.

On the one side, we can view $\vartheta$ as the element $\bigotimes_{\mfp\in S}\vartheta_\mfp\oti\vartheta^S$ of $H^0(T(F)_+,\mcC_1)$, $\vartheta^S=\vartheta\mid_{T(\A_F^{\iy\cup S})}$, and its image through the composition of 1, 2 and \eqref{quot+extesc2} is $\#\ipa{\mcO_+^S/\mcO_+}_{\rm tor}\overline{\left(\vartheta\cap\eta\right)_\lambda}$ by Lemma \ref{lemma.onedotseven}. On the other side, if we consider $\phi_\mfp:=\vartheta_\mfp-\varepsilon_\mfp\in \bar C^\diamond(T(F_\mfp),I)$, we have that $\tilde\vartheta:=\bigotimes_{\mfp\in S}\phi_\mfp\oti\vartheta^S\in \mcC_7$ is $T(F)_+$-invariant. Indeed, for any $x,t\in T(F_\mfp)$,
\[
(t\ast \phi_\mfp)(x)=\phi_\mfp(t\inv x)=\vartheta_\mfp(t\inv x)-\varepsilon_\mfp(t\inv x)=\vartheta_{\mfp}(t)\inv\phi_\mfp(x)+\ipa{\vartheta_{\mfp}(t\inv)-\varepsilon_\mfp(t\inv)}\varepsilon_\mfp(x)\equiv \vartheta_{\mfp}(t)\inv\phi_\mfp(x)\mod I\varepsilon_\mfp.
\]
Hence, for all $\gamma\in T(F)_+$,
\[
\gamma\ast \tilde\vartheta=\bigotimes_{\mfp\in S}\gamma\ast \phi_\mfp\oti \gamma\ast \vartheta^S=\vartheta^S(\gamma)\inv\prod_{\mfp\in S}\vartheta_\mfp(\gamma)\inv \ipa{\phi_\mfp\oti \vartheta^S}=\vartheta(\gamma)^{-1}\cdot \tilde\vartheta=\tilde\vartheta.
\]
Moreover, it is clear that the image of $\tilde \vartheta$ through 7 coincides with the image of $\vartheta$ through 3. Thus, the result follows from the commutativity of the diagram, \cite[Lemma 3.4]{DasSp}, and the fact that $d\vartheta_\mfp$ is the projection of $\phi_\mfp$.
\end{proof}

\subsection{Exceptional zero formulas}

Let $S$ be a set of primes above $p$ and let us fix an isomorphism  $G(F_S)\simeq \PGL_2(F_S)$. Assume that, for any $\mfp\in S$, the representation $V_\mfp$ is ordinary and either spherical or (twisted) Steinberg.
Let $\mu_{\phi_\lambda^S}$ be the measure of $\mcG_T$ constructed in \S \ref{antipLfunct}. We can decompose $S=S_+\cup S_-$, where 
\[
S_+=\{\mfp\in S;\;V_\mfp^{\Z_p}=\tno{St}_{\Z_p}(F_\mfp)(\varepsilon_\mfp)\}.
\] 
Let ${\rm Meas}(\mcG_T,\C_p)$ be the space of measures of $\mcG_T$ endowed with the group law given by convolution (see \cite[\S 2]{blanco2015anticyclotomic}). For any finite character $\xi\in C^0(\mcG_T,\C_p)$, write
\[
I_\xi=\ker(\varphi_\xi);\qquad \varphi_\xi(\mu)=\int_{\mcG_T}\xi d\mu.
\]
If we write $r:=\#S_+$ then remark \ref{remark:excepzeros} makes us think that $\mu_{\phi_\lambda^S}\in I_\xi^{r}$, for all $\xi\in C^0(\mcG_T,\bar\Z)^{\varepsilon_{S_+}}$ with $\xi\mid_{T(F_\iy)}=\lambda$. Moreover, Mazur-Tate interpretation realizes the image of $\mu_{\phi_\lambda^S}$ in $I_\xi^{r}/I_\xi^{r+1}$ as the $r$th derivative of the corresponding $p$-adic L-function. Hence, following the philosophy of the $p$-adic BSD conjecture, such image must be related to the extended Mordell-Weil group of $E$. 
In \cite[Theorem A]{fornea2021plectic}, Fornea and Gehrmann prove a similar result with plectic points, when $S=S_+$ and $T$ is inert at any $\mfp\in S$. Our aim in this section is to establish the general result for arbitrary $S$, $V_S$ and $T$.

Write $S_+=S_+^1\cup S_+^2$, where 
\[
S_+^1:=\{\mfp\in S_+,\;T\mbox{ splits in }\mfp\},\qquad S_+^2=\{\mfq\in S_+,\;T\mbox{ does not split in }\mfq\}.
\]
Notice that there exists a $c_\xi$-admissible Eichler order $\mcO_{N,\mfq}$ for any $\mfq\in S_-$. 
Let us choose a non-zero $\mcO_{N,\mfq}^\times$-invariant $x_\mfq\in V_\mfq$. 
For any $\mfp\in S_+^1$, we choose $x_\mfp\in V_\mfp$ to be the image of ${\rm vol}(\mfo_{F_\mfp}^\ti)^{-1}\indi_{\mfo_{F,\mfp}}\in C^0(\PP^1(F_\mfp),\Q)$, once we identify $T(F_\mfp)\simeq F_\mfp^\ti$ as a subset of $\PP^1(F_\mfp)$ by means of $t\mapsto \imath(t)\ast\iy$. Thus, we can write 
\[
x^{S_+^2}=x_{S_-}\oti x_{S_+^1}=\bigotimes_{\mfq\in S_-}x_\mfq\oti \bigotimes_{\mfp\in S_+^1}x_\mfp\in V_{S_-}\oti V_{S_+^1},\qquad V_{S_-}=\prod_{\mfq\in S_-}V_\mfq,\quad V_{S_+^1}=\prod_{\mfp\in S_+^1}V_\mfp.
\]
Notice that, for any $v^{S_+^2}\in V_{S_-}\oti V_{S_+^1}$, we have a $G(F)_+$-equivariant evaluation map analogous to \eqref{eqevaluationmap1}
\begin{equation}\label{eqevaluationmap}
    \cdot (v^{S_+^2}):\mcA^{S\cup\iy}\ipa{V_S,\Q}\longrightarrow \mcA^{S_+^2\cup\iy}\ipa{\bigotimes_{\mfp\in S_+^2}V_\mfp,\Q}.
\end{equation}
As explained in \S \ref{section:autcla}, a multiple of $\phi_\lambda^S(x^{S_+^2})$ actually lies in $H_\ast^u\ipa{G(F)_+,\mcA^{S_+^2\cup\iy}\ipa{\bigotimes_{\mfp\in S_+^2}\tno{St}_{\Z_p}(F_\mfp)(\varepsilon_\mfp),\Z_p}}^\lambda$, hence, the construction of \S \ref{section:fornea-gehrmann} provides a plectic point
 $P_\xi^{S_+^2}\in \hat E(K_S)_{\varepsilon_{S_+^2}}\oti_\Q\ovl\Q$ associated with $\phi_\lambda^S(x^{S_+^2})$ and $\xi$. 
 Notice that, for $\mfp\in S_+^2$, there exists an automorphism $\sigma_\mfp\in \Gal(H_\mfp/F_\mfp)$ that coincides on $E(K_\mfp)_{\varepsilon_\mfp}\simeq K_\mfp^\ti/q_{E_\mfp}^\Z$ with the non-trivial automorphism of $K_\mfp$. Hence, the following \emph{plectic invariant} can be seen as an element of $\hat T(F_{S_+^2}):=\bigotimes_{\mfp\in S_+^2}\widehat{T(F_\mfp)}$ 
\[
Q_\xi^{S_+^2}:=\ipa{\prod_{\mfp\in S_+^2}(\sigma_\mfp-1)}P_\xi^{S_+^2}\in \hat T(F_{S_+^2}),
\]
where again $\widehat{(\cdot)}$ stands for the $p$-adic completion of the non-torsion part.
In case $S_+^2=\emptyset$, we will assume that, with the notation of theorem \ref{THMintprop}, 
\[
Q_\xi^{\emptyset}:=\ipa{\frac{\varepsilon({\Sigma^D})2^{\#\Sigma_D}L_{c_\xi}(1,\psi_{K})^2h^2
}{\rho^\ast\xi^{-1}(k_{0}^{S_+})|c_\xi^2 D|^{\frac{1}{2}}C}\cdot L^{\Sigma^{S_+}}(1/2,\Pi,\chi)\cdot\frac{\langle \Phi_{\alpha_{S_+}},\Phi_{\alpha_{S_+}}\rangle}{\langle \Psi,\Psi\rangle}\cdot\frac{{\rm vol}(U_0(N))}{{\rm vol}(U_N)}\prod_{\mfp\in S_+}\frac{\varepsilon_\mfp\zeta_\mfp(1)^3}{\zeta_\mfp(2)}}^{\frac{1}{2}}
\] 
being $\Sigma^D=\{v\mid N:\;v\nmid c_\xi D\}$, $\Sigma_D=\{v\mid(N,D);\;v\nmid c_\xi\}$ and $\Phi_{\alpha_{S_+}}\in M_{\udl 2}(U_{NS_+})$ is such that ${\rm ES}_\lambda(\Phi_{\alpha_{S_+}})=\phi_\lambda^{S}(x^\emptyset)$.

We introduce the following well-defined morphism 
\begin{equation}\label{varphiGI}
\begin{tikzcd}
\varphi: &[-3em]\mcG_T \ar[r] & I_\xi/I_\xi^2; & \int_{\mcG_T}f d\varphi(\sigma) =\xi(\sigma)\inv\cdot f(\sigma)-f(1), & \sigma\in \mcG_T.
\end{tikzcd}
\end{equation}
Moreover, the Artin map induces a morphism
$\tno{rec}_\mfp: T(F_\mfp)\rightarrow \mcG_T$.
The product of morphisms $\varphi\circ\tno{rec}_\mfp$ provides
\begin{equation}
\begin{tikzcd}
\varphi\circ\tno{rec}_{S_+}: &[-3em]\hat T(F_{S_+})\oti_\Q\bar\Q \ar[r] & I^{r}_\xi/I_\xi^{r+1}.
\end{tikzcd}
\end{equation}

For every $\mfq\in S_-$ we have $V_\mfq=\Ind_P^G\chi_\mfq$, with $\chi_\mfq$ unramified. Let $J_{\mfq}$ be the unique non-trivial element in $P\subset G(F_\mfq)$ that also lies in the normalizer of $T(F_\mfp)$. Moreover, we consider the $\mfq$-stabilized elements $f_{0,\mfq}$ of \eqref{deff0}. As discussed in \S \ref{section:IntProp}, there exists $k_{0,\mfq}\in T(F_\mfq)$ such that $k_{0,\mfq}^{-1}J_\mfq\in\mcO_{N,\mfq}$. We recall the constants $C(\pi_\mfp)$ and $C(\pi_\mfp,\xi_\mfp)$ from theorem \ref{THMintprop} and define
\[
0\neq \epsilon_{S_-}(\pi,\xi)=\prod_{\mfq\in S_-}\frac{\langle f_{0,\mfq},f_{0,\mfq}\rangle_\mfp}{\langle v_{0,\mfp},v_{0,\mfp}\rangle_\mfp}\frac{\xi_\mfq^{-1}(k_{0,\mfq})}{{\rm vol}(\mfo_{K_\mfq}^\ti/\mfo_{F_\mfq}^\ti)^2}\frac{C(\pi_\mfq,\xi_\mfq)C(\pi_\mfq)}{L(1/2,\Pi_\mfq,\xi_\mfq)}\cdot\left\{\begin{array}{ll}
        1,&n_{\xi_\mfq}= 0;\\
        L(1,\psi_{K_\mfp})^{-2},&n_{\xi_\mfq}> 0.
    \end{array}\right.
\] 
\begin{theorem}\label{mainTHM2}
Let $\xi\in C^0(\mcG_T,\bar\Z)^{\varepsilon_{S_+}}$ be a finite character of conductor $c$ outside $S$ such that $\xi\mid_{T(F_\iy)}=\lambda$.
Then we have that $\mu_{\phi^S_\lambda}\in I_\xi^r$ and
the image of $\mu_{\phi^S_\lambda}$ in $I_\xi^r/I_\xi^{r+1}$ is given by
\[\mu_{\phi^S_\lambda}\equiv (-1)^s\cdot[\mcO_+^{S_+^2}:\mcO_+]^{-1}\cdot  \epsilon_{S_-}(\pi,\xi)^{\frac{1}{2}}\cdot\varphi\circ\tno{rec}_{S_+}\ipa{q_{S_+^1}\otimes Q_\xi^{S_+^2}}\mod I_\xi^{r+1},\]
where 
$q_{S_+^1}=\bigotimes_{\mfp\in S_+^1}q_{E_\mfp}\in \hat T(F_{S_+^1})$ is the product of Tate periods and $s=\#S_+^1$.
\end{theorem}
\bpf

Assume that $S=S_+$. For any $\mfp\in S_+$, any group $M$ and any finite rank $\Z_p$-module $N$, we can consider the diagram  
\begin{center}
\begin{tikzcd}
  H^m_\ast(G(F)_+,\mcA^{S_+\cup\iy}(M\otimes\tno{St}_{\Z_p}(F_\mfp)(\varepsilon_\mfp),N))^\lambda \ar[d,"i_\mfp"]\ar[rd,"c^m_\mfp"]& \\
H^m_\ast(G(F)_+,\mcA^{S_+\cup\iy}(M\otimes_\Z\Delta_\mfp^0,\hat K_\mfp^\ti\otimes_{\Z_p}N)(\varepsilon_\mfp))^\lambda \ar[r] & H^{m+1}_\ast(G(F)_+,\mcA^{S_+\cup\iy}(M,\hat K_\mfp^\ti\otimes_{\Z_p}N)(\varepsilon_\mfp))^\lambda,\\[-2em]
\end{tikzcd}
\end{center}
where the vertical arrow arises from \eqref{definition.multiplicativeintegral2} and the horizontal arrow is the connection morphism of the degree long exact sequence. We write $i_{S_+}$ for the composition of $i_\mfp$, for all $\mfp\in S_+^2$, and $c_{\mfp_1}^{u}$, $c_{\mfp_2}^{u+1}$, $\cdots$, $c_{\mfp_s}^{u+s-1}$, where $S_+^1=\{\mfp_1,\cdots,\mfp_s\}$. Thus,
\[
i_{S_+}\phi_\lambda^{S_+}\in 
H^{u+s}_\ast(G(F)_+,\mcA^{S_+\cup\iy}(\Delta^0_{S_+^2},\hat K_{S_+^2}^\ti\otimes\hat F_{S_+^1}^\ti)(\varepsilon_{S_+}))^\lambda.
\]
Let us consider the $G(F)$-equivariant morphism
\[
\begin{tikzcd}
\tno{EV}:&[-3em] \mcA^{S_+\cup\iy}(\Delta_{S_+^2}^0,\hat K_{S_+^2}^\ti\otimes\hat F_{S_+^1}^\ti)(\lambda\cdot \varepsilon_{S_+}) \ar[r] & \tno{co}\Ind_{T(F)}^{G(F)}\ipa{\mcA^{S_+\cup\iy}(\hat K_{S_+^2}^\ti\otimes\hat F_{S_+^1}^\ti)(\lambda\cdot\varepsilon_{S_+})} \\[-2em]
& f \ar[r,mapsto] &  (h,g^{S_+})\mapsto \varepsilon_{S_+}(h)\cdot\lambda(h)\cdot f(h\inv g^{S_+})\ipa{\bigotimes_{\mfp\in S_+^2}\ipa{h\inv\bar\tau_\mfp-h\inv\tau_\mfp}}, \\[-2em]
\end{tikzcd}
\]
where $(h,g^{S_+})\in G(F)\ti G(\A^{S_+\cup\iy})$. Since we are evaluating at divisors of the form $h\inv\bar\tau_\mfp-h\inv\tau_\mfp$ it is easy to show that the image of $\phi_\lambda^{S_+}$ lies in $\hat T(F_{S_+})$. Thus, by Shapiro,
\[
\tno{EV}i_{S_+}\phi_\lambda^{S_+}\in H^{u+s}_\ast(T(F)_+,\mcA^{S_+\cup\iy}(\hat T(F_{S_+}))(\varepsilon_{S_+}))^\lambda.
\]
We claim that
\begin{equation}\label{claimTHM}
    q_{S_+^1}\otimes Q_\xi^{S_+^2}= \#\ipa{\mcO_+^{S_+}/\mcO_+^{S_+^2}}_{\tno{tor}}^{-1}\cdot(\eta^{S_+}\cap\xi_\lambda)\cap{\rm EV}i_{S_+}\phi_\lambda^{S_+}\in \hat T(F_{S_+})\otimes_\Z\ovl\Z.
\end{equation}
Indeed, by functoriality and the description of the multiplicative integral given in Remark \ref{extSt}, we have that 
\[{\rm EV}i_{S_+}\phi_\lambda^{S_+}={\rm EV}i_{S_+^2}\varphi_{\tno{unv}}\phi_\lambda^{S_+}\cup\bigcup_{\mfp\in S_+^1}\tno{res}_{T(F)_+}^{G(F)_+}c_\mfp,
\]
where $\varphi_{\tno{unv}}$ is defined in \eqref{morextSt}, $c_\mfp\in H^1(G(F),\tno{St}_{F_\mfp^\ti})$ is the class associated with the extension $\mcE_{F_\mfp^\ti}$ of \eqref{defextSt}, and 
\[
\varphi_{\tno{unv}}\phi_\lambda^{S^+}\in H^u_\ast(G(F)_+,\mcA^{S_+\cup\iy}(\tno{St}_{\Z_p}(F_{S_+^2})\otimes \tno{St}_{F_{S_+^1}^\ti},\hat F_{S_+^1}^\ti(\varepsilon_{S_+})))^\lambda,\qquad \tno{St}_{ F_{S_+^1}^\ti}:=\bigotimes_{\mfp\in S_+^1}\tno{St}_{F_{\mfp}^\ti};\quad \tno{St}_{\Z_p}(F_{S_+^2}):=\bigotimes_{\mfp\in S_+^2}\tno{St}_{\Z_p}(F_{\mfp}),
\]
is the corresponding push-forward. In fact, Oda's conjecture (\cite[Conjecture 3.8]{guitart2017automorphic}) can be interpreted as
\[
\varphi_{\tno{unv}}\phi_\lambda^{S_+}\cup \bigcup_{\mfp\in S_+^1}c_\mfp=q_{S_+^1}\otimes \ipa{\phi_\lambda^{S_+}\cup \bigcup_{\mfp\in S_+^1}c^{\tno{ord}}_\mfp},
\]
where $c_\mfp^{\tno{ord}}\in H^1(G(F),\tno{St}_{\Z_p}(F_\mfp))$ is the class associated with the following extension (see \cite[\S 6.1]{blanco2015anticyclotomic})
\[
\tno{St}_{\Z_p}(F_\mfp)\stackrel{f\mapsto (f(g^{-1}\ast\iy),0)}{\hookrightarrow}\mcE_{\Z_p}(F_\mfp):=\left\{(\phi,y)\in C(\GL_2(F_\mfp),{\Z_p})\ti{\Z_p}:\phi\left(\bbm s & x \\  & t \ebm g\right)=y\cdot v_\mfp(t)+\phi(g) \right\}/({\Z_p},0).
\]
For $\mfp\in S_+^1$, let $\tau_\mfp$ and $\bar\tau_\mfp$ be the points in $\PP^1(F_\mfp)$ fixed by $T(F_\mfp)$ as in \S \ref{deltap}. Thus, if we write  
\begin{equation}\label{phi1}
\phi_1\bbm a & b \\ c & d \ebm =v_\mfp(d+\bar\tau_\mfp\cdot c)-\ipa{ (v_\mfp+1)\cdot\indi_{\mfo_{F_\mfp}}}\ipa{\frac{d+\bar\tau_\mfp\cdot c}{d+\tau_\mfp\cdot c}},
\end{equation}
the element $(\phi_1,1)\in\mcE_{\Z_p}(F_\mfp)$ is a generator mapping to $1$ under the natural $G(F)$-morphism $\mcE_{\Z_p}(F_\mfp)\ra {\Z_p}$. Hence the cocycle $\tno{res}_{T(F)}^{G(F)}c^{\tno{ord}}_\mfp$ has representative
\[
c^{\tno{ord}}_\mfp(t)(x)=(1-t)\phi_1(x)=(1-t)\ipa{ (v_\mfp+1)\cdot\indi_{\mfo_{F_\mfp}}}\ipa{\frac{x-\bar\tau_\mfp}{x-\tau_\mfp}}=\ipa{\int_{T(F_\mfp)}\indi_{\mfo_{F_\mfp}}(\tau)\frac{\ipa{\tau \indi_{\mfo_{F_\mfp}}-t\tau\indi_{\mfo_{F_\mfp}}}}{{\rm vol}({\mfo_{F_\mfp}^\ti})}d^\ti\tau}\ipa{\frac{x-\bar\tau_\mfp}{x-\tau_\mfp}}, 
\]
for all $t\in T(F)$, since ${\rm vol}({\mfo_{F_\mfp}^\ti})^{-1}\int_{\mfo_{F_\mfp}}\tau\indi_{\mfo_{F_\mfp}} d^\ti\tau=(v_\mfp+1)\indi_{\mfo_{F_\mfp}}$. 
Using lemma \ref{lemma.onedotseven} and the cocycle $z_\mfp$ defined there,
\begin{eqnarray*}
(\eta^{S_+}\cap\xi_\lambda)\cap{\rm EV}i_{S_+}\phi_\lambda^{S_+}&=&q_{S_+^1}\otimes(\eta^{S_+}\cap\xi_\lambda)\cap\ipa{{\rm EV}i_{S_+^2}\phi_\lambda^{S_+}\cup\bigcup_{\mfp\in S_+^1}\tno{res}_{T(F)_+}^{G(F)_+}c^{\tno{ord}}_\mfp}\\
&=&q_{S_+^1}\otimes\ipa{\eta^{S_+}\cap\bigcup_{\mfp\in S_+^1}\tno{res}_{T(F)_+}^{T(F_\mfp)}z_\mfp\cap\xi_\lambda}\cap{\rm EV}i_{S_+^2}\phi_\lambda^{S_+}(x_{S_+^1})\\
&=&\#\ipa{\mcO_+^{S_+}/\mcO_+^{S_+^2}}_{\tno{tor}}\cdot q_{S_+^1}\otimes\ipa{\eta^{S_+^2}\cap\xi_\lambda}\cap{\rm EV}i_{S_+^2}\phi_\lambda^{S_+}(x_{S_+^1}),
\end{eqnarray*}
where 
the second from the fact that, for any $\Phi\in \mcA^{\{\mfp\}\cup S\cup\iy}(\tno{St}_{\Z_p}(F_\mfp),M)$,
\begin{eqnarray*}
\Phi(c^{\tno{ord}}_\mfp(t))&=&\int_{\mfo_{F_\mfp}}\frac{\Phi\ipa{\tau \indi_{\mfo_{F_\mfp}}-t\tau\indi_{\mfo_{F_\mfp}}}}{{\rm vol}({\mfo_{F_\mfp}^\ti})}d^\ti\tau=\int_{T(F_\mfp)}\frac{\ipa{\indi_{\mfo_{F_\mfp}}(\tau)-t\indi_{\mfo_{F_\mfp}}(\tau)}}{{\rm vol}({\mfo_{F_\mfp}^\ti})}\cdot\Phi\ipa{\tau \indi_{\mfo_{F_\mfp}}}d^\ti\tau\\
&=&\int_{T(F_\mfp)}z_\mfp(t)(\tau)\cdot\Phi\ipa{\frac{\indi_{\mfo_{F_\mfp}}}{{\rm vol}({\mfo_{F_\mfp}^\ti})}}(\tau)d^\ti\tau.
\end{eqnarray*}
Similarly as in the proof of theorem \ref{THMintprop}, if $S_+^2=\emptyset$, the computations of \S \ref{normtestvec} and \S \ref{disting} imply that $\ipa{\eta\cap\xi_\lambda}\cap\phi^{S_+}_\lambda(x_{S_+})=Q_\xi^{\emptyset}$,
proving the claim \eqref{claimTHM} in this case.
If $S_+^2\neq \emptyset$, due to the fact that ${\rm EV}$ is defined by means of the evaluation at the divisor $\bar\tau_\mfp-\tau_\mfp\in \Delta_\mfp^0$, we obtain that  $(\eta^{S_+^2}\cap\xi_\lambda)\cap{\rm EV}i_{S_+^2}\phi_\lambda^{S_+^2}=Q_\xi^{S_+^2}$ as well, and our claim \eqref{claimTHM} follows in any of the settings.

The space ${\rm Meas}(\mcG_T,\C_p)$ is the inductive limit of ${\rm Meas}(\mcG_T^{U^{S_+}},\C_p)$, where $U^{S_+}\subset T(\A_{F}^{S_+\cup\iy})$ runs over the set of open compact subgroups and $\mcG_T^{U^{S_+}}$ is the corresponding $U^{S_+}$-invariant Galois group. Thus, it is enough to prove the statement for the images $\mu_{\phi_\lambda^\mfp}^{U^{S_+}}$ of $\mu_{\phi_\lambda^{S_+}}$ in ${\rm Meas}(\mcG_T^{U^{S_+}},\C_p)$, for all $U^{S_+}$ small enough so that $\xi$ is $U^{S_+}$-invariant. 

Write $R:={\rm Meas}(\mcG_T^{U^{S_+}},\C_p)$, and
let us consider the character
\[
\vartheta: T(\A_F)/T(F_\infty)_+U^{S_+}T(F)\longrightarrow R/I_\xi^{r+1}, \qquad \int_{\mcG_T}f d\vartheta(t) = f(\rho(t))\quad{\rm mod}\; I_\xi^{r+1},
\]
where, by abuse of notation, we also denote by $I_\xi$ the image of $I_\xi$ in $R$. Notice that reduction modulo $I_\xi$ provides $\overline\vartheta=\rho^\ast\xi$ under the isomorphism $R/I_\xi\simeq \C_p$.

Write the morphism $\ell_{S_+}:=\varphi\circ\tno{rec}_{S_+}$. Hence, by Equation \eqref{claimTHM},
\[
    \ell_{S_+}\ipa{q_{S_+^1}\otimes Q_\xi^{S_+^2}}=\frac{1}{\#\ipa{\mcO_+^{S_+}/\mcO_+^{S_+^2}}_{\tno{tor}}}\cdot(\eta^{S_+}\cap\xi_\lambda)\cap\left(\ell_{S_+^2}{\rm EV}i_{S_+^2}\varphi_{\tno{unv}}\phi_\lambda^{S_+}\cup\bigcup_{\mfp\in S_+^1}\ell_\mfp\tno{res}_{T(F)_+}^{G(F)_+}c_\mfp,\right)\in I_\xi^r/I_\xi^{r+1},
\]
where $\ell_\mfp=\varphi\circ\tno{rec}_{\mfp}$.

On the one hand, if $\mfp\in S_+^1$ we define similarly as in \eqref{phi1}  
\begin{equation}\label{phi2}
\bar\phi_1\bbm a & b \\ c & d \ebm =(d+\bar\tau_\mfp\cdot c)\cdot\ipa{\frac{d+\tau_\mfp\cdot c}{d+\bar\tau_\mfp\cdot c} }^{\indi_{\mfo_{F,\mfp}}\ipa{\frac{d+\bar\tau_\mfp\cdot c}{d+\tau_\mfp\cdot c}}},
\end{equation}
and the element $(\bar\phi_1,1)\in\mcE_{F_\mfp^\ti}$ is a generator mapping to $1$ under the natural $G(F)$-morphism $\mcE_{F_\mfp^\ti}\ra \Z$. Hence the cocycle $\tno{res}_{T(F)_+}^{G(F)_+}c_\mfp$ has representative
\[
c_\mfp(t)(x)=(1-t)\bar\phi_1(x)=(1-t)\ipa{ \ipa{\frac{x-\tau_\mfp}{x-\bar\tau_\mfp}}^{\indi_{\mfo_{F,\mfp}}\ipa{\frac{x-\bar\tau_\mfp}{x-\tau_\mfp}}}}=t^{-\indi_{t\mfo_{F,\mfp}}\ipa{\frac{x-\bar\tau_\mfp}{x-\tau_\mfp}}}\cdot \ipa{\frac{x-\tau_\mfp}{x-\bar\tau_\mfp}}^{z_\mfp(t)\ipa{\frac{x-\bar\tau_\mfp}{x-\tau_\mfp}}}\in F_\mfp^\ti.
\]
If we apply $\ell_{\mfp}$, we obtain the cocycle
\[
\ell_\mfp c_\mfp(t)\ipa{\frac{x-\bar\tau_\mfp}{x-\tau_\mfp}}=-\ell_\mfp(t)\cdot\indi_{t\mfo_{F,\mfp}}\ipa{\frac{x-\bar\tau_\mfp}{x-\tau_\mfp}}- \ell_\mfp\ipa{\frac{x-\bar\tau_\mfp}{x-\tau_\mfp}}{z_\mfp(t)\ipa{\frac{x-\bar\tau_\mfp}{x-\tau_\mfp}}}.
\]
Thus, as one can see in \cite[\S 3.2]{DasSp}, we obtain that 
\begin{equation}\label{cpS+1}
\ell_\mfp\tno{res}_{T(F)_+}^{G(F)_+}c_\mfp=-\delta_\mfp c_{\vartheta_\mfp}(\varepsilon_\mfp)\in H^1(T(F)_+, C(\PP^1(F_\mfp),I_{\xi}/I_\xi^2)/(I_\xi/I_\xi^2)),
\end{equation}
where $\delta_\mfp:C_c^\diamond(F_\mfp,I_\xi/I_\xi^2)\rightarrow {\rm St}_{I_\xi/I_\xi^2}:=C(\PP^1(F_\mfp),I_{\xi}/I_\xi^2)/(I_\xi/I_\xi^2)$ is the morphism induced by that of lemma \ref{lemma:mordelta} and $(\varepsilon_\mfp)$ stands for the twist provided by $\varepsilon_\mfp$.

On the other hand, since $\ell_\mfp$ is a group homomorphism, by \eqref{definition.multiplicativeintegral2},
\begin{equation}\label{cpS+2}
\ell_{S_+^2}{\rm EV}i_{S_+^2}\varphi_{\tno{unv}}\phi=\lim_{\mcU_\mfp\in\tno{Cov}(\PP^1(F_\mfp))}\sum_{U_\mfp\in\mcU_\mfp} \ipa{\bigotimes_{\mfp\in S_+^2}\ell_\mfp\ipa{\frac{x_{U_\mfp}-\ovl\tau_\mfp}{x_{U_\mfp}-\tau_\mfp}}}\cdot{\phi\ipa{\bigotimes_{\mfp\in S_+^2}\indi_{U_\mfp}}}=\phi\left( \bigotimes_{\mfp\in S_+^2}\delta_\mfp c_{\vartheta_\mfp}(\varepsilon_\mfp)\right)
\end{equation}
for all $\phi\in\Hom\ipa{\bigotimes_{\mfp\in S_+^2}\tno{St}_{\Z_p}(F_\mfp),{\Z_p}}$, where \[
\delta_\mfp:\bar C^\diamond(T(F_\mfp),I_\xi/I_\xi^2)\longrightarrow C^\diamond(T(F_\mfp),I_\xi/I_\xi^2)/(I_\xi/I_\xi^2)\longrightarrow {\rm St}_{I_\xi/I_\xi^2}:=C(\PP^1(F_\mfp),I_{\xi}/I_\xi^2)/(I_\xi/I_\xi^2)
\]
is again the isomorphism induced by that of Lemma \ref{lemma:mordelta}. 

Using relations \eqref{cpS+1}, \eqref{cpS+2} and Proposition \ref{propderchar}, one obtains that
\[
\ell_{S_+}\ipa{q_{S_+^1}\otimes Q_\xi^{S_+^2}}=\frac{(-1)^{s}}{\#\ipa{\mcO_+^{S_+}/\mcO_+^{S_+^2}}_{\tno{tor}}}\cdot(\eta^{S_+}\cap\bar\vartheta_\lambda)\cap\left(\delta_{S_+}^\ast\phi_\lambda^{S_+}\cup\bigcup_{\mfp\in S_+}c_{\vartheta_\mfp}\right)=\frac{(-1)^{s}\#\ipa{\mcO_+^{S_+}/\mcO_+}_{\rm tor}}{\#\ipa{\mcO_+^{S_+}/\mcO_+^{S_+^2}}_{\tno{tor}}}\cdot \overline{\left(\vartheta\cap\eta\right)_\lambda}\cap \delta_{S_+}^\ast\phi_\lambda^{S_+}.
\]
By Equation \eqref{equation:capprod}, we have that $\overline{\left(\vartheta\cap\eta\right)_\lambda}\cap \delta_{S_+}^\ast\phi_\lambda^{S_+}=\left(\vartheta\cap\eta\right)\cap \delta_{S_+}^\ast\phi_\lambda^{S_+}$. Moreover, by definition  the image of $\mu_{\phi_\lambda^{S_+}}$ in $R/I_\chi^{r+1}$ coincides with $\left(\vartheta\cap\eta\right)\cap \delta_{S_+}^\ast\phi_\lambda^{S_+}$, hence the result follows for $S=S_+$.

Assume now that $S\neq S_+$, hence $S_-\neq\emptyset$. 
Recall that the image of $\mu_{\phi_\lambda^S}$ in $R/I_\xi^{r+1}$, where $R:={\rm Meas}(\mcG_T^{U^{S}},\C_p)$, coincides with
\[
\mu_{\phi_\lambda^S}\;({\rm mod}\; I_\xi^{r+1})=\overline{\left(\vartheta\cap\eta\right)_\lambda}\cap \delta_{S}^\ast\phi_\lambda^{S}=\#\ipa{\mcO_+^{S_+}/\mcO_+}_{\rm tor}^{-1}(\eta^{S_+}\cap\xi_\lambda)\cap\left(\delta_{S}^\ast\phi_\lambda^{S}\cup\bigcup_{\mfp\in S_+}c_{\vartheta_\mfp}\right),
\]
where the last cup-product is provided by $\langle\;,\;\rangle_+$ of \eqref{equation.pairingplus}, once we identify
\[
\eta^{S_+}\cap\xi_\lambda\in H_{u}(T(F)_+,C_c^0(T(\A_F^{S_+\cup\iy}),\ovl\Q)(\varepsilon_{S_+}))=H_{u}(T(F)_+,C_c^0(T(\A_F^{S\cup\iy}),C_c^0(T(F_{S_-}),\ovl\Q))(\varepsilon_{S_+})).
\]
Notice that, given $f^S\oti f_{S_-}\in C_c^0(T(\A_F^{S\cup\iy}),C_c^0(T(F_{S_-}),\ovl\Q))=C_c^0(T(\A_F^{S\cup\iy}),\ovl\Q)\oti C_c^0(T(F_{S_-}),\ovl\Q)$, an open compact subgroup $H\subseteq T(F_{S_-})$ such that $f_{S_-}$ is  $H$-invariant,  and $\phi\in \mcA^{S\cup\iy}(V_{S_-},I_{\xi}^r/I_{\xi}^{r+1})$, we have
\begin{eqnarray*}
\langle f^S\oti f_{S_-},\delta_{S_-}^\ast\phi\rangle_+&=&\int_{T(\A_{F}^{S\cup\iy})}f^s(t)\phi(t)\ipa{\sum_{x\in T(F_{S_-})/H}f_{S_-}(x)\delta_{S_-}\indi_{xH}}d^\times t\\
&=&\sum_{x\in T(F_{S_-})/H}f_{S_-}(x)\int_{T(\A_{F}^{S\cup\iy})}f^s(t)\phi(t)\ipa{x\delta_{S_-}\indi_{H}}d^\times t={\rm vol}(H)^{-1}\langle f^S\oti f_{S_-},\phi(\delta_{S_-}\indi_H) \rangle_+.
\end{eqnarray*}
where $\phi(\delta_{S_-}\indi_H)\in \mcA^{S\cup\iy}\ipa{I_{\xi}^r/I_{\xi}^{r+1}}$ is the image through the evaluation map \eqref{eqevaluationmap} corresponding to $\delta_{S_-}\indi_H$.
Thus, if we choose $H$ to be small enough, by the above computations 
\[
\mu_{\phi_\lambda^S}\;({\rm mod}\; I_\xi^{r+1})=\frac{\#\ipa{\mcO_+^{S_+}/\mcO_+}_{\rm tor}^{-1}}{{\rm vol(H)}}(\eta^{S_+}\cap\xi_\lambda)\cap\left(\delta_{S_+}^\ast\ipa{\phi_\lambda^{S}(\delta_{S_-}(\indi_H))}\cup\bigcup_{\mfp\in S_+}c_{\vartheta_\mfp}\right)=\frac{(-1)^s[\mcO_+^{S_+^2}:\mcO_+]^{-1}}{{\rm vol}(H)}\ell_{S_+}(q_{S_+^1}\oti Q_{\xi,H}^{S_+^2}),
\]
where $Q_{\xi,H}^{S_+^2}$ is the plectic invariant constructed as above by means of the evaluation $\phi_\lambda^{S}(\delta_{S_-}(\indi_H))$. In general, for any $v_{S_-}\in V_{S_-}$, we can consider the  plectic invariant $Q_{\xi,v_{S_-}}^{S_+^2}$ constructed by means of the evaluation $\phi_\lambda^{S}(v_{S_-})$.
Since $\phi^S_\lambda(tv_{S_-})=t\phi^S_\lambda(v_{S_-})$, for all $t\in T(F_{S_-})$, the morphism
\[
\psi:V_{S_-}\longrightarrow I_\xi^r/I_\xi^{r+1},\qquad v_{S_-}\longmapsto \ell_{S_+}\ipa{q_{S_+^1}\oti Q_{\xi,v_{S_-}}^{S_+^2})},
\]
satisfies $\psi(tv_{S_-})=\xi_{S_-}(t)\inv\cdot \psi(v_{S_-})$ for all $t\in T(F_{S_-})$, where $\xi_{S_-}=\rho^\ast\xi\mid_{T(F_{S_-})}$. By Saito-Tunnel (see \cite{Sa} and \cite{Tu}) the space 
$\Hom_{T(F_{S_-})}(V_{S_-}\oti\xi_{S_-},I_\xi^r/I_\xi^{r+1})$
is at most one dimensional. Similarly as in \eqref{localfactlambda}, we define
\begin{equation*}\label{deflambdamfq}
     \lambda_\mfp(f_1,f_2):=\int_{T(F_\mfp)}\xi_\mfp(t)\langle\pi_\mfp(t)f_1,\pi_\mfp(J_\mfp)f_2\rangle_\mfp d^\times t;
\end{equation*} 
and $\lambda_{S_-}:=\prod_{\mfq\in S_-}\lambda_\mfq\in \Hom_{T(F_{S_-})}(V_{S_-}\oti\xi_{S_-},\bar\Q)^{\oti 2}$.
By  Saito-Tunnel
\begin{eqnarray*}
    \psi(\delta_{S_-}(\indi_H))&=&\psi(x_{S_-})\cdot \lambda_{S_-}\ipa{\delta_{S_-}(\indi_H),x_{S_-}}\cdot \lambda_{S_-}(x_{S_-},x_{S_-})^{-1}\\
    \lambda_{S_-}(v_{S_-},\delta_{S_-}(\indi_H))&=&\lambda_{S_-}\ipa{x_{S_-},\delta_{S_-}(\indi_H)} \cdot \lambda^S\ipa{v_{S_-},x_{S_-}}\cdot\lambda_{S_-}(x_{S_-},x_{S_-})^{-1}.
\end{eqnarray*}
By the symmetry of $\lambda_\mfq$, we have $\lambda_{S_-}(\delta_{S_-}(\indi_H),x_{S_-})\lambda_{S_-}(x_{S_-},x_{S_-})^{-1}=\lambda_{S_-}(\delta_{S_-}(\indi_H),\delta_{S_-}(\indi_H))^{\frac{1}{2}}\lambda_{S_-}(x_{S_-},x_{S_-})^{-\frac{1
}{2}}$.
Hence, by \eqref{PhiPhiint} and \eqref{PhiPhiint2}:
\begin{equation*}
    \frac{\psi(\delta_{S_-}(\indi_H))}{{\rm vol}(H)\psi(x_{S_-})}=\frac{1}{{\rm vol}(H)}\frac{\lambda_{S_-}(\delta_{S_-}(\indi_H),\delta_{S_-}(\indi_H))^{\frac{1}{2}}}{\lambda_{S_-}(x_{S_-},x_{S_-})^{\frac{1
}{2}}}=\prod_{\mfq\in S_-}\frac{\langle \delta_\mfq(\indi_{H_\mfq}),\delta_\mfq(\indi_{H_\mfq})\rangle_\mfq^{\frac{1}{2}}}{\langle x_{\mfq},x_{\mfq}\rangle_\mfq^{\frac{1}{2}}}\frac{\lambda_\mfq(\delta_\mfq(\indi_{H_\mfq}),J_\mfq)^{\frac{1}{2}}}{{\rm vol}(H_\mfq)\cdot\lambda_\mfp(x_{\mfq},J_\mfq)^{\frac{1}{2}}}=\epsilon_{S_-}(\pi,\xi)^{\frac{1}{2}}.
\end{equation*}
The result then follows from the fact that $\psi(\delta_{S_-}(\indi_H))=\ell_{S_+}(q_{S_+^1}\oti Q_{\xi,H}^{S_+^2})$ and $\psi\ipa{x_{S_-}}=\ell_{S_+}(q_{S_+^1}\oti Q_{\xi}^{S_+^2})$.
\epf

\section{Appendix: Local integrals}\label{Appendix}

Let $T$ be the torus associated to the quadratic extension $K_\mfp/F_\mfp$, namely, for any $\mfo_{F_\mfp}$-algebra $R$, we have  $T(R)=(\mfo_{K_\mfp}\oti_{\mfo_{F_\mfp}}R)^\times/R^\times$. Let $\psi_{K_\mfp}$ be the quadratic character associated with $K_\mfp/F_\mfp$. Choose $\beta\in K_\mfp^\ti\setminus F_\mfp^\ti$ such that $\mfo_{K_\mfp}=\mfo_{F_\mfp}+\beta\mfo_{F_\mfp}$ and write 
\[
H_n:=\ipa{\mfo_{F_\mfp}+\beta\mfp^n}^\ti/\mfo_{F_\mfp}^\ti\subseteq T(F_\mfp).
\]
If $K_\mfp$ is non-split, write $\bar\beta$ for the conjugation of $\beta$. If $K_\mfp$ splits and $\beta=(b_1,b_2)$, we will eventually write $\beta=b_1$ and $\bar\beta=b_2$. 

\begin{lemma}\label{lemmalocint1}
Let $\xi:T(F_\mfp)\rightarrow\C^\times$ and $\chi:F_\mfp^\ti\rightarrow\C^\times$ be locally constant characters. Assume that $\chi\mid_{\mfo_{F_\mfp}^\ti}=1$ and let $n\geq{\rm cond}\xi+1$,  then we have
    \[
    \frac{\chi((\beta-\bar\beta)^{-2})}{{\rm vol}(H_0)}\int_{H_0 \setminus H_n}\chi\ipa{\frac{(t-1)^2}{t}}\xi(t)d^\ti t=
    \left\{\begin{array}{ll}
        \ipa{\frac{\chi(\varpi_\mfp)^{2n}q_\mfp^{-n}L(1,\psi_{K_\mfp})}{(\chi(\varpi_\mfp)^2q_\mfp^{-1}-1)\zeta_\mfp(1)}}+\ipa{1+\frac{(1-\chi(\varpi_\mfp)^{-2})L(1,\psi_{K_\mfp})}{\chi(\varpi_\mfp)^{-2}q_\mfp-1}},&n_\xi= 0;\\
        \ipa{\frac{\chi(\varpi_\mfp)^{2n}q_\mfp^{-n}L(1,\psi_{K_\mfp})}{(\chi(\varpi_\mfp)^2q_\mfp^{-1}-1)\zeta_\mfp(1)}}+\ipa{\frac{\chi(\varpi_\mfp)^{2n_\xi}q_\mfp^{-n_\xi}(1-\chi(\varpi_\mfp)^{-2})L(1,\psi_{K_\mfp})}{1-\chi(\varpi_\mfp)^2q_\mfp^{-1}}},&n_\xi> 0.
    \end{array}\right.
    \]
\end{lemma}
\begin{proof}
Notice that $H_0/H_n\simeq T(\mfo_{F_\mfp}/\mfp^n)$, hence  ${\rm vol}(H_n)=(\#T(\mfo_{F_\mfp}/\mfp^n))^{-1}{\rm vol}(H_0)=q_\mfp^{-n}L(1,\psi_{K_\mfp}){\rm vol}(H_0)$.
We have then by orthogonality
    \[
    I_0:=\int_{H_0}\xi(t)d^\ti t=\left\{\begin{array}{ll}
        {\rm vol}(H_0),&n_\xi= 0;\\
        0,&n_\xi> 0,
    \end{array}\right.,\qquad I_k:=\int_{H_k}\xi(t)d^\ti t=\left\{\begin{array}{ll}
        q_\mfp^{-k}L(1,\psi_{K_\mfp})\cdot {\rm vol}(H_0),&n_\xi\leq k;\\
        0,&n_\xi> k.
    \end{array}\right.
    \]
    For any class $t=[t_1+t_2\beta]\in T(F_\mfp)$, we have $\frac{(t-1)^2}{t}=(\beta-\bar\beta)^2\ipa{\frac{t_1}{t_2}+\beta}^{-1}\ipa{\frac{t_1}{t_2}+\bar\beta}^{-1}$, because we identify $t$ with $\frac{t_1+\beta t_2}{t_1+\bar\beta t_2}$. This implies that, for any $t\in H_k\setminus H_{k+1}$, we have $\frac{(t-1)^2}{t}\in (\beta-\bar\beta)^2\varpi_\mfp^{2k}\mfo_{F_\mfp}^\ti$. Thus, we compute
    \begin{eqnarray*}
    \int_{H_0 \setminus H_n}\chi\ipa{\frac{(t-1)^2}{t}}\xi(t)d^\ti t&=&\sum_{k=0}^{n-1}\int_{H_k \setminus H_{k+1}}\chi\ipa{\frac{(t-1)^2}{t}}\xi(t)d^\ti t=\chi((\beta-\bar\beta)^2)\ipa{\sum_{k=0}^{n-1}\chi(\varpi_\mfp)^{2k}\int_{H_k \setminus H_{k+1}}\xi(t)d^\ti t}\\
    &=&\chi((\beta-\bar\beta)^2)\ipa{\sum_{k=0}^{n-1}\chi(\varpi_\mfp)^{2k}\ipa{I_k-I_{k+1}}}\\
    &=&\chi((\beta-\bar\beta)^2)\ipa{I_0-\chi(\varpi_\mfp)^{2n-2}I_n+\sum_{k=1}^{n-1}\ipa{\chi(\varpi_\mfp)^{2k}-\chi(\varpi_\mfp)^{2k-2}}I_k}.
    \end{eqnarray*}
    and the result follows.
\end{proof}

\begin{lemma}\label{lemmalocint2}
Let $\xi:T(F_\mfp)\rightarrow\C^\times$ and $\chi:F_\mfp^\ti\rightarrow\C^\times$ be locally constant characters. Assume that $\chi\mid_{\mfo_{F_\mfp}^\ti}=1$ and, in the split case, $|\xi^{-1}\chi^{-1}(\varpi_\mfp)|,|\xi\chi^{-1}(\varpi_\mfp)|<1$.
    For $n$ big enough
    \[
    \frac{\chi((\beta-\bar\beta)^{-2})}{{\rm vol}(H_0)}\int_{T(F_\mfp)\setminus H_n}\chi\ipa{\frac{(t-1)^2}{t}}\xi(t)d^\times t=
    \left\{\begin{array}{ll}
        \ipa{\frac{\chi(\varpi_\mfp)^{2n}q_\mfp^{-n}L(1,\psi_{K_\mfp})}{(\chi(\varpi_\mfp)^2q_\mfp^{-1}-1)\zeta_\mfp(1)}}+\ipa{\frac{\varepsilon(\xi,\chi)\cdot L(1,\psi_{K_\mfp})}{1-\chi(\varpi_\mfp)^2q_\mfp^{-1}}},&n_\xi= 0;\\
        \ipa{\frac{\chi(\varpi_\mfp)^{2n}q_\mfp^{-n}L(1,\psi_{K_\mfp})}{(\chi(\varpi_\mfp)^2q_\mfp^{-1}-1)\zeta_\mfp(1)}}+\ipa{\frac{\chi(\varpi_\mfp)^{2n_\xi}q_\mfp^{-n_\xi}(1-\chi(\varpi_\mfp)^{-2})L(1,\psi_{K_\mfp})}{1-\chi(\varpi_\mfp)^2q_\mfp^{-1}}},&n_\xi> 0.
    \end{array}\right.
    \]
    where 
    \[
    \varepsilon(\xi,\chi):=\left\{\begin{array}{ll}
        \frac{\ipa{\chi(\varpi_\mfp)^2-1}\ipa{\chi(\varpi_\mfp)\xi^{-1}(\varpi_\mfp)q_\mfp^{-1}-1}\ipa{\chi(\varpi_\mfp)\xi(\varpi_\mfp)q_\mfp^{-1}-1}}{\ipa{\chi(\varpi_\mfp)\xi^{-1}(\varpi_\mfp)-1}\ipa{\chi(\varpi_\mfp)\xi(\varpi_\mfp)-1}},&T(F_\mfp)\mbox{ splits};\\
        (1-\chi(\varpi_\mfp)^2q_\mfp^{-2}),&T(F_\mfp)\mbox{ inert};\\
        (1-\xi(\varpi_\mfp^{\frac{1}{2}})\chi(\varpi_\mfp)q_\mfp^{-1})(1+\xi(\varpi_\mfp^{\frac{1}{2}})\chi(\varpi_\mfp)^{-1})&T(F_\mfp)\mbox{ ramified};
    \end{array}\right.
    \]
\end{lemma}
\begin{proof}
We have the following cases depending on the behavior of $T$:
    
    \emph{Assume that $T(F_\mfp)$ splits}: Notice that, using the change of variables $t\mapsto t^{-1}$, for any $n<0$
        \[
        \int_{\varpi_\mfp^k\mfo_{F_\mfp}^\ti}\chi\ipa{\frac{(t-1)^2}{t}}\xi(t)d^\times t=\int_{\varpi_\mfp^{-k}\mfo_{F_\mfp}^\ti}\chi\ipa{\frac{(t-1)^2}{t}}\xi(t)^{-1}d^\times t=\int_{\varpi_\mfp^{-k}\mfo_{F_\mfp}^\ti}\chi(t)^{-1}\xi(t)^{-1}d^\times t.
        \]
        Thus,  
        \begin{eqnarray*}
        \int_{T(F_\mfp)\setminus H_n}\chi\ipa{\frac{(t-1)^2}{t}}\xi(t)d^\times t&=&\sum_{k>0}\int_{\varpi_\mfp^k\mfo_{F_\mfp}^\ti}\chi^{-1}\xi(t) d^\times t+\sum_{k<0}\int_{\varpi_\mfp^{-k}\mfo_{F_\mfp}^\ti}\chi^{-1}\xi^{-1}(t) d^\times t+\int_{\mfo_{F_\mfp}^\ti\setminus H_n}\chi\ipa{1-t}^2\xi(t)d^\times t\\
        &=&\ipa{\sum_{k>0}\chi^{-1}\xi(\varpi_\mfp)^k+\sum_{k<0}\chi^{-1}\xi^{-1}(\varpi_\mfp)^{-k}}\int_{\mfo_{F_\mfp}^\ti}\xi(t) d^\times t+\int_{\mfo_{F_\mfp}^\ti\setminus H_n}\chi\ipa{1-t}^2\xi(t)d^\times t,\\
        \end{eqnarray*}
    and the result follows from lemma \ref{lemmalocint1}, since $\chi((\beta-\bar\beta)^2)=1$ in this setting.

    \emph{Assume that $T(F_\mfp)$ is inert}:  Notice that the image of $\beta\in\mfo_{K_\mfp}$ in $\mfo_{K_\mfp}/\mfp$ does not lie in $\mfo_{F_\mfp}/\mfp$. This implies that $(\beta-\bar\beta)\not\in\mfp$ and $\chi((\beta-\bar\beta)^{2})=1$. In this situation $T(F_\mfp)=H_0$, hence the result also follows from lemma \ref{lemmalocint1}.

     \emph{Assume that $T(F_\mfp)$ ramifies}: In this situation we can suppose that $\beta=\varpi_\mfp^{\frac{1}{2}}$ is a generator of the maximal ideal of $K_\mfp^\ti$. In this situation $T(F_\mfp)=H_0\cup\beta H_0$, hence we obtain
     \[
     \int_{T(F_\mfp)\setminus H_n}\chi\ipa{\frac{(t-1)^2}{t}}\xi(t)d^\times=\int_{\beta H_0}\chi\ipa{\frac{(t-1)^2}{t}}\xi(t)d^\times+\int_{H_0\setminus H_n}\chi\ipa{\frac{(t-1)^2}{t}}\xi(t)d^\times.
     \]
     Moreover, since any $t\in \beta H_0$ is of the form $t=[t_1+t_2\beta]$ where $t_1\in\mfp$ and $t_2\in\mfo_{F_\mfp}^\ti$, we have $\frac{(t-1)^2}{t}=\frac{(\beta-\bar\beta)^2}{\beta\bar\beta}\ipa{\frac{t_1}{\beta t_2}+1}^{-1}\ipa{\frac{t_1}{\bar\beta t_2}+1}^{-1}\in (\beta-\bar\beta)^2\varpi_\mfp^{-1}\mfo_{F_\mfp}^\ti$, hence
     \[
     \int_{\beta H_0}\chi\ipa{\frac{(t-1)^2}{t}}\xi(t)d^\times=\xi(\varpi_\mfp^{\frac{1}{2}})\cdot\chi((\beta-\bar\beta)^2)\cdot\chi(\varpi_\mfp)^{-1}\int_{H_0}\xi(t)d^\ti t,
     \]
     and the result again follows from lemma \ref{lemmalocint1} and the fact that $\xi(\varpi_\mfp^{\frac{1}{2}})=\pm1$. 
\end{proof}

\subsection{Local periods}\label{localint}

Write $\imath:T(F_\mfp)\hookrightarrow G(F_\mfp)$ for the fixed embedding.
Given the morphism $\delta_\mfp$ defined in \S \ref{deltap},
we aim to calculate the following integrals
\begin{equation}\label{pairingTH}
    \int_{T(F_\mfp)}\xi_\mfp(t)\langle \imath(t)\delta_\mfp(\indi_H),J\delta_\mfp(\indi_H)\rangle d^\ti t
\end{equation}
where $\xi_\mfp$ is a locally constant finite character, $H\subset T(\mfo_{F_\mfp})$ is an open and compact subgroup small enough so that $\xi_\mfp$ is $H$-invariant, $J\in G(F_\mfp)$ is so that $J\imath(t)=\imath(t^{-1})J$ for all $t\in T(F_\mfp)$, $d^\ti$ is a Haar measure of $T(F_\mfp)$, and $\langle\cdot,\cdot\rangle$ is the natural $G(F_\mfp)$-invariant $\C$-bilinear pairing on $V_\mfp$.

As in previous sections, we will assume that $V_\mfp$ is a unitary quotient of a principal series representation, thus, there exists a surjective $G(F_\mfp)$-invariant homomorphism $\Ind_P^G(\chi_\mfp)\rightarrow V_\mfp$, for some locally constant character $\chi_\mfp$. Thus,
if we write $J=b_J\cdot \imath(t_J)$ for some $t_J\in T(F_\mfp)$ and $b_J\in P$, then we obtain
\begin{equation}\label{jocambJ}
    J\delta_\mfp(f)(\imath(t))=\delta_\mfp(f)(\imath(t)\cdot J)=\delta_\mfp(f)(J\cdot\imath(t^{-1}))=\chi_\mfp\ipa{b_J}\cdot\delta_\mfp(f)(\imath(t_J\cdot t\inv))=\chi_\mfp\ipa{b_J}\cdot\delta_\mfp(f^*)(\imath(t)),
\end{equation}
where $f^*(t)=f(t\inv t_J^{-1})$ and $b_J^2\in F_\mfp^\ti$. Since $\indi_H^\ast=t_J^{-1}\indi_H$, we deduce that in case $\chi_\mfp$ unramified
\[
\int_{T(F_\mfp)}\xi_\mfp(t)\langle \imath(t)\delta_\mfp(\indi_H),J\delta_\mfp(\indi_H)\rangle d^\ti t=\xi_\mfp(t_J)^{-1}\int_{T(F_\mfp)}\xi_\mfp(t)\langle \imath(t)\delta_\mfp(\indi_H),\delta_\mfp(\indi_H)\rangle d^\ti t.
\]

We proceed to describe the bilinear pairing $\langle\cdot,\cdot\rangle$: On the one hand, for every $s\in \C$ we have the character
$\chi_{\mfp,s}=\chi_\mfp|\cdot|^s$ and the isomorphism
\[
\varphi_\mfp^s:\Ind_P^G(\chi_{\mfp})\stackrel{\simeq}{\longrightarrow}\Ind_P^G(\chi_{\mfp,s});\qquad \varphi_\mfp^s(f)(bk)=\chi_{\mfp,s}(b)f(k),\quad b\in P,\;k\in G(\mfo_{F_\mfp}).
\]
By \cite[Proposition 4.5.6]{Bump}, when ${\rm Re}(s)$ is big enough the $G(F_\mfp)$-equivariant integral 
\[
M_s:\Ind_P^G(\chi_{\mfp,s})\longrightarrow \Ind_P^G(|\cdot|\chi^{-1}_{\mfp,s}),\qquad M_sf(g)=\int_{F_\mfp}f(wn_xg)dx,\quad n_x=\bbm1&x\\&1\ebm ,\;w=\bbm&-1\\1&\ebm
\]
converges absolutely. Moreover by \cite[Proposition 4.5.7]{Bump}, for every $f\in \Ind_P^G(\chi_\mfp)$ the vector $M_s\varphi_\mfp^sf$ admits meromorphic continuation, having a possible pole at $s=a$ only if $\chi_\mfp^2=|\cdot|^{1-a}$. By abuse of notation, we write also $M_s\varphi_\mfp^sf$ for such meromorphic continuation.

On the other hand, by \cite[Proposition 4.5.5]{Bump} and \cite[Corollary 8.2]{blanco2015anticyclotomic} we have a $G(F_\mfp)$-invariant bilinear pairing 
\[
\Ind_P^G(\chi_{\mfp,s})\ti \Ind_P^G(|\cdot|\chi_{\mfp,s}^{-1}))\longrightarrow \C;\qquad ( f_1,f_2)\mapsto\int_{T(F_\mfp)}f_1(\imath(\tau))\cdot f_2(\imath(\tau))d^\ti \tau,
\]
admitting again meromorphic continuation.

With the above ingredients and assuming that $\chi_\mfp^2\neq|\cdot|$, we construct our $G(F_\mfp)$-invariant $\C$-bilinear pairing
\[
\langle f_1,f_2\rangle=(\varphi^s_{\mfp}f_1, M_s\varphi_\mfp^sf_2)_{s=0}.
\]
The corresponding pairing in $V_\mfp$ is derived from this one defined in $\Ind_P^G(\chi_\mfp)$. 
\begin{hypothesis}\label{restEF}
From now on we will restrict ourselves to the case $\chi_\mfp\mid_{\mcO_{F_\mfp}^\ti}$ trivial. This will provide a more explicit formula for our local integrals.
\end{hypothesis} 
\begin{proposition}\label{propoonM0}
If $H=H_n$ is small enough, we have that $\varphi_\mfp^s\delta_\mfp(\indi_H)(\imath( t))=\indi_H(t)$ and 
\[
M_s\varphi_\mfp^s\delta_\mfp(1_{H})(\imath(t))=
\frac{C_T^0\cdot {\rm vol}(H)}{\chi_{\mfp, s-1}(\tau_\mfp-\ovl\tau_\mfp)^2}\ipa{\chi_{\mfp, s-1}\ipa{\frac{(t-1)^2}{t}}\cdot\indi_{T(F_\mfp)\setminus H}(t)+C_\beta^s\cdot\ipa{1-\frac{1}{q_\mfp}}\cdot\frac{\alpha_\mfp^{-2n} q_\mfp^{-2n(s-1)}}{1-\alpha_\mfp^{-2} q_\mfp^{-2s+1}}\cdot\indi_{H}(t)},
\]
where  $C_T^0$ is a non-zero constant depending on $T(F_\mfp)$, $C_\beta^s=\chi_{\mfp, s-1}\ipa{\beta-\bar\beta}^2$ and $\alpha_\mfp=\chi_\mfp(\varpi_\mfp)\inv$.
\end{proposition}
\begin{proof}
Using the explicit description of $\delta_\mfp$ of Lemma \ref{lemma:mordelta} we deduce that, for $\bsm a&b\\c&d\esm\in\GL_2(\mfo_{F_\mfp})$,
\begin{eqnarray*}
\varphi_\mfp^s\delta_\mfp(\indi_H)\ipa{\bbm x_1&y\\&x_2\ebm\bbm a&b\\c&d\ebm}
&=&\chi_{\mfp,s}\ipa{\frac{x_1}{x_2}}\cdot\chi_{\mfp}\ipa{\frac{ad-bc}{(c\tau_\mfp+d)(c\ovl\tau_\mfp+d)}}\cdot \indi_H\ipa{\frac{c\bar\tau_\mfp+d}{c\tau_\mfp+d}}.
\end{eqnarray*}
Notice that, by the expression \eqref{eqiotat}, if 
\[
\imath(\tilde t)=\frac{\ovl\lambda_{\tilde t}}{\tau_\mfp-\ovl \tau_\mfp}\bbm t\tau_\mfp-\ovl \tau_\mfp & \tau_\mfp\ovl\tau_\mfp(1-t) \\ t-1  & \tau_\mfp-t\ovl\tau_\mfp \ebm=\bbm x_1&y\\&x_2\ebm\bbm a&b\\c&d\ebm,\qquad \bbm a&b\\c&d\ebm\in\GL_2(\mfo_{F_\mfp}),
\]
then $x_2c=\frac{\ovl\lambda_{\tilde t}(t-1)}{\tau_\mfp-\ovl \tau_\mfp}$ and $x_2d=\frac{\ovl\lambda_{\tilde t}(\tau_\mfp-t\ovl\tau_\mfp)}{\tau_\mfp-\ovl \tau_\mfp}$. Since either $c$ or $d$ lives in $\mfo_{F_\mfp}^\ti$ and $\lambda_{\tilde t}\ovl\lambda_{\tilde t}\in x_1x_2\mfo_{F_\mfp}^\ti$,
we have
\[
\varphi_\mfp^s\delta_\mfp(\indi_H)(\imath(\tilde t))=\left\{\begin{array}{ll}\chi_{\mfp,s}\left(\frac{t(\tau_\mfp-\ovl\tau_\mfp)^2}{(\tau_\mfp-t\ovl\tau_\mfp)^2}\right)\cdot\indi_H(t\inv),&v_\mfp\ipa{\frac{t-1}{\tau_\mfp-\ovl\tau_\mfp t}}\geq 0,\\
\chi_{\mfp,s}\left(\frac{t(\tau_\mfp-\ovl\tau_\mfp)^2}{(t-1)^2}\right)\cdot\indi_H(t\inv),&v_\mfp\ipa{\frac{t-1}{\tau_\mfp-\ovl\tau_\mfp t}}\leq 0.\end{array}\right.
\]
If $H$ is small enough open subgroup and $t\inv\in H$, then $t$ is close to 1 and $\ipa{\frac{t-1}{\tau_\mfp-\ovl\tau_\mfp t}}$ has big $\mfp$-adic valuation, moreover, 
$\left(\frac{t(\tau_\mfp-\ovl\tau_\mfp)^2}{(\tau_\mfp-t\ovl\tau_\mfp)^2}\right)\in\mfo_{F_\mfp}^\times$. We conclude that $\varphi_\mfp^s\delta_\mfp(\indi_H)(\imath( t))=\indi_H(t\inv)$.

On the other hand, by definition,
\[
M_s\varphi_\mfp^s\delta_\mfp(\indi_H)(g)=\int_{F_\mfp}\varphi_\mfp^s\delta_\mfp(\indi_H)(wn_xg)dx.
\]
By \cite[Corollary 8.4]{blanco2015anticyclotomic}, for any $h\in \Ind_P^G(|\cdot|)^0$, we have 
$C_T^0\int_{T(F_\mfp)}h(\imath(t))d^\ti t=\int_{F_\mfp}h(wn_x)dx$,
for some non-zero constant $C_T^0$. For any $t\in T(F_\mfp)$, let us consider $h_t\in \Ind_P^G(|\cdot|)^0$:
\[
h_t\ipa{\bbm x_1&y\\&x_2\ebm\imath(z)}=\left|\frac{x_1}{x_2}\right|\chi_{\mfp,s-1}\ipa{\frac{(z-1)^2}{z(\tau_\mfp-\ovl\tau_\mfp)^2}}\indi_{tH}(z\inv),
\]
extended by zero outside $P\cdot\imath(T(F_\mfp))\subseteq G(F_\mfp)$.
Thus, using the above identities and Lemma \ref{lemma:mordelta},
\begin{eqnarray*}
M_s\varphi_\mfp^s\delta_\mfp(\indi_H)(\imath(t))&=&\int_{F_\mfp}\varphi_\mfp^s\delta_\mfp(\indi_H)\ipa{\bbm&-1\\1&x\ebm\bbm\frac{x+\ovl\tau_\mfp}{x+\tau_\mfp}\tau_\mfp-\ovl\tau_\mfp&\tau_\mfp\ovl\tau_\mfp\ipa{1-\frac{x+\ovl\tau_\mfp}{x+\tau_\mfp}}\\\frac{x+\ovl\tau_\mfp}{x+\tau_\mfp}-1&\tau_\mfp-\frac{x+\ovl\tau_\mfp}{x+\tau_\mfp}\ovl\tau_\mfp\ebm\imath\ipa{\frac{x+\tau_\mfp}{x+\ovl\tau_\mfp}}\imath(t)}dx\\
&=&\int_{F_\mfp}\left|\frac{1}{(x+\tau_\mfp)(x+\ovl\tau_\mfp)}\right|\chi_{\mfp,s-1}\ipa{\frac{1}{(x+\tau_\mfp)(x+\ovl\tau_\mfp)}}\cdot \indi_H\ipa{t^{-1}\ipa{\frac{x+\ovl\tau_\mfp}{x+\tau_\mfp}}}dx=\int_{F_\mfp}h_t(wn_x)dx\\
&=&C_T^0\int_{T(F_\mfp)}h_t(\imath(h))d^\ti h=C_T^0\int_{T(F_\mfp)}\chi_{\mfp,s-1}\ipa{\frac{(h-1)^2}{h(\tau_\mfp-\ovl\tau_\mfp)^2}}\indi_{tH}(h\inv)d^\ti h,
\end{eqnarray*}
and the computation of $M_s\varphi_\mfp^s\delta_\mfp(\indi_H)(\imath(t))$ amounts to finding the integral $I_{H}(t):=\int_H\chi_{\mfp, s-1}\ipa{\frac{(h t\inv-1)^2}{ht\inv}}d^\ti h$. 

On the one hand, if $t\not\in H_n$, $1+\frac{(t^2-h)(h-1)}{(t-1)^2h}=\frac{(ht^{-1}-1)^2}{ht^{-1}}\frac{t}{(t-1)^2}\in 1+\mfp^n$, hence,
$I_{H_n}(t)=\chi_{\mfp, s-1}\ipa{\frac{(t-1)^2}{t}}{\rm vol}(H_n)$.

On the other hand, if $t\in H_n$ we use the bijection
$\mfo_{F_\mfp}/\mfp\stackrel{\simeq}{\rightarrow} H_n/H_{n+1}$, $a\mapsto h_a:=[1+\beta a\varpi_\mfp^n]$ to obtain
\[
I_{H_n}(t)=\int_{H_n}\chi_{\mfp, s-1}\ipa{\frac{(h -1)^2}{h}}d^\ti h=\int_{H_{n+1}}\chi_{\mfp, s-1}\ipa{\frac{(h -1)^2}{h}}d^\ti h+\sum_{a\in (\mfo_{F_\mfp}/\mfp)^\ti}\int_{H_{n+1}}\chi_{\mfp, s-1}\ipa{\frac{(h_ah -1)^2}{h_ah}}d^\ti h.
\]
Recall that in the above expression we have identified $h_a$ with $\lambda_{\tilde h_a}/\bar\lambda_{\tilde h_a}=(1+\beta a\varpi_\mfp^n)/(1+\bar\beta a\varpi_\mfp^n)$. 
Thus, assuming that $n$ is big enough, we have
\begin{eqnarray*}
\sum_{a\in (\mfo_{F_\mfp}/\mfp)^\ti}\int_{H_{n+1}}\chi_{\mfp, s-1}\ipa{\frac{(h_ah -1)^2}{h_ah}}d^\ti h&=&
C_\beta^s\chi_{\mfp, s-1}(\varpi_\mfp)^{2n}\sum_{a\in (\mfo_{F_\mfp}/\mfp)^\ti}\int_{H_{n+m}}\chi_{\mfp, s-1}\ipa{\frac{\ipa{\frac{h-1}{(\beta-\bar\beta)\varpi_\mfp^n}+a\frac{(\beta h-\bar \beta)}{(\beta-\bar\beta)}}^2}{(1+\beta a\varpi_\mfp^n)(1+\bar\beta a\varpi_\mfp^n)h}}d^\ti h\\
&=&C_\beta^s\cdot\alpha_\mfp^{-2n} q_\mfp^{-2n(s-1)}\cdot{\rm vol}(H_{n}),
\end{eqnarray*}
since $q_\mfp{\rm vol}(H_{n+1})={\rm vol}(H_{n})$.
Applying a simple induction we conclude that
\[
I_{H_n}(t)=\int_{H_n}\chi_{\mfp, s-1}\ipa{\frac{(h -1)^2}{h}}d^\ti h=
   C_\beta^s\cdot\ipa{1-\frac{1}{q_\mfp}}\cdot {\rm vol}(H_n)\ipa{\sum_{k\geq 0}\alpha_\mfp^{-2n-2k} q_\mfp^{-2(n+k)(s-1)-k}},  
\]
if $t\in H_n$, and the result follows.
\end{proof}

Notice that, if $\chi_\mfp^2\neq|\cdot|$, then we have
\begin{eqnarray*}
\int_{T(F_\mfp)}\xi_\mfp(t)\langle t\delta_\mfp(\indi_H),\delta_\mfp(\indi_H)\rangle d^\ti t&=&\int_{T(F_\mfp)}\int_{T(F_\mfp)}\xi_\mfp(t)\varphi^s_{\mfp}\delta_\mfp(\indi_H)(\imath(\tau t))\cdot M_s\varphi^s_{\mfp}\delta_\mfp(\indi_H)(\imath(\tau))d^\ti td^\ti\tau\mid_{s=0}\\
&=&{\rm vol}(H)\int_{T(F_\mfp)}\xi_\mfp(\tau\inv)\cdot M_s\varphi^s_{\mfp}\delta_\mfp(\indi_H)(\imath(\tau))d^\ti\tau\mid_{s=0}.
\end{eqnarray*}
Hence, we can apply the previous result to obtain:
\begin{proposition}\label{EulerFapp}
   If $\chi_\mfp^2\neq |\cdot|$ then, for $H=H_n\in T(F_\mfp)$ small enough, we have
\[
\int_{T(F_\mfp)}\xi_\mfp(t)\langle \imath(t)\delta_\mfp(\indi_H),J\delta_\mfp(\indi_H)\rangle d^\ti t=\frac{C^0_T\cdot C_\beta^0\cdot{\rm vol}(H_0) \cdot{\rm vol}(H)^2\cdot L(1,\psi_{K_\mfp})\cdot L(-1,\chi_\mfp^2)}{\xi_\mfp(t_J)\cdot L(2,\chi_\mfp^{-2})\cdot \chi_{\mfp, -1}(\tau_\mfp-\ovl\tau_\mfp)^{2}}\cdot\left\{\begin{array}{ll}
        \varepsilon(\pi_\mfp,\xi_\mfp),&n_{\xi_\mfp}= 0;\\
        \alpha_\mfp^{-n_\xi}\beta_\mfp^{n_\xi},&n_{\xi_\mfp}> 0.
    \end{array}\right.
\] 
 where $\alpha_\mfp=\chi_\mfp(\varpi_\mfp)^{-1}$, $\beta_\mfp=\chi_\mfp(\varpi_\mfp)q_\mfp$ and
    \[
    \varepsilon(\pi_\mfp,\xi_\mfp):=\frac{L(1,\xi_\mfp\cdot\chi^{-1}_\mfp\circ{\rm N}_{K_\mfp/F_\mfp})}{L(0,\xi_\mfp\cdot\chi_\mfp\circ{\rm N}_{K_\mfp/F_\mfp})}=\left\{\begin{array}{ll}
        \frac{\ipa{1-\xi_\mfp^{-1}(\varpi_\mfp)\alpha_\mfp^{-1}}\ipa{1-\xi_\mfp(\varpi_\mfp)\alpha_\mfp^{-1}}}{\ipa{1-\xi_\mfp^{-1}(\varpi_\mfp)\beta_\mfp^{-1}}\ipa{1-\xi_\mfp(\varpi_\mfp)\beta_\mfp^{-1}}},&T(F_\mfp)\mbox{ splits};\\
        \frac{1-\alpha_\mfp^{-2}}{1-\beta_\mfp^{-2}},&T(F_\mfp)\mbox{ inert};\\
        \frac{1-\xi_\mfp(\varpi_\mfp^{\frac{1}{2}})\alpha_\mfp^{-1}}{1-\xi_\mfp(\varpi_\mfp^{\frac{1}{2}})\beta_\mfp^{-1}},&T(F_\mfp)\mbox{ ramified}.
    \end{array}\right.
    \]
\end{proposition}
\begin{proof}
Combining proposition \ref{propoonM0} 
with lemma \ref{lemmalocint2} we obtain that, if ${\rm Re}(s)$ is small enough, 
\[
\frac{\int_{T(F_\mfp)}\xi_\mfp^{-1}(\tau)\cdot M_s\varphi^s_{\mfp}\delta_\mfp(\indi_{H_n})(\imath(\tau))d^\ti\tau}{C^0_T\cdot {\rm vol}(H_n)\cdot\chi_{\mfp, s-1}(\tau_\mfp-\ovl\tau_\mfp)^{-2}}=C_\beta^s\cdot{\rm vol}(H_0)\cdot\left\{\begin{array}{ll}
        \frac{L(1-s,\xi_\mfp\cdot\chi^{-1}_\mfp\circ{\rm N}_{K_\mfp/F_\mfp})(1-\alpha_\mfp^{2}q_\mfp^{2s-2}) L(1,\psi_{K_\mfp})}{L(s,\xi_\mfp\cdot\chi_\mfp\circ{\rm N}_{K_\mfp/F_\mfp})(1-\alpha_\mfp^{-2}q_\mfp^{1-2s})},&n_{\xi_\mfp}= 0;\\
        \frac{\alpha_\mfp^{-2n_\xi}q_\mfp^{n_\xi(1-2s)}(1-\alpha_\mfp^{2}q_\mfp^{2s-2})L(1,\psi_{K_\mfp})}{1-\alpha_\mfp^{-2}q^{1-2s}_\mfp},&n_{\xi_\mfp}> 0.
    \end{array}\right.
    \]
        and the result then follows from the computations performed above.
\end{proof}

\subsection{Norms of $p$-stabilized test vectors}\label{normtestvec}

Write $U_0(1)=\PGL_2(\mfo_{F_\mfp})$ and $U_0(\mfp)\subset U_0(1)$ the usual subgroup of upper triangular matrices modulo $\mfp$.
As above, let $\chi_\mfp$ be an unramified character. We can consider the $U_0(\mfp)$-invariant vector
\begin{equation}\label{deff0}
    f_0\bsm a&b\\c&d\esm=\chi_\mfp\ipa{\frac{ad-bc}{c^2}}\cdot\indi_{\mfo_{F_\mfp}}\ipa{\frac{d}{c}},\qquad\mbox{or}\qquad f_0(bk)=\chi_{\mfp}(b)\indi_{U_0(1)\setminus U_0(\mfp)}(k);\quad b\in P;\quad k\in U_0(1).
\end{equation}
Notice that $U_\mfp f_0=\alpha_\mfp f_0$, where $U_\mfp$ is the usual Hecke operator.
We aim to compute $\langle f_0,f_0\rangle$. 
By definition, $\varphi_\mfp^s(f_0)\bsm a&b\\c&d\esm=\chi_{\mfp,s}\ipa{\frac{ad-bc}{c^2}}\cdot\indi_{\mfo_{F_\mfp}}\ipa{\frac{d}{c}}$ and if we consider $k=\bsm a&b\\c&d\esm\in U_0(1)$ then we have
\begin{eqnarray*}    
M_s\varphi_\mfp^sf_0(k)&=&\int_{F_\mfp}\varphi_\mfp^sf_0(wn_xk)dx=\int_{F_\mfp}\varphi_\mfp^sf_0\ipa{\bbm-c&-d\\a+cx&b+dx\ebm k} dx=\int_{F_\mfp}\chi_{\mfp,s}(a+cx)^{-2}\indi_{\mfo_{F_\mfp}}\ipa{\frac{b+dx}{a+cx}} dx\\
&=&\int_{\mfo_{F_\mfp}}\indi_{\mfo_{F_\mfp}}\ipa{\frac{b+dx}{a+cx}} dx+\int_{F_\mfp\setminus \mfo_{F_\mfp}}\chi_{\mfp,s}(a+cx)^{-2}\indi_{\mfo_{F_\mfp}}\ipa{\frac{b+dx}{a+cx}} dx\\
&=&{\rm vol}(\mfo_{F_\mfp})\indi_{U_0(\mfp)}(k)+{\rm vol}\ipa{\mfo_{F_\mfp}\setminus\ipa{-\frac{a}{c}+\mfp}}\indi_{U_0(1)\setminus U_0(\mfp)}(k)+\int_{F_\mfp\setminus\mfo_{F_\mfp}}\chi_{\mfp, s}(x)^{-2}\indi_{U_0(1)\setminus U_0(\mfp)}(k) dx\\
&=&{\rm vol}(\mfo_{F_\mfp})\indi_{U_0(\mfp)}(k)+\frac{q_\mfp-1}{q_\mfp}{\rm vol}(\mfo_{F_\mfp})\indi_{U_0(1)\setminus U_0(\mfp)}(k)+{\rm vol}(\mfo_{F_\mfp})\indi_{U_0(1)\setminus U_0(\mfp)}(k)\left(\frac{\chi_\mfp(\varpi_\mfp)^2q_\mfp^{1-2s}}{1-\chi_\mfp(\varpi_\mfp)^2q_\mfp^{1-2s}}\frac{q_\mfp-1}{q_\mfp}\right)\\
&=&{\rm vol}(\mfo_{F_\mfp})\ipa{\indi_{U_0(\mfp)}(k)+\indi_{U_0(1)\setminus U_0(\mfp)}(k)\ipa{\frac{1-q_\mfp^{-1}}{1-\chi_\mfp(\varpi_\mfp)^2q_\mfp^{1-2s}}}}.
\end{eqnarray*}
Thus, $h_1=f_0\cdot M_0f_0\in \Ind_P^G(|\cdot|)^0$ satisfies $h_1(bk)=\frac{{\rm vol}(\mfo_{F_\mfp})L(-1,\chi_\mfp^2)}{\zeta_\mfp(1)}|b|\indi_{U_0(1)\setminus U_0(\mfp)}(k)$. We conclude that,
\[
C_T^0\langle f_0,f_0\rangle=C_T^0\int_{T(F_\mfp)}h_1(\imath(t))d^\ti t=\int_{F_\mfp}h_1(wn_x)dx=\int_{F_\mfp}h_1\bbm&-1\\1&x\ebm dx=\frac{L(-1,\chi_\mfp^2)}{\zeta_\mfp(1)}{\rm vol}(\mfo_{F_\mfp})^2.
\]
\begin{remark}\label{remarkonconstants}
    The value of ${\rm vol}(\mfo_{F_\mfp})$ can be computed in terms of the constant $C_T^0$. Indeed, by definition,   
$C_T^0\int_{T(F_\mfp)}h(\imath(t))d^\ti t=\int_{F_\mfp}h(wn_x)dx$ for any $h\in \Ind_P^G(|\cdot|)^0$. Hence, if we choose $h=\delta_\mfp(\indi_H)$ for some $H$ small enough, we obtain by lemma \ref{lemma:mordelta}
\[
C_T^0{\rm vol}(H)=C_T^0\int_{T(F_\mfp)}\delta_\mfp(\indi_H)(\imath(t))d^\ti t=\int_{F_\mfp}\delta_\mfp(\indi_H)\bbm&-1\\1&x\ebm dx=\int_{F_\mfp}|(\tau_\mfp+x)(\ovl\tau_\mfp+x)|^{-1}\cdot \indi_H\ipa{\frac{\bar\tau_\mfp+x}{\tau_\mfp+x}}dx.
\]
Assume that $H=H_n=(\mfo_{F_\mfp}+\beta\mfp^n)^\ti/\mfo_{F_\mfp}^\ti$. Since 
\[
\bar\tau_\mfp+x=x+\frac{\tau_\mfp+\bar\tau_\mfp}{2}+\frac{\bar\tau_\mfp-\tau_\mfp}{2}=\ipa{x+\frac{\tau_\mfp+\bar\tau_\mfp}{2}}\ipa{1+\ipa{x+\frac{\tau_\mfp+\bar\tau_\mfp}{2}}^{-1}\ipa{\frac{\bar\tau_\mfp-\tau_\mfp}{2}}},
\]
we deduce that $\frac{\bar\tau_\mfp+x}{\tau_\mfp+x}\in H$ if and only if $\ipa{x+\frac{\tau_\mfp+\bar\tau_\mfp}{2}}^{-1}\ipa{\frac{\bar\tau_\mfp-\tau_\mfp}{\bar\beta-\beta}}\in \mfp^n$. If we denote the $\mfp$-adic valuation of $\frac{\bar\tau_\mfp-\tau_\mfp}{\bar\beta-\beta}$ as $n_0$, we obtain
\begin{eqnarray*}
\frac{C_T^0L(1,\psi_{K_\mfp}){\rm vol}(H_0)}{q_\mfp^{n}}&=&C_T^0{\rm vol}(H_n)=\int_{\ipa{x+\frac{\tau_\mfp+\bar\tau_\mfp}{2}}^{-1}\in \mfp^{n-n_0}}|(\tau_\mfp+x)(\ovl\tau_\mfp+x)|^{-1}dx=\int_{y^{-1}\in \mfp^{n-n_0}}\left|y^2-\frac{(\tau-\bar\tau)^2}{4}\right|^{-1}dy\\
&=&\int_{y^{-1}\in \mfp^{n-n_0}}|y|^{-2}dy =\sum_{k\geq n-n_0}q_\mfp^{-k}\int_{\varpi_\mfp^{-k}\mfo_{F_\mfp}^\ti}\frac{dy}{|y|}=\frac{q_\mfp^{n_0-n}}{1-q_\mfp^{-1}}\int_{\mfo_{F_\mfp}^\ti}\frac{dy}{|y|}=q_\mfp^{n_0-n}{\rm vol}(\mfo_{F_\mfp})
\end{eqnarray*}
Thus, we obtain
\begin{equation}\label{volOvsCT}
    C_T^0\left|\frac{\bar\tau_\mfp-\tau_\mfp}{\bar\beta-\beta}\right|L(1,\psi_{K_\mfp}){\rm vol}(H_0)={\rm vol}(\mfo_{F_\mfp}).
\end{equation}
\end{remark}

\begin{corollary}\label{coroLIcomp}
 For $H\subseteq T(F_\mfp)$ small enough, we have
\[
\int_{T(F_\mfp)}\xi_\mfp(t)\frac{\langle \imath(t)\delta_\mfp(\indi_H),J\delta_\mfp(\indi_H)\rangle}{\langle f_0,f_0\rangle} d^\ti t=\chi_{\mfp}\left(\frac{\beta-\ovl\beta}{\tau_\mfp-\ovl\tau_\mfp}\right)^{2}\frac{\zeta_\mfp(1)}{L(1,\psi_{K_\mfp}) L(2,\chi_\mfp^{-2})}\frac{{\rm vol}(H)^2}{\xi_\mfp(t_J){\rm vol}(H_0)}\cdot\left\{\begin{array}{ll}
        \varepsilon(\pi_\mfp,\xi_\mfp),&n_{\xi_\mfp}= 0;\\
        \alpha_\mfp^{-n_\xi}\beta_\mfp^{n_\xi},&n_{\xi_\mfp}> 0.
    \end{array}\right.
\] 
\end{corollary}
\begin{proof}
    If $\chi_\mfp^2=|\cdot|$, then 
    \[
        \int_{T(F_\mfp)}\xi_\mfp(t)\langle \imath(t)\delta_\mfp(\indi_H),\delta_\mfp(\indi_H)\rangle d^\ti t=\int_{T(F_\mfp)}\int_{T(F_\mfp)}\xi_\mfp(t)\imath(t)\delta_\mfp(\indi_H)(\imath(z))\delta_\mfp(\indi_H)(\imath(z))d^\ti td^\ti z={\rm vol}(H)^2.
    \]
    Similarly, since $f_0^2(bk)=|b|\indi_{K\setminus K_0(\mfp)}(k)$, we have
\[
C_T^0\langle f_0,f_0\rangle=C_T^0\int_{T(F_\mfp)}f_0(\imath(t))^2d^\ti t=\int_{F_\mfp}f_0^2(wn_x)dx=\int_{F_\mfp}f_0^2\bbm&-1\\1&x\ebm dx={\rm vol}(\mfo_{F_\mfp}).
\]
Thus, the result follows from remark \ref{remarkonconstants} in this case. For $\chi_\mfp^2\neq|\cdot|$ the result follows from proposition \ref{EulerFapp}, remark \ref{remarkonconstants} and the above computations.
\end{proof}
\begin{remark}
    Notice that by definition $(\beta-\bar\beta)^2\mfo_{F_\mfp}$ is the relative discriminant of $K_\mfp/F_\mfp$ when it is a field. Otherwise $(\beta-\bar\beta)^2\mfo_{F_\mfp}=\mfo_{F_\mfp}$.
\end{remark}

\subsection{Distinguished periods in the Steinberg case}\label{disting}

  Throughout this section we will assume that $T(F_\mfp)$ splits, $\chi_\mfp=1$ and therefore $V_\mfp=\tno{St}_\C(F_\mfp)$ is Steinberg.
Notice that we have a distinguished element $\hat f_0\in \tno{St}_\C(F_\mfp)$;
\[
\hat f_0\bsm a&b\\c&d\esm=\indi_{\mfo_{F_\mfp}}\ipa{\frac{c\ovl\tau_\mfp+d}{c\tau_\mfp+d}}.
\]
By lemma \ref{lemma:mordelta} we have that $\hat f_0(\imath(t))=\indi_{\mfo_{F_\mfp}}(t^{-1})$, 
moreover, $\hat f_0=\bsm \tau_\mfp&\ovl\tau_\mfp\\ 1&1\esm f_0$.
In this appendix we also aim to compute 
\[
I=\int_{T(F_\mfp)}\langle\pi_\mfp(t)\hat f_0,\pi_\mfp(J_\mfp)\hat f_0\rangle_\mfp d^\times t
\]

Since $\imath(t)\bsm \tau_\mfp&\ovl\tau_\mfp\\ 1&1\esm=\bsm \tau_\mfp&\ovl\tau_\mfp\\ 1&1\esm\bsm t&\\ &1\esm$ and $J\bsm \tau_\mfp&\ovl\tau_\mfp\\ 1&1\esm=\bsm \tau_\mfp&\ovl\tau_\mfp\\ 1&1\esm\bsm &1\\ -t_J&\esm$, we have by the above computation of $M_s\varphi_\mfp^sf_0(k)$
\begin{eqnarray*}
  I&=&  \int_{F_\mfp^\times}\langle\pi_\mfp\bsm t&\\ &1\esm f_0,\pi_\mfp\bsm &1\\ -t_J&\esm f_0\rangle_\mfp d^\times t=\int_{F_\mfp^\times}\int_{T(F_\mfp)}\pi_\mfp\ipa{\bsm &1\\ -t_J&\esm^{-1}\bsm t&\\ &1\esm} f_0(\imath(z)) M_0f_0(\imath(z))d^\times z d^\times t\\
&=&\frac{1}{C_T^0}\int_{F_\mfp^\times}\int_{F_\mfp}\pi_\mfp\bsm &1\\ -t&\esm f_0\bsm &-1\\1&x\esm M_0f_0\bsm &-1\\1&x\esm d x d^\times t\\
&=&\frac{1}{C_T^0}\ipa{\int_{F_\mfp^\times}\int_{\mfo_{F_\mfp}}f_0\bsm t&\\-tx&1\esm M_0f_0\bsm &-1\\1&x\esm d x d^\times t+\int_{F_\mfp^\times}\int_{F_\mfp\setminus\mfo_{F_\mfp}}|x|^{-2}f_0\bsm t&\\-tx&1\esm M_0f_0\bsm 1&\\x^{-1}&1\esm d x d^\times t}\\
&=&\frac{{\rm vol}(\mfo_{F_\mfp})}{C_T^0}\ipa{\int_{F_\mfp^\times}-q_\mfp^{-1}\int_{\mfo_{F_\mfp}}\indi_{\mfo_{F_\mfp}}\ipa{t^{-1}x^{-1}} d x d^\times t+\int_{F_\mfp^\times}\int_{F_\mfp\setminus\mfo_{F_\mfp}}|x|^{-2}\indi_{\mfo_{F_\mfp}}\ipa{t^{-1}x^{-1}} d x d^\times t}.
\end{eqnarray*}
Notice that ${\rm vol}\ipa{\mfo_{F_\mfp}\cap t^{-1}(F_\mfp\setminus\mfp)}
={\rm vol}(\mfo_{F_\mfp})(1-q_\mfp^{-1}|t|^{-1})\indi_{\mfo_{F_\mfp}}(t^{-1})$. Moreover,
\begin{eqnarray*}
    \int_{F_\mfp\setminus\mfo_{F_\mfp}}|x|^{-2}\indi_{\mfo_{F_\mfp}}\ipa{t^{-1}x^{-1}} d x&=&\indi_{\mfp^{-1}}(t^{-1})\sum_{k\geq 1}q_\mfp^{-k}\int_{\mfo_{F_\mfp}^{\ti}}\frac{d x}{|x|}+\indi_{F_\mfp\setminus\mfp^{-1}}(t^{-1})\sum_{|t|\geq q_\mfp^{-k}}q_\mfp^{-k}\int_{\mfo_{F_\mfp}^{\ti}}\frac{d x}{|x|}\\
    &=&{\rm vol}(\mfo_{F_\mfp})\ipa{q_\mfp^{-1}\indi_{\mfp^{-1}}(t^{-1})+|t|\indi_{\mfp^{2}}(t)}.
\end{eqnarray*}
Hence, we obtain
\begin{eqnarray*}
  I
  &=&\frac{{\rm vol}(\mfo_{F_\mfp})^2}{C_T^0}\ipa{\int_{F_\mfp^\times}q_\mfp^{-2}|t|^{-1}\indi_{\mfo_{F_\mfp}}(t^{-1}) d^\times t+\int_{F_\mfp^\times}q_\mfp^{-1}\ipa{\indi_{\mfp^{-1}}(t^{-1})-\indi_{\mfo_{F_\mfp}}(t^{-1})}d^\ti t+\int_{F_\mfp^\times}|t|\indi_{\mfp^{2}}(t) d^\times t}\\
  &=&\frac{{\rm vol}(\mfo_{F_\mfp})^2{\rm vol}(H_0)}{C_T^0}\ipa{q_\mfp^{-2}\sum_{k\geq 0} q_\mfp^{-k}+q_\mfp^{-1}+\sum_{k\geq 2} q_\mfp^{-k}}=\frac{{\rm vol}(\mfo_{F_\mfp})^2{\rm vol}(H_0)}{q_\mfp C_T^0}\ipa{\frac{q_\mfp^{-1}+1}{1-q_\mfp^{-1}}}=\frac{C_T^0{\rm vol}(H_0)^3(q_\mfp^{-1}+1)}{q_\mfp\left|\bar\tau_\mfp-\tau_\mfp\right|^{-2}(1-q_\mfp^{-1})^3}.
\end{eqnarray*}

\printbibliography

\end{document}